\documentclass[11pt,leqno]{amsart} 
\usepackage{amssymb} \usepackage{xypic} 
\begin{document} \theoremstyle{plain} 
\newtheorem{thm}{Theorem}[section] \newtheorem*{thm1}{Theorem 1} 
\newtheorem*{thm2}{Theorem 2} \newtheorem{lemma}[thm]{Lemma} 
\newtheorem{lem}[thm]{Lemma} \newtheorem{cor}[thm]{Corollary} 
\newtheorem{prop}[thm]{Proposition} \newtheorem{propose}[thm]{Proposition} 
\newtheorem{variant}[thm]{Variant} \theoremstyle{definition} 
\newtheorem{notations}[thm]{Notations} \newtheorem{rem}[thm]{Remark} 
\newtheorem{rmk}[thm]{Remark} \newtheorem{rmks}[thm]{Remarks} 
\newtheorem{defn}[thm]{Definition} \newtheorem{ex}[thm]{Example} 
\newtheorem{claim}[thm]{Claim} \newtheorem{ass}[thm]{Assumption} 
\numberwithin{equation}{section} \newcounter{elno} 
\def\points{\list {\hss\llap{\upshape{(\roman{elno})}}}{\usecounter{elno}}} 
\let\endpoints=\endlist


\catcode`\@=11 
\def\opn#1#2{\def#1{\mathop{\kern0pt\fam0#2}\nolimits}} 
\def\bold#1{{\bf #1}}
\def\underrightarrow@#1#2{\vtop{\ialign{$##$\cr
 \hfil#1#2\hfil\cr\noalign{\nointerlineskip}%
 #1{-}\mkern-6mu\cleaders\hbox{$#1\mkern-2mu{-}\mkern-2mu$}\hfill
 \mkern-6mu{\to}\cr}}} \let\underarrow\underrightarrow 
\def\underleftarrow{\mathpalette\underleftarrow@} 
\def\underleftarrow@#1#2{\vtop{\ialign{$##$\cr
 \hfil#1#2\hfil\cr\noalign{\nointerlineskip}#1{\leftarrow}\mkern-6mu
 \cleaders\hbox{$#1\mkern-2mu{-}\mkern-2mu$}\hfill
 \mkern-6mu{-}\cr}}}
%
%

%
\def\:{\colon}
\let\oldtilde=\tilde
\def\tilde#1{\mathchoice{\widetilde{#1}}{\widetilde{#1}}%
{\indextil{#1}}{\oldtilde{#1}}}
\def\indextil#1{\lower2pt\hbox{$\textstyle{\oldtilde{\raise2pt%
\hbox{$\scriptstyle{#1}$}}}$}}
\def\pnt{{\raise1.1pt\hbox{$\textstyle.$}}}
%

%
\let\amp@rs@nd@\relax
\newdimen\ex@\ex@.2326ex
\newdimen\bigaw@
\newdimen\minaw@
\minaw@16.08739\ex@
\newdimen\minCDaw@
\minCDaw@2.5pc
\newif\ifCD@
\def\minCDarrowwidth#1{\minCDaw@#1}
\newenvironment{CD}{\@CD}{\@endCD}
\def\@CD{\def\A##1A##2A{\llap{$\vcenter{\hbox
 {$\scriptstyle##1$}}$}\Big\uparrow\rlap{$\vcenter{\hbox{%
$\scriptstyle##2$}}$}&&}%
\def\V##1V##2V{\llap{$\vcenter{\hbox
 {$\scriptstyle##1$}}$}\Big\downarrow\rlap{$\vcenter{\hbox{%
$\scriptstyle##2$}}$}&&}%
\def\={&\hskip.5em\mathrel
 {\vbox{\hrule width\minCDaw@\vskip3\ex@\hrule width
 \minCDaw@}}\hskip.5em&}%
\def\verteq{\Big\Vert&&}%
\def\noarr{&&}%
\def\vspace##1{\noalign{\vskip##1\relax}}\relax\let\amp@rs@nd@&\iffalse}\fi
 \CD@true\vcenter\bgroup\relax\let\\=\cr\iffalse}\fi\tabskip\z@skip\baselineskip20\ex@
 \lineskip3\ex@\lineskiplimit3\ex@\halign\bgroup
 &\hfill$\m@th##$\hfill\cr}
\def\@endCD{\cr\egroup\egroup}
%
\def\>#1>#2>{\amp@rs@nd@\setbox\z@\hbox{$\scriptstyle
 \;{#1}\;\;$}\setbox\@ne\hbox{$\scriptstyle\;{#2}\;\;$}\setbox\tw@
 \hbox{$#2$}\ifCD@
 \global\bigaw@\minCDaw@\else\global\bigaw@\minaw@\fi
 \ifdim\wd\z@>\bigaw@\global\bigaw@\wd\z@\fi
 \ifdim\wd\@ne>\bigaw@\global\bigaw@\wd\@ne\fi
 \ifCD@\hskip.5em\fi
 \ifdim\wd\tw@>\z@
 \mathrel{\mathop{\hbox to\bigaw@{\rightarrowfill}}\limits^{#1}_{#2}}\else
 \mathrel{\mathop{\hbox to\bigaw@{\rightarrowfill}}\limits^{#1}}\fi
 \ifCD@\hskip.5em\fi\amp@rs@nd@}
\def\<#1<#2<{\amp@rs@nd@\setbox\z@\hbox{$\scriptstyle
 \;\;{#1}\;$}\setbox\@ne\hbox{$\scriptstyle\;\;{#2}\;$}\setbox\tw@
 \hbox{$#2$}\ifCD@
 \global\bigaw@\minCDaw@\else\global\bigaw@\minaw@\fi
 \ifdim\wd\z@>\bigaw@\global\bigaw@\wd\z@\fi
 \ifdim\wd\@ne>\bigaw@\global\bigaw@\wd\@ne\fi
 \ifCD@\hskip.5em\fi
 \ifdim\wd\tw@>\z@
 \mathrel{\mathop{\hbox to\bigaw@{\leftarrowfill}}\limits^{#1}_{#2}}\else
 \mathrel{\mathop{\hbox to\bigaw@{\leftarrowfill}}\limits^{#1}}\fi
 \ifCD@\hskip.5em\fi\amp@rs@nd@}
%
%
\newenvironment{CDS}{\@CDS}{\@endCDS}
\def\@CDS{\def\A##1A##2A{\llap{$\vcenter{\hbox
 {$\scriptstyle##1$}}$}\Big\uparrow\rlap{$\vcenter{\hbox{%
$\scriptstyle##2$}}$}&}%
\def\V##1V##2V{\llap{$\vcenter{\hbox
 {$\scriptstyle##1$}}$}\Big\downarrow\rlap{$\vcenter{\hbox{%
$\scriptstyle##2$}}$}&}%
\def\={&\hskip.5em\mathrel
 {\vbox{\hrule width\minCDaw@\vskip3\ex@\hrule width
 \minCDaw@}}\hskip.5em&}
\def\verteq{\Big\Vert&}
\def\novarr{&}
\def\noharr{&&}
\def\SE##1E##2E{\slantedarrow(0,18)(4,-3){##1}{##2}&}
\def\SW##1W##2W{\slantedarrow(24,18)(-4,-3){##1}{##2}&}
\def\NE##1E##2E{\slantedarrow(0,0)(4,3){##1}{##2}&}
\def\NW##1W##2W{\slantedarrow(24,0)(-4,3){##1}{##2}&}
\def\slantedarrow(##1)(##2)##3##4{%
\thinlines\unitlength1pt\lower 6.5pt\hbox{\begin{picture}(24,18)%
\put(##1){\vector(##2){24}}%
\put(0,8){$\scriptstyle##3$}%
\put(20,8){$\scriptstyle##4$}%
\end{picture}}}
\def\vspace##1{\noalign{\vskip##1\relax}}\relax\let\amp@rs@nd@&\iffalse}\fi
 \CD@true\vcenter\bgroup\relax\let\\=\cr\iffalse}\fi\tabskip\z@skip\baselineskip20\ex@
 \lineskip3\ex@\lineskiplimit3\ex@\halign\bgroup
 &\hfill$\m@th##$\hfill\cr}
\def\@endCDS{\cr\egroup\egroup}
%
\newdimen\TriCDarrw@
\newif\ifTriV@
\newenvironment{TriCDV}{\@TriCDV}{\@endTriCD}
\newenvironment{TriCDA}{\@TriCDA}{\@endTriCD}
\def\@TriCDV{\TriV@true\def\TriCDpos@{6}\@TriCD}
\def\@TriCDA{\TriV@false\def\TriCDpos@{10}\@TriCD}
\def\@TriCD#1#2#3#4#5#6{%
\setbox0\hbox{$\ifTriV@#6\else#1\fi$}
\TriCDarrw@=\wd0 \advance\TriCDarrw@ 24pt
\advance\TriCDarrw@ -1em
\def\SE##1E##2E{\slantedarrow(0,18)(2,-3){##1}{##2}&}
\def\SW##1W##2W{\slantedarrow(12,18)(-2,-3){##1}{##2}&}
\def\NE##1E##2E{\slantedarrow(0,0)(2,3){##1}{##2}&}
\def\NW##1W##2W{\slantedarrow(12,0)(-2,3){##1}{##2}&}

\def\slantedarrow(##1)(##2)##3##4{\thinlines\unitlength1pt
\lower 6.5pt\hbox{\begin{picture}(12,18)%
\put(##1){\vector(##2){12}}%
\put(-4,\TriCDpos@){$\scriptstyle##3$}%
\put(12,\TriCDpos@){$\scriptstyle##4$}%
\end{picture}}}
\def\={\mathrel {\vbox{\hrule
   width\TriCDarrw@\vskip3\ex@\hrule width
   \TriCDarrw@}}}
\def\>##1>>{\setbox\z@\hbox{$\scriptstyle
 \;{##1}\;\;$}\global\bigaw@\TriCDarrw@
 \ifdim\wd\z@>\bigaw@\global\bigaw@\wd\z@\fi
 \hskip.5em
 \mathrel{\mathop{\hbox to \TriCDarrw@
{\rightarrowfill}}\limits^{##1}}
 \hskip.5em}
\def\<##1<<{\setbox\z@\hbox{$\scriptstyle
 \;{##1}\;\;$}\global\bigaw@\TriCDarrw@
 \ifdim\wd\z@>\bigaw@\global\bigaw@\wd\z@\fi
 \mathrel{\mathop{\hbox to\bigaw@{\leftarrowfill}}\limits^{##1}}
 }
 \CD@true\vcenter\bgroup\relax\let\\=\cr\iffalse}\fi
 \tabskip\z@skip\baselineskip20\ex@
 \lineskip3\ex@\lineskiplimit3\ex@
 \ifTriV@
 \halign\bgroup
 &\hfill$\m@th##$\hfill\cr
#1&\multispan3\hfill$#2$\hfill&#3\\
&#4&#5\\
&&#6\cr\egroup%
\else
 \halign\bgroup
 &\hfill$\m@th##$\hfill\cr
&&#1\\%
&#2&#3\\
#4&\multispan3\hfill$#5$\hfill&#6\cr\egroup
\fi}
\def\@endTriCD{\egroup} 
\newcommand{\mc}{\mathcal} 
\newcommand{\mb}{\mathbb} 
\newcommand{\surj}{\twoheadrightarrow} 
\newcommand{\inj}{\hookrightarrow} \newcommand{\zar}{{\rm zar}} 
\newcommand{\an}{{\rm an}} \newcommand{\red}{{\rm red}} 
\newcommand{\Rank}{{\rm rk}} \newcommand{\codim}{{\rm codim}} 
\newcommand{\rank}{{\rm rank}} \newcommand{\Ker}{{\rm Ker \ }} 
\newcommand{\Pic}{{\rm Pic}} \newcommand{\Div}{{\rm Div}} 
\newcommand{\Hom}{{\rm Hom}} \newcommand{\im}{{\rm im}} 
\newcommand{\Spec}{{\rm Spec \,}} \newcommand{\Sing}{{\rm Sing}} 
\newcommand{\sing}{{\rm sing}} \newcommand{\reg}{{\rm reg}} 
\newcommand{\Char}{{\rm char}} \newcommand{\Tr}{{\rm Tr}} 
\newcommand{\Gal}{{\rm Gal}} \newcommand{\Min}{{\rm Min \ }} 
\newcommand{\Max}{{\rm Max \ }} \newcommand{\Alb}{{\rm Alb}\,} 
\newcommand{\GL}{{\rm GL}\,} 
\newcommand{\ie}{{\it i.e.\/},\ } \newcommand{\niso}{\not\cong} 
\newcommand{\nin}{\not\in} 
\newcommand{\soplus}[1]{\stackrel{#1}{\oplus}} 
\newcommand{\by}[1]{\stackrel{#1}{\rightarrow}} 
\newcommand{\longby}[1]{\stackrel{#1}{\longrightarrow}} 
\newcommand{\vlongby}[1]{\stackrel{#1}{\mbox{\large{$\longrightarrow$}}}} 
\newcommand{\ldownarrow}{\mbox{\Large{\Large{$\downarrow$}}}} 
\newcommand{\lsearrow}{\mbox{\Large{$\searrow$}}} 
\renewcommand{\d}{\stackrel{\mbox{\scriptsize{$\bullet$}}}{}} 
\newcommand{\dlog}{{\rm dlog}\,} 
\newcommand{\longto}{\longrightarrow} 
\newcommand{\vlongto}{\mbox{{\Large{$\longto$}}}} 
\newcommand{\limdir}[1]{{\displaystyle{\mathop{\rm lim}_{\buildrel\longrightarrow\over{#1}}}}\,} 
\newcommand{\liminv}[1]{{\displaystyle{\mathop{\rm lim}_{\buildrel\longleftarrow\over{#1}}}}\,} 
\newcommand{\norm}[1]{\mbox{$\parallel{#1}\parallel$}} 
\newcommand{\boxtensor}{{\Box\kern-9.03pt\raise1.42pt\hbox{$\times$}}} 
\newcommand{\into}{\hookrightarrow} \newcommand{\image}{{\rm image}\,} 
\newcommand{\Lie}{{\rm Lie}\,} 
\newcommand{\CM}{\rm CM}
\newcommand{\sext}{\mbox{${\mathcal E}xt\,$}} 
\newcommand{\shom}{\mbox{${\mathcal H}om\,$}} 
\newcommand{\coker}{{\rm coker}\,} 
\newcommand{\sm}{{\rm sm}} 
\newcommand{\tensor}{\otimes} 
\renewcommand{\iff}{\mbox{ $\Longleftrightarrow$ }} 
\newcommand{\supp}{{\rm supp}\,} 
\newcommand{\ext}[1]{\stackrel{#1}{\wedge}} 
\newcommand{\onto}{\mbox{$\,\>>>\hspace{-.5cm}\to\hspace{.15cm}$}} 
\newcommand{\propsubset} {\mbox{$\textstyle{ 
\subseteq_{\kern-5pt\raise-1pt\hbox{\mbox{\tiny{$/$}}}}}$}} 
\newcommand{\sA}{{\mathcal A}} 
\newcommand{\sB}{{\mathcal B}} \newcommand{\sC}{{\mathcal C}} 
\newcommand{\sD}{{\mathcal D}} \newcommand{\sE}{{\mathcal E}} 
\newcommand{\sF}{{\mathcal F}} \newcommand{\sG}{{\mathcal G}} 
\newcommand{\sH}{{\mathcal H}} \newcommand{\sI}{{\mathcal I}} 
\newcommand{\sJ}{{\mathcal J}} \newcommand{\sK}{{\mathcal K}} 
\newcommand{\sL}{{\mathcal L}} \newcommand{\sM}{{\mathcal M}} 
\newcommand{\sN}{{\mathcal N}} \newcommand{\sO}{{\mathcal O}} 
\newcommand{\sP}{{\mathcal P}} \newcommand{\sQ}{{\mathcal Q}} 
\newcommand{\sR}{{\mathcal R}} \newcommand{\sS}{{\mathcal S}} 
\newcommand{\sT}{{\mathcal T}} \newcommand{\sU}{{\mathcal U}} 
\newcommand{\sV}{{\mathcal V}} \newcommand{\sW}{{\mathcal W}} 
\newcommand{\sX}{{\mathcal X}} \newcommand{\sY}{{\mathcal Y}} 
\newcommand{\sZ}{{\mathcal Z}} \newcommand{\ccL}{\sL} 
 \newcommand{\A}{{\mathbb A}} \newcommand{\B}{{\mathbb B}} 
\newcommand{\C}{{\mathbb C}} \newcommand{\D}{{\mathbb D}} 
\newcommand{\E}{{\mathbb E}} \newcommand{\F}{{\mathbb F}} 
\newcommand{\G}{{\mathbb G}} \newcommand{\HH}{{\mathbb H}} 
\newcommand{\I}{{\mathbb I}} \newcommand{\J}{{\mathbb J}} 
\newcommand{\M}{{\mathbb M}} \newcommand{\N}{{\mathbb N}} 
\renewcommand{\P}{{\mathbb P}} \newcommand{\Q}{{\mathbb Q}} 
\newcommand{\nnu}{{\boldsymbol \nu}}
\newcommand{\mmu}{{\boldsymbol \mu}}

\newcommand{\eell}{{\boldsymbol r}}

\newcommand{\R}{{\mathbb R}} \newcommand{\T}{{\mathbb T}} 
\newcommand{\U}{{\mathbb U}} \newcommand{\V}{{\mathbb V}} 
\newcommand{\W}{{\mathbb W}} \newcommand{\X}{{\mathbb X}} 
\newcommand{\Y}{{\mathbb Y}} \newcommand{\Z}{{\mathbb Z}} 

\title{The Hilbert-Kunz density functions of
 quadric hypersurfaces}

\author{Vijaylaxmi Trivedi} 
\date{} \address{School of Mathematics, Tata Institute of Fundamental 
Research, Homi Bhabha Road, Mumbai-40005, India } \email{vija@math.tifr.res.in}

\keywords{Hilbert-Kunz multiplicity, spinor bundles, smooth quadric, ACM bundles}

\thanks{}
\begin{abstract}
We show that 
  the Hilbert-Kunz density function of a quadric hypersurface of Krull 
dimension  $n+1$ is a piecewise polynomial on a subset of $[0, n]$, whose 
complement  in $[0, n]$ has measure zero. Our explicit description of 
the Hilbert-Kunz density function confirms a 
 conjecture of Watanabe-Yoshida on the lower 
bound of the Hilbert-Kunz  multiplicity of the quadric of dimension $n+1$, 
provided the characteristic is at least $n-1$. We also show that
the Hilbert-Kunz multiplicity of a quadric of fixed dimension is 
an eventually strictly decreasing  function of the characteristic 
confirming a conjecture of Yoshida. 

The main input comes from the 
classification of Arithmetically Cohen-Macaulay  bundles on the 
projective variety  defined by the quadric via matrix factorizations. 
\end{abstract} \maketitle 
\section{Introduction}

Let $R$ be a Noetherian ring containing a field of characteristic $p>0$ and
 $I$ be an ideal of finite colength in $R$. For such a pair P. Monsky (in [M]) had 
introduced a characteristic $p$ invariant known as the Hilbert-Kunz (HK) 
multiplicity $e_{HK}(R, I)$. This is a positive real number given by 
$$e_{HK}(R, I) = \lim_{n\to \infty}\frac{\ell(R/I^{[q]})}{q^{\dim R}}\geq 1.$$

If $(R, {\bf m} , k)$ is a formally unmixed Noetherian local ring then it was 
proved by Watanabe-Yoshida (Theorem~1.5 in [WY1]) that $e_{HK}(R,{\bf m}) = 1$ 
if and only if $R$ is regular.  For the next best class of rings, namely 
the quadric hypersurfaces,  
 they made the following conjecture in 2005:

\vspace{5pt}

\noindent{\bf Conjecture}~(1)~~(Conjecture~4.2 in [WY2]).\quad{\it Let $p>2$ be 
prime and $K = {\bf {\bar {F}_p}}$ and let 
$$ R_{p, n+1} = K[x_0, \ldots, x_{n+1}]/ (x_0^2+\cdots + x_{n+1}^2)$$
denote  the quadric 
 hypersurface of dimension $n+1$. Then for any
formally 
unmixed non regular local ring
$(A, {\bf m}_A, K)$  of dimension $n+1$
we have 
$$e_{HK}(A, {\bf m}_A) 
\geq e_{HK}(R_{p,n+1}, {\bf m}) \geq 1+m_{n+1},$$

where $m_{n+1}$ are the constants occurring as the coefficients of 
the following expression 
$$\mbox{sec}(x)+\mbox{tan}(x) = 1+\sum_{n=0}^{\infty}m_{n+1}x^{n+1}, 
\quad\mbox{for}\quad |x|<\pi/2.$$}

In the same paper they showed that the conjecture holds for $n\leq 3$. 
The second inequality of the conjecture for $n\leq 5$ was proved by Yoshida in 
[Y]. Later the conjecture up to $n\leq 5$ was proved by Aberbach-Enescu in 
[AE2].

Enescu and Shimomoto in [ES] have proved the first inequality $e_{HK}(A) \geq 
e_{HK}(R_{p, n+1})$, where $A$ belongs to the class of
 complete intersection local rings.

On the other hand, around 2010,  Gessel-Monsky in [GM],  
proved  the following result:
$$\lim_{p\to \infty}e_{HK}(R_{p, n+1}, {\bf m}) = 1+m_{n+1}.$$

\vspace{5pt}

Progresses on Conjecture~(1) is discussed in Huneke's survey article [Hu].

Various people ([AE1], [AE2], 
Celikbas-Dao-Huneke-Zhang in [CDHZ]) have given a lower bound of the form 
 $e_{HK}(R, {\bf m}) \geq C(d)$, where $C(d)$ is an explicit constant and   
$(R, {\bf m})$ is an arbitrary 
 formally unmixed nonregular  local ring    
of dimension $d$.

However the above result of Gessel-Monsky implies that
 the lower bounds such as $C(d)$ are weaker than the bound given in the above 
conjecture: 
Recall so far the best constant  $C(d) = 1+\frac{1}{d!d^d}$ 
(see Section~8, Theorem~13 in [Hu]), whereas, 
as noted in [AE2],
$$1+m_3 = 1+\frac{1}{3!},\quad 1+m_4 = 1+\frac{5}{4!}, \quad 1+m_5 = 1+
\frac{16}{5!},
\quad 1+m_6  = 1+\frac{61}{6!}\quad\mbox{etc}.$$

Regarding the HK multiplicities of the quadric hypersurfaces,
  another conjecture was made by Yoshida ([Y]):

\vspace{5pt}

\noindent{{\bf Conjecture}~(2)}.\quad{\em $e_{HK}(R_{p, n+1}, {\bf m})$ is a 
decreasing function in $p$ for a fixed $n$.}

\vspace{5pt}
Note that it is not true in general that $e_{HK}(R_p)$ is a monotonic 
function of $p$. The first example which is  given by Han-Monsky ([HM]) is the ring 
$R_p = k[x,y,z]/(x^4+y^4+z^4)$, where $k$ is  a field of characteristic $p>2$
and  $e_{HK}(R_p, (x,y,z)))$ is a function of the  congruence class of 
$p~~(\mbox{mod}~~8)$
and not monotonic.

\vspace{5pt}

Conjecture~(1) and Conjecture~(2)
have been revisited in the recent paper [JNSWY]. 

\vspace{5pt}

Hilbert-Kunz density functions, which were  introduced in [T1], generalize
Hilbert-Kunz multiplicities in the graded setting. 
In this paper, our key new idea is to
describe the Hilbert-Kunz density functions of quadric hypersurfaces.

\vspace{5pt}
As applications of our description in Section~4, we obtain 
upper and lower  bounds of the values of the respective 
HK multiplicities of the quadric hypersurfaces.

In particular we prove the second inequality of the Conjecture~(1) 
by showing 

\vspace{5pt}

\noindent{\bf Theorem}~(A)~~(Theorem~{\ref{t1}}).\quad {\em Let $p\neq 2$ and  let $p > n-2$. Then, 
for $n\geq 3$, 
$$1+m_{n+1} + \tfrac{2n-4}{p} 
\geq e_{HK}(R_{p, n+1}, {\bf m})\geq   1+m_{n+1}.$$

 In fact the second inequality is strict if $p\geq 3n-4$.}

\vspace{5pt}

In Section~8, we obtain a more refined description of the 
Hilbert-Kunz density function than in Section~4. As an application, 
we show that

\vspace{5pt}

\noindent{\bf Theorem}~(B)~~(Theorem~\ref{ehkexpr} and 
Theorem~\ref{eehkexpr}).\quad{\em Given $n\geq 4$, there exist 
polynomials $p_n(t), q_n(t)\in  \Q[t]$ of degrees $\leq (n+1)^{n+3}$, 
 such that for all $p > 2^{\lfloor n/2\rfloor}(n-2)$
$$e_{HK}(R_{p, n+1}) = 1+m_{n+1}+ 
\frac{p_n(t)}{q_n(t)}\Bigr\rvert_{t={1/p}}.$$
The same assertion holds if $n=3$ and $p\geq 5$.}

Given $n$, the polynomials $p_n(t)$ and  $q_n(t)$  can be computed 
explicitly by 
Theorem~\ref{ehkexpr} and Theorem~\ref{eehkexpr}. For a prime $p>2$ and 
$n = 3, 4, 5$, the respective  Hilbert-Kunz 
multiplicity is also computed in [WY2] and [Y].

As an application of our computation, we provide the following confirmation
of  Conjecture~(2).

\vspace{5pt}

\noindent{\bf Theorem}~(C)~~(Theorem~\ref{pinfty}).\quad{\em Given any $n\geq 3$, there exists $\epsilon  
>0$ such that 
 $e_{HK}(R_{p, n+1})$ is a strictly decreasing function of $p$ for $p\geq 1/\epsilon$.}

\vspace{5pt}

A value of $\epsilon $ in Theorem~(C) is explicitly determined 
in Section~8; see Theorem~\ref{pinfty}.

\vspace{5pt}

Now we outline our methods of obtaining the Hilbert-Kunz density function 
and provide descriptions of the Hilbert-Kunz density functions. 

Recall  that for a pair $(R, I)$, where $I$ is a finite colength 
homogeneous ideal in a standard graded ring of dimension at least two, the 
 HK density 
function $f_{R,I}:\R\longto [0, \infty)$ is a
 compactly supported continuous function given by 
$$f_{R, I}(x) = \lim_{s\to \infty}\frac{\ell(R/I^{[q]})_{\lfloor xq \rfloor}}{q^{\dim R-1}}, 
\quad\mbox{where}\quad q=p^s.$$
The HK density function recovers the HK multiplicity since 
$$e_{HK}(R,I) = \int_0^{\infty} f_{R,I}(x)dx.$$ 

For further details we refer to Section~2.

The theory of Hilbert-Kunz density function was initiated in [T1] and 
subsequently developed in [TW2] (see Theorem~1.1 and the preceding definition  
in [T1] and [TW2]), and [T2] (see Proposition~2.1).  

The HK density function of $(R, I)$ tracks the length of the 
graded pieces of $R/I^{[p^n]}$ separately. This very feature of the 
HK density functions allows us to use the machinery of sheaf cohomology and
other algebro-geometric gadgets on $\mbox{Proj}(R)$ to 
describe the HK density functions; and thus in turn the HK multiplicity.
These advantages of using HK density functions 
 over the mere HK multiplicities 
are subsequently apparent.

Before going further  we remark that 
  for $n=1$ and $n=2$, 
 the ring $R_{p, n+1}$ is the homogeneous coordinate ring
of $\P^1_k$ and $\P_k^1\times \P_k^1$ respectively. In both the cases the invariants
  $e_{HK}(R_{p, n+1})$ and 
$f_{R_{p, n+1}, {\bf m}}$ are independent of the characteristic (see 
Eto-Yoshida [EY] and [T1]). 

Therefore in the rest of the paper we consider 
$$R_{p, n+1} = K[x_0, \ldots, x_{n+1}]/ (x_0^2+\cdots + x_{n+1}^2),\quad\mbox{where}\quad n\geq 3,$$
where $K$ is a perfect field of characteristic $p>0$; $R_{p, n+1}$ is 
equipped with the standard graded structure. Set $Q_n = 
\mbox{Proj}(R_{p, n+1})$.

One key input that facilitates calculation in the quadric 
hypersurface case is the complete classification of arithmetically 
Cohen-Macaulay bundles (abbreviated as ACM bundles) on $Q_n$ as a 
 direct sum of line bundles and twisted spinor bundles.

This classification follows from the classification of 
maximal Cohen-Macaulay modules on $R_{p, n+1}$ due to 
Knorrer [K] and Buchweitz-Eisenbud-Herzog [BEH]; we recall the 
relevant details in Section~2.
Since  $F^{s}_*(\sO(a))$ is an  ACM bundle on $Q_n$, for every
 $s^{th}$ iterated Frobenius map $F^s:Q_n\longto Q_n$ we have
\begin{equation}\label{*1}F_*^s(\sO(a)) = \oplus _{t\in \Z}\sO(t)^{\nu^s_t(a)} \oplus \oplus_{t\in \Z}
\sS(t)^{\mu^s_t(a)},\end{equation}
where $\sO(t) = \sO_{Q_n}(t)$ and $\sS(t)$ is a twisted spinor bundle if $n$ is odd, and 
a twisted direct sum of the two spinor bundles if $n$ is even.

Later Achinger  in [A] showed that 
 the multiplicities  of the bundles $\sO(t)$ and $\sS(t)$ occurring in 
(\ref{*1}) are related  to the lengths of 
graded components of the ring $R_{p, n+1}/{\bf m}^{[q]}$, with $q = p^s$,
 by the formula
\begin{equation}\label{*}\ell (R_{p, n+1}/{\bf m}^{[q]})_a = 
\nu^{s}_0(a) + 2\lambda_0 \mu^s_1(a),\end{equation}
where 
${\bf m} = (x_0, \ldots, x_{n+1})$ and $\lambda_0 =  2^{\lfloor n/2\rfloor}$.

Here  we crucially use this,  and  another result, namely Theorem~2 from [A] which
 determines the  integers $t_1$ and $t_2$,  in terms of $q=p^s$ and  $a$, for which 
 the bundles
 $\sO(t_1)$ and $\sS(t_2)$ occur with non-zero multiplicity  in  the decomposition 
of $F^s_*(\sO(a))$ (or of  $F^s_*(\sS(a))$). However [A] 
does not determine this multiplicity.

Using (\ref{*})  and the theory of HK density functions we show that 
 the functions 
$\nnu_{i}^{(p)}:\R\longto \R$, $\mmu_{i-1}^{(p)}:\R\longto \R$
given by
\begin{equation}\label{contifun}
\nnu_{i}^{(p)}(x) =  \lim_{s\to \infty}\frac{\nu^s_{-i}(\lfloor xq\rfloor)}{q^n}, 
\quad 
\mmu_{i-1}^{(p)}(x) = \lim_{s\to \infty}\frac{2\lambda_0\mu^s_{-i+1}(\lfloor xq\rfloor)}{q^n}\end{equation}
are   well defined functions continuous functions  and for any nonnegative
integer $i$
 $$f_{R_{p, n+1}}(x+i) = \nnu_{i}^{(p)}(x)+\mmu_{i-1}^{(p)}(x),\quad\mbox{for}~~ x\in [0, 1).$$

Next we  analyze the functions
$\nnu_{i}^{(p)}$, $\mmu_{i-1}^{(p)}$.

We show that there exists a subset of $[0, 1]$  
which we call the {\it difficult range} in the interval $[0, 1)$, provided  $p>n-2$,  
given as follows:
\begin{enumerate}
\item  $[\frac{1}{2}-\frac{n-2}{2p}, \quad \frac{1}{2}+\frac{n-2}{2p})$,\quad
if $n\geq 3$ is an odd integer and 
\item  $[0, ~~~\frac{n-2}{2p})\cup [1-\frac{n-2}{2p},~~~1)$,\quad
if $n\geq 4$ is an even  integer.
\end{enumerate}

We show (see Proposition~\ref{lp}) that the rank
 functions $\nnu_i^{(p)}$, $\mmu_{j}^{(p)}$, when restricted to  the 
complement of the difficult range in $[0, 1]$, are
{\it piecewise polynomials}. Moreover  we see that 
these polynomials  are independent of the 
characteristic $p$.
Such a  description  springs 
from the fact that 
 the bundle $F^s_*(\sO(\lfloor xq\rfloor))$ 
has atmost one twist of spinor bundle(s) if  $x$ is outside the difficult 
range.

The dependence of the density function of $R_{p, n+1}$ on the characteristic 
$p>0$ is only via  its description on the difficult range.
We note that 
 the length of the difficult range goes to zero  as $p\to \infty$. 
and in this range 
characteristic $p$ shows up in  the guise of functions $\mmu^{(p)}_{i}$, 
which is a bounded function 
(see Lemma~\ref{cont}).
More precisely we have:

\subsection*{The first formulation of $f_{R_{p, n+1}}$}
 Here we 
express $f_{R_{p, n+1}}$, piecewise, as a sum of 
a polynomial in  $\Q[x]$ (and therefore characteristic free) and
another nonnegative  continuous function which is exactly the ``pure characteristic 
$p$'' contribution.  We prove Proposition~\ref{l4} and Corollary~\ref{t2} 
in Section~4.

\begin{propose}\label{l4}Let $p\neq 2$ be a prime such that $p > n-2$ 
and $n_0 = \lceil {n}/{2}\rceil -1$. There exists an explicit  set 
$\{{\bf Z}_0(x), \ldots, {\bf Z}_{n_0+1}(x), 
{\bf Y}_{n_0+2}(x), \ldots, {\bf Y}_{n-1}(x)\}$
of   degree $n$ polynomials in 
$\Q[x]$ which are independent of the characteristic $p$, and a set 
$\{\mmu^{(p)}_{n_0-1}, \mmu^{(p)}_{n_0}, \mmu^{(p)}_{n_0+1}\}$  of 
bounded continuous  functions from $[0, 1) \longto \R_{\geq 0}$ such that 
 
\begin{enumerate}
\item if $n\geq 4$ is  an even number then

$$f_{R_{p, n+1}}(x) =
 \begin{cases}  {\bf Z}_{i}(x), \quad\mbox{if}\quad 
 i\leq x < i+1 \quad\mbox{and}\quad 0\leq i \leq  n_0\\\\
{\bf  Z}_{n_0+1}(x),\quad\mbox{if}\quad  
(n_0+1) \leq x < (n_0+2) -\frac{n-2}{2p}\\\\
 {\bf  Z}_{n_0+1}(x)+
{\mmu}^{(p)}_{n_0-1}(x-n_0-1),
\quad\mbox{if}\quad  1-\tfrac{n-2}{2p}\leq x -(n_0+1)< 1 \\\\
{\bf  Y}_{n_0+2}(x)+ 
{\mmu}^{(p)}_{n_0+1}(x-n_0-2),
\quad\mbox{if\quad  $0\leq x -(n_0+2) < \small{\frac{n-2}{2p}}$}\\\\
 {\bf Y}_{n_0+2}(x),\quad\mbox{if}\quad  
(n_0+2)+\frac{n-2}{2p} \leq x < (n_0+3)\\\\
  {\bf Y}_{i}(x),\quad\mbox{if}\quad 
 i\leq x < i+1 \quad\mbox{and}\quad n_0+3\leq i < n
\end{cases}$$
and $f_{R_p, n+1}(x) = 0$ elsewhere.

\vspace{10pt}

\item If $n\geq 3$ is  an odd number  then

$$f_{R_p, n+1}(x)  = 
\begin{cases} {\bf Z}_{i}(x) 
 \quad\mbox{if}\quad 
 i\leq x < i+1 \quad\mbox{and}\quad 0\leq i \leq  n_0\\\\
 {\bf  Z}_{n_0+1}(x), \quad\mbox{if}\quad  
(n_0+1) \leq x < (n_0+\frac{3}{2}) -\frac{n-2}{2p}\\\\
{\bf  Y}_{n_0+1}(x)+ \mmu^{(p)}_{n_0}(x-n_0-1),\quad
\mbox{if}\quad \tfrac{1}{2}-\tfrac{n-2}{2p}\leq x -(n_0+1) < 
\tfrac{1}{2}+\tfrac{n-2}{2p}\\\\
{\bf Y}_{n_0+1}(x),\quad\mbox{if}\quad  
(n_0+1)+\frac{1}{2}+\frac{n-2}{2p} \leq x < (n_0+2)\\\\
{\bf Y}_i(x),\quad\mbox{if}\quad 
 i \leq x < i+1 \quad\mbox{and}\quad n_0+2\leq i <  n
\end{cases}$$
and $f_{R_p, n+1}(x) = 0$ elsewhere.
Here we can also write 
$$f_{R_p, n+1}(x)  = 
{\bf  Z}_{n_0+1}(x)+ {\mmu^{(p)}_{n_0-1}}(x-n_0-1),\quad
\mbox{if}\quad \tfrac{1}{2}-\tfrac{n-2}{2p}\leq x -(n_0+1) < 
\tfrac{1}{2}.$$
\end{enumerate}
Moreover in both the cases  the function $f_{R_p, n+1}$ is partially symmetric 
and the symmetry is given by 
$$f_{R_p, n+1}(x) = f_{R_p, n+1}(n-x)\quad \mbox{for}\quad 0\leq x \leq 
\tfrac{n-2}{2}\left(1-\tfrac{1}{p}\right).$$ 
\end{propose}

We can look at this result as follows

\begin{cor}\label{t2}If $p > n-2$ and $p\neq 2$ then 
$$\begin{array}{lcll}
 f_{R_{p, n+1}, {\bf m}}(x) & = &  f_{R^{\infty}_{n+1}}(x) & 
x\in [0,\quad \frac{n+2}{2}-\frac{n-2}{2p}]\\\\

& = & f_{R^{\infty}_{n+1}}(x)  + \mmu^{(p)}_{n_0-1}(x-n_0-1)  &
x\in [\frac{n+2}{2}-\frac{n-2}{2p},\quad \frac{n+2}{2})\\\\

& = &  f_{R^{\infty}_{n+1}}(x) + \mmu^{(p)}_{i}(x-i-1)  & 
x\in [\frac{n+2}{2},\quad \frac{n+2}{2}+\frac{n-2}{2p}]\\\\
& = &  f_{R^{\infty}_{n+1}}(x) & 
x\in [\frac{n+2}{2}+\frac{n-2}{2p},\quad \infty),
\end{array}$$
where $i = n_0+1$ if $n$ is even and $i=n_0$ if $n$ is odd. Moreover 
the function $f_{R^{\infty}_{n+1}}:[0, \infty)
\longto [0, \infty)$ given by 
$$f_{R^{\infty}_{n+1}}(x) :=\lim_{p\to \infty}f_{R_{p, n+1}, {\bf m}}(x)$$ is a 
well defined continuous function.
\end{cor}

Theorem~(A) follows from Proposition~\ref{l4} and Corollary~\ref{t2} 
as detailed in Theorem~\ref{t1}.

\vspace{5pt}

Further, this  explicit description of support of the
$f_{R_p, n+1}$ along  with the result of [TW1] gives 
 the following result (in Corollary~\ref{c1})

\vspace{10pt}

\noindent{\bf Corollary}~\quad {\em Let $p>2$ be a prime number 
such that $p > n-2$ then the $F$-threshold of the ring
$R_{p ,n+1}$ is $c^{\bf m}({\bf m}) = n$.}

\subsection*{The second formulation of $f_{R_{p, n+1}}$}

In Section~7, we further analyse  the rank functions
$\nnu_i^{(p)}$, $\mmu^{(p)}_i$ on  the 
difficult range.

\begin{itemize}
\item We find a countable collection of mutually disjoint semi open intervals, 
{\em i.e.}, interval of the form $[a, b)$, which are contained in 
the difficult range. 
The complement of the union of these intervals in 
the difficult range has measure zero. Our indexing of the members 
in this countable collection is independent of the underlying characteristic.
See Lemma~\ref{emeasure}, Lemma~\ref{l1c}.
\item In Proposition~\ref{pex} and Proposition~\ref{epex}, we show 
that there is a finite set of matrices with entries in  $\Q[t]$,  $\Q[x]$
such that on a given semi open subinterval the rank functions 
$\nnu_i^{(p)}(x)$, $\mmu^{(p)}_i(x)$
are products of matrices from the above collection, evaluated at $t=1/{p}$.
Given the semi open subinterval the choices of matrices in the product 
only depends on the chosen indexing of the interval.
\end{itemize}

The description of  $\mmu_i^{(p)}$ in Proposition~\ref{pex} and 
Proposition~\ref{epex} gives us 
 a ground where we can compare the  integrals 
of $\mmu_i^{(p)}$ as $p$ varies; the characteristic free indexing 
of the semi open intervals facilitates the comparison.
 Further, using  detailed 
analysis, we choose one subinterval
on which the integral of 
$\mmu_i^{(p)}$ is a strictly decreasing function of $p$ and 
on all other subintervals it is a decreasing function of $p$.

Moreover the integrals  of $\mmu_i^{(p)}$ have nice features over such subintervals:
For example, for given integer $l\geq 1$,   $\int \mmu_i^{(p)}$ over 
the union of subintervals 
which are indexed by tuples of length $l$,  will correspond 
to  $\A^l$ where $\A$ is the  sum of the above mentioned  finite set of 
matrices in $\Q[t]$.
 In particular $\int \mmu_i^{(p)}$  over the difficult range involves a 
power series expression of the matrix. Now arguing that $\A$ has no eigenvalue $\geq 1$ 
   leads to the polynomials as given in Theorem~(B).

\vspace{10pt}

Using Theorem~(B) and the result of [ES] one can 
 answer the Conjecture~(1) for the class of complete local rings in the 
following way:

\vspace{5pt}

{\it Let  $n\geq 4$ and   let $p > 2^{\lfloor n/2\rfloor}(n-2)$.
Let $(A, {\bf m}_A, K)$ be  a complete intersection but nonregular local ring of 
dimension $n+1$. 
Then  
$$e_{HK}(A, {\bf m}_A) \geq e_{HK}(R_{p,n+1}, {\bf m}) = 
1+m_{n+1}+\frac{p_n(t)}{q_n(t)}\bigr\rvert_{t=1/p} > 1+m_{n+1},$$
where 
for given $n$, the polynomials $p_n(t)$ and  $q_n(t)$  can be computed explicitly by 
Theorem~\ref{ehkexpr} and Theorem~\ref{eehkexpr}.}

In particular we have a sharp lower bound on the class of complete intersection rings.

\vspace{5pt}

In the end  we explicitly write down the HK density function for $n=3$ case in 
Theorem~\ref{t3}.
Note that $n=3$ is the first $n$ where 
the $e_{HK}(R_{p,n+1})$ involves the characteristic $p$. Here 
the HK density function, which is a compactly supported 
continuous function, is non smooth at infinitely many points, however, 
intriguingly, it is a $\sC^2(\R)$ function. 
 This is different (see Remark~\ref{future}) from the 
 previous known cases such as  two dimensional graded rings 
(see Example~3.3 in [T1]) or projective toric varieties 
(see [MT] Theorem~3.4 and the proof of Theorem~1.1, or 
Theorem~4.1.18 of [Mo]) 
  where the set of non smooth points is a  finite set.

\vspace{5pt}

\noindent{\bf Question~1}:\quad{\em If $R$ is a standard graded ring over
a perfect
field and is
of dimension $d\geq2$ and $I\subset R$ is a graded ideal of finite colength,
 then when does the HK density function $f_{R,I}$  belong to  $\sC^{d-2}(\R)$?}

An approach to address this question is mentioned in [Muk], where it
is claimed that an affirmative answer to one of the questions in
Question~5.6 shows that the HK density function is in $\sC^{d-2}(\R)$.

\vspace{5pt}

Based on the examples of quadric hypersurfaces 
and also other above mentioned examples
we pose a variant (suggested by the referee) of a question due to K.I. 
Watanabe (private communication).

\vspace{5pt}

\noindent{\bf Question~2}:\quad{\em Let  $(R, I)$ be  a standard graded 
pair of  dimension $d\geq 2$. Let $[0, \alpha]$ be the support of 
the  the HK density function $f_{R, I}$ of $(R, I)$. Does there exist a 
subset $S$ of $[0, \alpha]$ of Lebesgue measure zero such that 
\begin{enumerate}
\item There is a countable collection of mutually disjoint semi open intervals $\{[a_n, b_n)\}_{n\in \N}$ such that 
$$[0, \alpha]\setminus S = \cup_{n\in \N}[a_n, b_n);$$
\item and for each $n$, there is a polynomial $P_n\in \R[x]$ such that 
on $[a_n, b_n)$ the density function is given by $P_n$?
\end{enumerate}
Moreover, can $S$ be taken to  a countable set?}

\vspace{5pt}

Looking further, the methods used in this paper suggest possible computations for the HK density and related 
invariants in other situations, where we have information on ACM bundles 
using matrix factorizations.

\vspace{5pt}

We would like to thank the referee for a careful and thorough 
reading of the paper, and for detailed suggestions to improve 
the exposition.

\vspace{5pt}

\begin{notations} 
In the rest of the paper 
$$R_{p, n+1} = \frac{k[x_0, \ldots, x_{n+1}]}{(x_0^2+\cdots+ x_{n+1}^2]}
\quad\mbox{and}\quad Q_n = \mbox{Proj}~R_{p, n+1},$$
where $n\geq 3$ and $k$ is a perfect field of characteristic $p> 2$.
\end{notations}

\section{preliminaries}

In this section  we recall the relevant results which are known in the literature.

\vspace{5pt}

First we  recall the following 
 notion 
of HK density function for $(R, I)$, where 
$R$ is a standard graded ring  and $I$ is a homogeneous ideal in $R$ 
of finite colength, which  was introduced  by the author ([T1]) for 
standard graded rings and 
later this notion was generalized by the author and Watanabe ([TW2]) for 
$\N$-graded rings.

\begin{defn}\label{dhkd} Let $R$ be a Noetherian standard graded ring of 
dimension 
$d\geq 2$ over a perfect field of characteristic $p>0$, and let 
$I\subset R$ be a homogeneous ideal such that 
$\ell(R/I) < \infty$. For $s\in \N$ and $q = p^s$, let 
$f_s:[0, \infty)\longto [0, \infty)$ be defined as 
$$f_s(R, I)(x) = \frac{\ell(R/I^{[q]})_{\lfloor xq \rfloor}}{q^{\dim R-1}}, 
\quad\mbox{where}\quad q=p^s.$$

\vspace{5pt}

\noindent {\bf Theorem}~ (Theorem~1.1 of [T1]).\quad {\em The sequence 
$\{f_s(R, I)\}_{s\in\N}$ converges 
uniformly to a compactly supported continuous function 
$f_{R, I}:[0, \infty)\longto [0, \infty)$, where 
$$f_{R, I}(x) = \lim_{s\to \infty}f_s(R, I)(x). \quad\mbox{Moreover}\quad
e_{HK}(R, I) = \int_0^\infty f_{R, I}(x)dx.$$}

We call $f_{R, I}$ to be the HK density function of $R$ with respect to the ideal $I$.
\end{defn}

\begin{defn} A vector bundle $E$ on a smooth $n$-dimensional hypersurface 
$X = \mbox{Proj}~S/(f)$, where $S= k[x_0, \ldots, x_{n+1}]$ 
is called arithmetically Cohen-Macaulay (ACM) if 
$H^i(X, E(m))= 0$, for $0 < i < n$ and for all $m$. 

It is easy to check that 
a vector bundle $E$ on $X$  is ACM if and only if the corresponding graded 
$S/(f)$-module 
$\oplus_{j\in \N}H^0(X, E(j))$
is maximal Cohen-Macaulay (MCM).

\end{defn}

Let $Q_n = \mbox{Proj}~S/(f)$ be the quadric given by the hypersurface 
 $f = x_0^2+\cdots +x_{n+1}^2 = 0$ in $\P^{n+1}_k = \mbox{Proj}~S$, where $n\geq 3$. 
Let $k$ be an 
algebraically closed field. Henceforth we assume $n>2$.

By B-E-H classification ([BEH]) of indecomposable graded 
MCM modules over
quadric we have: Other than free modules on $S/(f)$, there is (up to shift) only 
one  indecomposable module $M$ (which corresponds to  the single spinor bundle $\Sigma$ on $Q_n$)
 if $n$ is odd and there are only two of them 
$M_+$ and $M_-$ (which correspond to the 
two spinor bundles $\Sigma_+$ and $\Sigma_{-}$  on $Q_n$) if $n$ is even.

Moreover an MCM module over $S/(f)$ corresponds to  a {\it matrix factorization} of the 
polynomial $f$  (such an equivalence is given by Eisenbud in [E], for more general 
hypersurfaces $(f)$), which is a pair $(\phi, \psi)$ of square matrices of polynomials, 
of the  same size, 
such that 
$\phi\cdot\psi = f\cdot id = \psi\cdot \phi$ and the MCM module is the 
cokernel of $\phi$.

Now  the matrix factorization $(\phi_n, \psi_n)$ for indecomposable bundles on 
$Q_n$ (see Langer [L], Section~2.2)
gives an exact sequence  of locally free sheaves  on $\P_k^{n+1}$, where 
$i:Q_n\longrightarrow \P_k^{n+1}$ is the natural inclusion.
  
\begin{equation}\label{a1}0\longto \sO_{\P_k^{n+1}}(-2)^{2^{\lfloor n/2\rfloor 
+1}}\longby{\Phi_n}
\sO_{\P_k^{n+1}}(-1)^{2^{\lfloor n/2\rfloor +1}}\longto i_*\sS\longto 0,
\end{equation}
$\sS = \Sigma $ and $\Phi_n = \phi_n = \psi_n$  for $n$ odd and 
$\sS = \Sigma_{+}\oplus \Sigma_{-}$ and $\Phi_n = \phi_n\oplus \psi_n$  
for $n$ even. 
Moreover 
we have  (see (2.4), (2.5) and (2.6) in [L])  the short exact 
sequences of vector bundles on $Q_n$:
If $n$ is odd
then 
$$ 0\longto \sS \longto \sO_{Q_n}^{2^{\lfloor n/2\rfloor+1}}\longto \sS(1)\longto 0.$$
If $n$ is even then 
$$ 0\longto \Sigma_{-} \longto \sO_{Q_n}^{2^{\lfloor n/2\rfloor}}\longto \Sigma_{+}(1)
\longto 0$$
and 
$$0\longto \Sigma_{+} \longto \sO_{Q_n}^{2^{\lfloor n/2\rfloor}}\longto \Sigma_{-}(1)
\longto 0.$$

We also have the natural  exact sequence 

\begin{equation}\label{a2}0\longto \sO_{\P_k^{n+1}}(-2)\longto 
\sO_{\P_k^{n+1}}\longto \sO_{Q_n}\longto 0.\end{equation}

We denote  $\sO_{Q_n}(m) = \sO(m)$ and 
$$R_{p, n+1} = \frac{k[x_0, \ldots, x_{n+1}]}{(x_0^2+\cdots+ x_{n+1}^2]} =  
\oplus_{m\geq 0}H^0(Q_n, \sO(m))\quad\mbox{and}\quad n\geq 3,$$ 
where  $k$ is a field of characteristic $p> 2$.
In particular  the $m^{th}$ graded component of $R_{p, n+1}$ is 
$H^0(Q_n, \sO(m))$. We will be using the following set of equalities 
in our forthcoming computations.

\begin{equation}\label{ee} \left\{
\begin{matrix} h^0(Q_n, \sO(m))  = h^0(\P_k^{n+1}, \sO_{\P_k^{n+1}}(m))-
h^0(\P_k^{n+1}, \sO_{\P_k^{n+1}}(m-2))\\\
h^0(Q_n, \sS(m))  =  2\lambda_0[h^0(\P_k^{n+1}, \sO_{\P_k^{n+1}}(m-1))-
h^0(\P_k^{n+1}, \sO_{\P_k^{n+1}}(m-2))],\end{matrix}\right.\end{equation}

where $2\lambda_0 = 2^{\lfloor n/2\rfloor +1}$.

We note that  $\omega_{Q_n} = \sO(-n)$ and $\sS^{\vee} = \sS(1)$, where 
 $\omega_{Q_n}$ denotes the canonical bundle of $Q_n$. Therefore 
 by Serre duality
\begin{equation}\label{sd}\left\{
\begin{matrix} h^n(Q_n, \sO(m)) = h^0(Q_n, \sO(-m-n))\\\
 h^n(Q_n, \sS(m)) = h^0(Q_n, \sS(1-m-n)).\end{matrix}\right.\end{equation}

The rank of $Q_n$-bundle $\sS = \lambda_0 = 2^{\lfloor n/2\rfloor }$.

\vspace{5pt}

The statement and the proof of the following lemma is contained in [A].

\begin{lemma}\label{rk1}
For given integer $a$ and $q=p^s$, if the nonnegative 
integers $\nu^s_{t}(a)$, $\mu^s_{t}(a)$ are the integers occurring in the 
decomposition  
$$F_*^s(\sO(a))
 = \oplus _{t\in \Z}\sO(t)^{{\nu}^s_t(a)} \oplus \oplus_{t\in \Z}
\sS(t)^{{\mu}^s_t(a)}$$
then
$$\ell(R_{p, n+1}/{\bf m}^{[q]})_a 
= \nu^{s}_0(a) + 2\lambda_0 \mu^s_1(a),\quad\mbox{where}\quad 
{\bf m} = (x_0, \ldots, x_{n+1}).$$
\end{lemma}
\begin{proof} 
Since $\sO(a)$ and $\sS(a)$ are ACM bundles (follows from (\ref{a1}) and 
(\ref{a2})), the  projection formula implies that  
$F_*^s(\sO(a))$ is an ACM bundle on $Q_n$. 
Thus for $q=p^s$ and  $a\in \Z$, we indeed have  a decomposition of the form

\begin{equation}\label{dc}F_*^s(\sO(a))
 = \oplus _{t\in \Z}\sO(t)^{\nu^s_t(a)} \oplus \oplus_{t\in \Z}
\sS(t)^{\mu^s_t(a)}.\end{equation}

Restricting the Euler sequence in
$\P_k^{n+1}$ to $Q_n$ we get
the short exact sequence
$$0\longto \Omega^1_{\P^{n+1}_k}(1)\mid_{Q_n}\longto 
\oplus^{n+2}\sO\longto \sO(1)\longto 0$$
of sheaves of $\sO$-modules,
where the second map is given by $(a_0, \ldots, a_{n+1})\to \sum a_ix_i$.
This gives the long exact sequence 

\begin{multline*}0\longto H^0(Q_n, F^{s*}\Omega^1_{\P^{n+1}_k}(1)\mid_{Q_n}
\tensor \sO(a)) 
\longto H^0(Q_n, \oplus^{n+2} F^{s*}\sO(a))\longby{\Psi_{a+q}} 
H^0(Q_n, \sO(a+q))\\
\longto H^1(Q_n, F^{s*}\Omega^1_{\P^{n+1}_k}(1)\mid_{Q_n}\tensor \sO(a)) 
\longto 0.
\end{multline*}

Therefore
\begin{multline*}
\ell\Bigl(\frac{R_{p, n+1}}{{\bf m}^{[q]}}\Bigr)_{a+q} = 
\ell(\coker~\Psi_{a+q})\\ = h^1\bigl(Q_n, 
F^{s*}\Omega^1_{\P^{n+1}_k}(1)\mid_{Q_n}\tensor \sO(a)\bigr)
 = h^1\bigl(Q_n, \Omega^1_{\P^{n+1}_k}\mid_{Q_n}\tensor F^s_*\sO(a+q)\bigr).\end{multline*}

 Now by Lemma~1.2 in [A] we have 
$$h^1\bigl(Q_n, \Omega^1_{\P^{n+1}_k}(t)\mid_{Q_n}\bigr) = \delta_{t, 0}\quad and \quad 
h^1\bigl(Q_n, \sS\tensor \Omega^1_{\P^{n+1}_k}(t)\mid_{Q_n}\bigr) = 2^{\lfloor n/2\rfloor +1} \delta_{t, 1}.$$ 
Hence 
$$\ell(R_{p, n+1}/{\bf m}^{[q]})_a = \coker~{\Psi_{a}}
= \nu^{s}_0(a) + 2\lambda_0 \mu^s_1(a).$$
\end{proof}

\begin{lemma}\label{rk2}Let $i\geq 0$ be an integer and  $x\in [i, i+1)$
then 
$$\nu^s_0(\lfloor xq \rfloor)  = \nu^s_{-i}(\lfloor xq \rfloor-iq)\quad
\mbox{and}\quad 
\mu^{s}_1(\lfloor xq \rfloor)  = \mu^{s}_{-i+1}(\lfloor xq \rfloor-iq).$$
In particular  
$$f_{R_{p, n+1}, {\bf m}}(x) = \nnu_i^{(p)}(x-i)+\mmu_{i-1}^{(p)}(x-i)
\quad\mbox{for}\quad x\in [i, i+1),$$
where $\nnu_i^{(p)}:\R\longto \R$ and $\mmu_i^{(p)}:\R\longto \R$ are 
the functions as defined in 
(\ref{contifun}).
\end{lemma}
\begin{proof} Note that for any integer $m\geq 0$, there is an integer 
$i\geq 0$ such that 
$iq\leq m <(i+1)q$.
Hence by the  projection formula 
\begin{equation}\label{ap}F_*^s(\sO(m)) \simeq 
F^s_*\bigl(\sO(m-iq)\tensor F^{s*}(\sO(i))\bigr) \simeq
F^s_*\bigl(\sO(m-iq)\bigr)\tensor \sO(i).\end{equation}

 In particular  
\begin{equation}\label{e1}\ell (R_{p, n+1}/{\bf m}^{[q]})_m 
= \nu^{s}_0(m) + 2\lambda_0 \mu^s_1(m) = 
\nu^s_{-i}(m-iq) +2\lambda_0 \mu^s_{-i+1}(m-iq).\end{equation}

Therefore for $x\in [i, i+1)$ 
\begin{multline*}
f_{R_{p, n+1}, {\bf m}}(x)  
 =  \lim_{s\to \infty}\frac{\ell(R/{\bf m}^{[q]})_{\lfloor xq \rfloor}}{q^n}\\
  = \lim_{s\to \infty}
\frac{\nu^s_{-i}(\lfloor xq \rfloor-iq) +2\lambda_0 \mu^s_{-i+1}(\lfloor xq 
\rfloor-iq)}{q^n} =  (\nnu_i^{(p)}+\mmu_{i-1}^{(p)})(x-i).\end{multline*}
\end{proof}

We also use the following   
 result of Achinger (Theorem~2 in [A]) which  determines, in terms of $s$, $a$ and $n$,
 when the numbers
$\nu^s_i(a)$ and $\mu^s_i(a)$ are nonzero in the decomposition of $F^s_*(\sO(a))$. 
Langer in [L] has given such formula for the occurance of line bundles in 
 the  Frobenius direct image. 

\vspace{5pt}

\noindent{\bf Theorem}\quad[A].\quad  {\it Let $p\neq 2$, $s\geq 1$ and $n\geq 3$ and 
$$F_*^s(\sO(a)) = \oplus _{t\in \Z}\sO(t)^{\nu^s_t(a)} \oplus \oplus_{t\in \Z}
\sS(t)^{\mu^s_t(a)}.$$
$$F_*^s(\sS(a)) = \oplus _{t\in \Z}\sO(t)^{{\tilde {\nu}}^s_t(a)} \oplus \oplus_{t\in \Z}
\sS(t)^{{{\tilde \mu}^s}_t(a)}.$$

Then
\begin{enumerate}
\item $F_*^s(\sO(a))$ contains $\sO(t)$ if and only if $0\leq a-tq \leq n(q-1)$.
\item $F_*^s(\sO(a))$ contains $\sS(t)$ if and only if 
$$\left(\tfrac{(n-2)(p-1)}{2}\right)\tfrac{q}{p} \leq a-tq \leq
\left(\tfrac{(n-2)(p-1)}{2} + n-2 +p\right)\tfrac{q}{p} -n.$$
\item $F_*^s(\sS(a))$ contains $\sO(t)$ if and only if $1\leq a-tq \leq n(q-1)$.
\item $F_*^s(\sS(a))$ contains $\sS(t)$ if and only if
$$\left(\tfrac{(n-2)(p-1)}{2}\right)\tfrac{q}{p} +1 -\delta_{s,1}\leq a-tq \leq
\left(\tfrac{(n-2)(p-1)}{2} + n-2 +p\right)\tfrac{q}{p} -n +\delta_{s,1}.$$
\end{enumerate}}

\vspace{5pt}

\section{Formula for  the rank functions $\nnu_i^{(p)}$ and  $\mmu_{i-1}^{(p)}$}

Let the rank functions 
$\nnu_i^{(p)}$  $\mmu_{i-1}^{(p)}$ be as given in (\ref{contifun}). 
As noted earlier
to analyze the rank functions it is enough to study on the interval $[0, 1)$.

One of the goal of this section is to  determine, for a given rank function
 a   
 subinterval in $[0, 1)$ such that the function 
   is {\em polynomial}, that is there is a polynomial $p(x) \in \Q[x]$ such that 
the rank function at $x$ is equal to $p(x)$ for $x$ in that subinterval.
We give the description of such subintervals in  Proposition~\ref{rc}.

Since this section is technical,  the reader can choose to skip the proof of 
Lemmas~\ref{l1}, \ref{l2}, \ref{l3} and \ref{cont} 
 for the moment and move to the next section.

\begin{notations}\label{n1} \begin{enumerate}
\item For given integer $a$ and $q=p^s$, the nonnegative 
integers $\nu^s_{t}(a)$, $\mu^s_{t}(a)$ are the integers occurring in the 
decomposition  
$$F_*^s(\sO(a))
 = \oplus _{t\in \Z}\sO(t)^{{\nu}^s_t(a)} \oplus \oplus_{t\in \Z}
\sS(t)^{{\mu}^s_t(a)}.$$

\item For the sake of abbreviation henceforth we will denote
 \begin{enumerate}
\item $\sO_{Q_n}(m)) = \sO(m)$,   
\item $h^0(Q_n, \sO(m)) = L_m$ and 
\item $2\lambda_0 = 2^{\lfloor n/2\rfloor +1}$.
\end{enumerate}
\item Let $n_0$ and $\Delta$ be  given as  
 $$n_0 = \lceil \tfrac{(n-2)(p-1)}{2p}\rceil\quad\mbox{and}\quad
 n_0 - \Delta = \tfrac{(n-2)(p-1)}{2p}.$$
Hence for $p>n-2 \implies n_0 = \lceil n/2\rceil -1$
\item
A  bundle is {\em a spinor bundle of type} $t$ if it is isomorphic to $\sS(t)$, where 
$\sS = \Sigma$ if $n$ is odd and $\sS = \Sigma_++\Sigma_-$ if $n$ is even.
We say two spinor bundles $\sS(t)$ and $\sS(t')$ are of {\it the same type} if $t= t'$.
\end{enumerate}
\end{notations}

\vspace{5pt}

First we make the observation that 
 there can occur at most $\lceil (n-2)/p\rceil +2$ twists
of spinor bundles in the decomposition of 
$F^s_*(\sO(a))$, hence for $p > n-2$ the number reduces to three. 
The next result shows for $p>n-2$ and any $s$ and $a$, there can be 
at most two twists of Spinor bundles appearing in the decomposition of 
$F_*^s(\sO(a))$.

\begin{lemma}\label{l1}If $1-n \leq a < q =p^s$ and $p >2$ then 
$$F^s_*(\sO(a))  =  
\oplus_{t=0}^{n-1}\sO(-t)^{{\nu}^s_{-t}(a)}\oplus
\oplus_{i=-1}^{\lceil \tfrac{n-2}{p}\rceil}
 \sS(-n_0-i)^{\mu^s_{-n_0-i}(a)}.$$

\begin{enumerate}
\item If $n\geq 4$ is an even number and  $n-2 < p$  then 
 $n_0 = ({n-2})/{2}$
and 
 \begin{equation}\label{dce}F_*^s(\sO(a)) = \oplus_{t=0}^{n-1}\sO(-t)^{{\nu}^s_{-t}(a)} \oplus
\oplus_{i=n_0-1}^{n_0+1}
 \sS(-i)^{\mu^s_{-i}(a)}.\end{equation}
Moreover 
\begin{enumerate}
\item $\mu^s_{-n_0+1}(a)\neq 0 \implies a/q\in [ 1-\frac{n-2}{2p},~~ 1)$
\item $\mu^s_{-n_0}(a)\neq 0 \implies  a/q \in [0,~~ 1)$
\item $\mu^s_{-n_0-1}(a)\neq 0  \implies a/q \in [0,~~ \frac{n-2}{2p})$.
\end{enumerate}

\item If  $n\geq 3$ is an odd number and $n-2 < p$ then  
$n_0 = ({n-1})/{2}$  and 
  \begin{equation}\label{dco} F_*^s(\sO(a)) = \oplus _{t=0}^{n-1}\sO(-t)^{{\nu}^s_{-t}(a)} \oplus
 \sS(-n_0+1)^{{\mu}^s_{-n_0+1}(a)}\oplus \sS(-n_0)^{{\mu}^s_{-n_0}(a)}.\end{equation}
Moreover 
\begin{enumerate}
\item $\mu^s_{-n_0+1}(a)\neq 0 \implies a/q\in [\frac{1}{2} - \frac{n-2}{2p},~~ 1)$
\item $\mu^s_{-n_0}(a)\neq 0 \implies a/q \in [0,~~ \frac{1}{2}+\frac{n-2}{2p})$.
\end{enumerate}
\end{enumerate}
\end{lemma}
\begin{proof}The formula for the decomposition for $F^s_*(\sO(a))$
for $1-n\leq a <q$ follows  from the assertion~(1) of Theorem~[A].

By the  assertion~(2) of [A], if $\sS(t)$ occurs in the decomposition of 
$F^s_*(\sO(a))$ then 
$$ (n_0- \Delta)q \leq a -tq \leq  (n_0- \Delta)q + \tfrac{(n-2)q}{p} + q -n$$
$$ \implies (n_0- \Delta) \leq \tfrac{a}{q} -t \leq  (n_0- \Delta) + \tfrac{(n-2)}{p} + 1 
- \tfrac{n}{q}.$$

Hence 
\begin{equation}\label{os}0\leq \tfrac{a}{q} + \Delta -t-n_0 
\leq \tfrac{n-2}{p} + 1 -\tfrac{n}{q}.\end{equation}

We note that $0\leq a/q +\Delta <2$ and $t$ and $n_0$ are integers. Therefore
$$n_0\leq a/q +\Delta -t <2-t \implies n_0-1 \leq -t$$
and 
$$  -t-n_0 -1  \leq a/q+\Delta -t -n_0-1\leq 
 \tfrac{n-2}{p} -\tfrac{n}{q} < \lceil\tfrac{n-2}{p}\rceil.$$

This proves $-n_0-t\in \{-1, 0, \cdots, \lceil {(n-2)}/{p}\rceil\}$.

In particular  if   $n-2 < p$ then 
  $-n_0-t\in \{-1, 0, 1\}$. 

\vspace{5pt}
\begin{enumerate}
\item[] Let $n$ be even and $n-2 <p$ then $n=2m$ for some integer $m\geq 2$. 
Then  $$
n_0  = \lceil \tfrac{(n-2)(p-1)}{2p}\rceil  = 
\lceil \tfrac{(m-1)(p-1)}{p}\rceil = \lceil (m-1) -\tfrac{m-1}{p}\rceil = m-1 = 
\tfrac{n-2}{2}.$$
This implies $\Delta = (n-2)/2p$.

\item[]
Let $n$ be odd and $n-2 <p$  then $n=2m+1$ for some integer $m\geq 1$.  Now
$$n_0  = 
\lceil \tfrac{(2m-1)(p-1)}{2p}\rceil = 
\lceil \tfrac{2m(p-1)}{2p}- \tfrac{p-1}{2p}\rceil
 = \lceil m -\tfrac{2m+p-1}{2p}\rceil = m = 
\tfrac{n-1}{2},$$
where the second last equality follows as $0 < (2m+p-1)/2p <1$.
This implies $\Delta = (n-2)/2p+1/2$.
\end{enumerate}

Putting the three possible values of $t$ in   (\ref{os})
we get the following three 
inequalities

\begin{enumerate}
\item If $-t = n_0-1$ then we 
have $0\leq \tfrac{a}{q} +\Delta -1 \leq 1 + \frac{n-2}{p}-\frac{n}{q} < 1 + \frac{n-2}{p}$.\\
\item If $-t = n_0$ then we 
have $0\leq \tfrac{a}{q} +\Delta  \leq 1 + \frac{n-2}{p}-\frac{n}{q} < 1 + \frac{n-2}{p} $.\\
\item If $-t = n_0+1$ then we 
have $0\leq \tfrac{a}{q} +\Delta  \leq  \frac{n-2}{p}-\frac{n}{q} < \frac{n-2}{p}$.
\end{enumerate}
From these three equalities it is easy to obtain  
 the assertions  (1)~(a), (1)~(b), (1)~(c), (2)~(a) and (2)~(b).
\end{proof}

\begin{notations}\label{n2} Fix an integer $n\geq 3$. Then for an integer
$a\geq  1-n$
$$L_a = (2a+n)\frac{(a+n-1)\cdots (a+1)}{n!} 
= \frac{2a^n}{n!} + O(a^{n-1}).$$
We note that $L_a = 0$, for $1-n\leq a \leq -1$.
Let $q=p^s$, where $s\geq 1$.
  For $0\leq i \leq n_0+1$, we define iteratively the integers
$Z_{-i}(a, q)$ as follows:

$$ Z_0(a,q) = Z_0(a) = L_{a},\quad  
\quad Z_{-1}(a, q) = Z_0(a+q)-L_1Z_0(a)$$
 and in general 
$$Z_{-i}(a, q) = Z_0(a+iq) - \left[L_1Z_{-i+1}(a, q)+L_2Z_{-i+2}(a, q)
+\cdots +L_i Z_0(a)\right].$$
 
Similarly,  for $n_0+1\leq i \leq n-1$, 
we define iteratively another set of integers  $Y_{-i}(a, q)$
 as follows:
$$Y_{-n+1}(a, q) = L_{q-a-n} $$
and for  $n_0+1\leq i < n-1$
 $$ Y_{-i}(a, q) = L_{(n-i)q-a-n}-\left[L_{n-i-1}Y_{-n+1}(a, q)+
\cdots + L_1Y_{-i-1}(a, q)\right].$$ 
\end{notations}

\begin{rmk}\label{r3} By construction it follows that
there exists unique rational numbers $\{r_{ij}, s_{ik}\}_{j, k}$ such that 
for all $a\in \Z$ and $q = p^s$ we have  
$Z_0(a) = L_{a}$ and for $0\leq i\leq n_0+1$

 $$Z_{-i}(a, q) = r_{i0}L_{a} + r_{i1}L_{a+q}+ 
\cdots  + r_{i(i-1)}L_{a+(i-1)q}+L_{a+iq}$$
and 
$$Y_{-i}(a, q) = s_{i0}L_{q-a-n} + s_{i1}L_{2q-a-n}+ \cdots  + 
s_{i(n-i-2)}L_{(n-i-1)q-a-n}+L_{(n-i)q-a-n}.$$

Now let  $\{{\bf Z}_i(x)\}_{0\leq i \leq n_0+1}$ and 
$\{{\bf Y}_i(x)\}_{n_0+1\leq i <n}$ denote  the two sets of  polynomials, where  

$${\bf Z}_{i}(x) = 
\frac{2}{n!}\left[r_{i0}(x-i)^n+r_{i1}(x-i+1)^n+\cdots +
r_{i,i-1}(x-1)^n + (x)^n\right]$$
and 
$${\bf Y}_{i}(x) = 
\frac{2}{n!}\left[s_{i0}(i+1-x)^n+s_{i1}(i+2-x)^n+\cdots +s_{i,n-i-2}(n-1-x)^n
+(n-x)^n\right].$$

Since the set $\Z[1/p]$ is  a dense subset of $\R$ 
we have
$ \lim_{q\to \infty} {\lfloor xq \rfloor}/{q} = x$ for $x\geq 0$. 
This implies 
$${\bf Z}_i(x+i) = \lim_{q\to \infty}\frac{Z_{-i}(\lfloor xq\rfloor, q)}{q^n} 
\quad\mbox{and}\quad
{\bf Y}_i(x+i) = \lim_{q\to \infty}\frac{Y_{-i}(\lfloor xq\rfloor, q)}{q^n}$$ 
for $x\in [0, 1)$.
\end{rmk}

\begin{lemma}\label{l2} Let $p$ be an odd prime such that 
$p > n-2$. Then for given  integer $1-n\leq a <q = p^s$  
\begin{enumerate}

\item $\nu^s_{-i}(a)+2\lambda_0\mu^s_{-i+1}(a) = Z_{-i}(a, p^s)$, if  $0\leq i \leq n_0$.\\\

\item $\nu^s_{-n_0-1}(a)+2\lambda_0\mu^s_{-n_0}(a) = 
Z_{-n_0-1}(a, p^s) + 2\lambda_0\mu^s_{-n_0+1}(a)$.\\\

\item $\nu^s_{-i}(a)+2\lambda_0\mu^s_{-i+1}(a) = Y_{-i}(a,p^s)$, if  $n_0+3\leq i\leq n-1$.\\\

\item $\nu^s_{-n_0-2}(a) = Y_{-n_0-2}(a,p^s)$.\\\

\item $\nu^s_{-n_0-1}(a) = Y_{-n_0-1}(a,p^s) - 2\lambda_0\mu^s_{-n_0-1}(a)$.\\\

\item $\nu^s_{-i}(a) = 0$, for $i\geq n$ and $\mu^s_{-j}(a) =0$ if 
$j\not\in \{n_0+1, n_0, n_0-1\}$.
\end{enumerate}
\end{lemma}
\begin{proof} We fix $1-n\leq a<q = p^s$. 

Note that the assertion~(6) is just the assertions~(1) and (2) of 
Lemma~\ref{l1}.

Now, by Lemma~\ref{l1}
\begin{multline*}
F^s_*(\sO(a))  =  \sO(-n+1)^{\nu^s_{-n+1}(a)}\oplus\cdots\oplus
\sO(-1)^{\nu^s_{-1}(a)}\oplus \sO^{\nu^s_0(a)}\\
 \oplus \sS(-n_0+1)^{\mu^s_{-n_0+1}(a)}\oplus 
\sS(-n_0)^{\mu^s_{-n_0}(a)}\oplus \sS(-n_0-1)^{\mu^s_{-n_0-1}(a)}.\end{multline*}

Tensoring the above equation by $\sO(i)$ and by the projection formula, we get

\begin{multline*}
F^s_*(\sO(a+iq))  =  \sO(i-n+1)^{\nu^s_{-n+1}(a)}\oplus\cdots\oplus
\sO(i-1)^{\nu^s_{-1}(a)}\oplus \sO(i)^{\nu^s_{0}(a)}\\\
 \oplus \sS(i-n_0+1)^{\mu^s_{-n_0+1}(a)}\oplus
\sS(i-n_0)^{\mu^s_{-n_0}(a)}\oplus 
\sS(i-n_0-1)^{\mu^s_{-n_0-1}(a)}.\end{multline*}

By (\ref{ee}), $h^0(Q, \sO(m))= 0$ for $m \leq -1$ and $h^0(Q, \sS(m)) = 0$ for 
$m\leq 0$.
Also $h^0(Q, \sO(m)) = L_m$ for all $m\geq 1-n$. 
 Hence applying  the functor $H^0(Q, -)$ to the above decomposition  we get
\begin{equation}\label{**}
\begin{split}
 \nu^s_0(a)  = &  L_a = Z_0(a,p^s)\quad\mbox{and}\\
\nu^s_{-1}(a) = &   L_{a+q} - L_1L_a = Z_{-1}(a,p^s).\\
\end{split}
\end{equation}
In general, for $i\leq n_0-1$, 
$$L_{a+iq} = L_0\nu^s_{-i}(a) + L_1\nu^s_{-i+1}(a) + \cdots + L_i\nu^s_0(a)
$$
which implies
$$\nu^s_{-i}(a) =  \nu^s_{-i}(a)+2\lambda_0\mu^s_{-i+1}(a)  = Z_{-i}(a,p^s). $$

Let $i= n_0$, Since $h^0(Q, \sS(1)) = 2\lambda_0$ we have
$$\begin{array}{lcl}
L_{a+n_0q} & = & L_0\nu^s_{-n_0}(a) + L_1\nu^s_{-n_0+1}(a) + \cdots + L_{n_0}\nu^s_0(a)
+ 2\lambda_0\mu^s_{-n_0+1}(a)\\\\
&  = &  \nu^s_{-n_0}(a) + L_1Z_{-n_0+1}(a,p^s) + \cdots + L_{n_0}Z_0(a,p^s)
+ 2\lambda_0\mu^s_{-n_0+1}(a)\end{array}$$
which implies 

\begin{equation}\label{n_0}
\nu^s_{-n_0}(a)+2\lambda_0\mu^s_{-n_0+1}(a) = Z_{-n_0}(a,p^s).\end{equation}
This proves assertion~(1).

Let $i= n_0+1$. Since $h^0(Q, \sS(2)) = 2\lambda_0(L_1-L_0)$, we have

\begin{multline*} 
L_{a+(n_0+1)q} = 
L_0\nu^s_{-n_0-1}(a) + L_1\nu^s_{-n_0}(a) + \cdots + L_{n_0+1}\nu^s_0(a)\\
+2\lambda_0\left[(L_1-L_0)\mu^s_{-n_0+1}(a)+ \mu^s_{-n_0}(a)\right]\\
= (\nu^s_{-n_0-1}(a)+2\lambda_0\mu^s_{-n_0}(a))
 + \left[L_1(\nu^s_{-n_0}(a)+2\lambda_0\mu^s_{-n_0-1}(a))\right.\\
\left.+ L_2\nu^s_{-n_0+1}(a)+ \cdots +
L_{n_0+1}\nu^s_0(a)\right]
-2\lambda_0\mu^s_{-n_0+1}(a)
\end{multline*}
which implies, by assertion~(1) of the lemma,

\begin{equation}\label{n_0+1}\nu^s_{-n_0-1}(a)+2\lambda_0\mu^s_{-n_0}(a) = Z_{-n_0-1}(a,p^s) + 
2\lambda_0\mu^s_{-n_0+1}. \end{equation}
This proves assertion~(2).

Now if we  tensor the decomposition of $F^s_*(\sO(a))$ by $\sO(-j)$, 
where $1\leq j \leq n-n_0-1$, then we get
\begin{multline*}
F^s_*(\sO(a-jq))  =  \sO(-j-n+1)^{\nu^s_{-n+1}(a)}\oplus\cdots\oplus
\sO(-j-1)^{\nu^s_{-1}(a)}\oplus \sO(-j)^{\nu^s_0(a)}\\
 \oplus \sS(-j-n_0+1)^{\mu^s_{-n_0+1}(a)}\oplus 
\sS(-j-n_0)^{\mu^s_{-n_0}(a)}\oplus \sS(-j-n_0-1)^{\mu^s_{-n_0-1}(a)}.\end{multline*}

Now applying  the functor $H^n(Q_n, -)$ and then using the Serre duality 
$$h^n(Q, \sO_Q(m)) = h^0(Q, \sO_{Q}(-m-n)) = L_{-m-n}$$
and $h^n(Q, \sS(m)) = h^0(Q, \sS(1-m-n))$ (here $jq-a-n \geq 1-n$) we get

\begin{multline*}L_{jq-a-n}   =   L_{j-1}\nu_{-n+1}^s(a)+ L_{j-2}\nu_{-n+2}^s(a)+
\cdots +
L_0{\nu^s}_{-n+j}(a)\\
+\mu^s_{-n_0+1}(a)h^0(Q_n, S(n_0+j-n))
  +\mu^s_{-n_0}(a)h^0(Q_n, S(n_0+j+1-n))\\+
\mu^s_{-n_0-1}(a)h^0(Q_n, S(n_0+j+2-n)).\end{multline*}

Hence   $1\leq j \leq n-n_0-2 $, we get 
$$L_{jq-a-n} = L_{j-1}\nu_{-n+1}^s(a)+ L_{j-2}\nu_{-n+2}^s(a)+\cdots +
L_0\nu_{-n+j}^s(a).$$

Therefore
\begin{multline*}
\nu^s_{-n+j}(a) = L_{jq-a-n} -\left[L_{j-1}\nu^s_{-n+1}(a) + \cdots +
L_2\nu^s_{-n+j-2}(a)+L_1\nu^s_{-n+j-1}(a)\right]\\ 
= Y_{-n+j}(a,p^s).\end{multline*}

\begin{equation}
\begin{split}
\mbox{For}\quad  j = 1,\quad \nu^s_{-n+1}(a)  = & L_{q-a-n},\\
\mbox{For}\quad  j = 2,\quad \nu^s_{-n+2}(a)  = & L_{2q-a-n} - L_1\nu^s_{-n+1}(a).
\end{split}\end{equation}
In general

$$\begin{array}{lcl}
\nu^s_{-i}(a)  & = & \nu^s_{-i}(a)+2\lambda_0\mu^s_{-i+1}(a) =  Y_{-i}(a,p^s),\quad{for}\quad  
n_0+3 \leq i \leq n-1\\\
\nu^s_{-n_0-2}(a) & = &  Y_{-n_0-2}(a,p^s).\end{array}$$
This proves assertions~(3) and (4).

\vspace{5pt}

For $j = n-n_0-1$ we get
$$L_{(n-n_0-1)q-a-n} = L_{n-n_0-2}\nu^s_{-n+1}(a)+ \cdots  + L_0\nu^s_{-n_0-1}(a) + 
2\lambda_0\mu^s_{-n_0-1}(a).$$

\begin{equation}\label{n_0+1}
\nu^s_{-n_0-1}(a) = Y_{-n_0-1}(a,p^s) - 2\lambda_0\mu^s_{-n_0-1}(a).\end{equation}
This proves assertion~(5).
\end{proof}

\begin{rmk}\label{r2} By Lemma~\ref{l2}  it follows that 
if  there is an interval $I_1\subset [0, 1)$ with the property that   $a/p^s \in I_1$
implies there is at the most
one type of spinor bundle in the decomposition of $F^s_*(\sO(a))$, then 
all the functions $\nu^s_{-i}(a)$ and $\mu^s_{-i+1}(a)$ have  polynomial 
expressions for $a/q\in I_1$. More precisely we have the following.
\end{rmk}

\begin{lemma}\label{l3}Let $0\leq a < q=p^s$, where $p\neq 2$ is a prime
such that $p>n-2$ and   $n\geq 3$ is an   integer. Let $n_0 = \lceil \tfrac{n}{2}\rceil -1$. Then
we have 
$$F_*^s(\sO(a)) = \bigoplus_{t=0}^{n-1}\sO(-t)^{\nu^s_{-t}(a)}\oplus 
\bigoplus_{t = n_0-1}^{n_0+1}\sS(-t)^{\mu^s_{-t}(a)}.$$

\vspace{5pt}

\noindent {\underline{Case}}~(1).\quad Let $n\geq 3$ be odd  then   

\begin{enumerate}\item[]
$$\begin{array}{lcl}
0\leq a/q < 1 & \implies &
\nu^s_{-i}(a)  = \begin{cases} Z_{-i}(a, p^s) & 
   \quad{for}\quad 0\leq i \leq  n_0-1\\\
  Y_{-i}(a, p^s) &  \quad{for}\quad n_0+1\leq i < n \end{cases}\\\
& & \mu^s_{-i}(a) = 0\quad{for}\quad i \neq n_0-1,  n_0,\end{array}$$

\item[]
$$\begin{array}{lcl}
 a/q\in [0, \frac{1}{2}-\frac{n-2}{2p}]\implies \begin{cases}\nu^s_{-n_0}(a) = Z_{-n_0}(a, p^s)\\
\mu^s_{-n_0+1}(a) = 0\\
\mu^s_{-n_0}(a) = \frac{1}{2\lambda_0}\left[Z_{-n_0-1}(a, p^s) - Y_{-n_0-1}(a, p^s)
\right],\end{cases}\end{array}$$

 \item[] 
$$\begin{array}{lcl}
a/q\in [\frac{1}{2}+\frac{n-2}{2p}, 1) & \implies & 
\begin{cases}
\nu^s_{-n_0}(a) = Z_{-n_0}(a, p^s)
-Y_{-n_0-1}(a, p^s) + Z_{-n_0-1}(a, p^s)\\
\mu^s_{-n_0+1}(a) = \frac{1}{2\lambda_0}\left[Y_{-n_0-1}(a, p^s) - Z_{-n_0-1}(a, p^s)\right]\\
\mu^s_{-n_0}(a)  = 0.
\end{cases}\end{array}$$\end{enumerate}

\noindent Case~(2).\quad If $n\geq 4$ is even  
then

\begin{enumerate}\item[]
$$\begin{array}{lcl}
0\leq a/q < 1 & \implies &
\nu^s_{-i}(a)  = \begin{cases} Z_{-i}(a, p^s) & 
   \quad{for}\quad 0\leq i \leq  n_0-1\\\
  Y_{-i}(a, p^s) &  \quad{for}\quad n_0+2\leq i < n \end{cases}\\\
& & \mu^s_{-i}(a) = 0\quad{for}\quad i \neq n_0-1,  n_0, n_0+1,\end{array}$$

\item[]
$$\begin{array}{lcl}
 a/q\in [\tfrac{n-2}{p},\quad  1-\tfrac{n-2}{p})
\implies \begin{cases}\nu^s_{-n_0}(a) = Z_{-n_0}(a, p^s)\\
\nu^s_{-n_0-1}(a) = Y_{-n_0-1}(a, p^s)\\
\mu^s_{-n_0+1}(a) = 0\\
\mu^s_{-n_0}(a) = \frac{1}{2\lambda_0}\bigl[Z_{-n_0-1}(a, p^s) - Y_{-n_0-1}(a, p^s)\bigr],\end{cases}\end{array}$$

\item[]
$$ a/q\in [\tfrac{n-2}{p},\quad  1)
\implies \mu^s_{-n_0-1}(a) = 0.$$
\end{enumerate}
\end{lemma}

\begin{proof}We prove the lemma for the case when  $n$ is odd, as the proof is 
similar when $n$ is even. 

Now by Lemma~\ref{l1}~(2)  we have $\mu^s_{-j}(a) = 0$ if $j\not\in 
\{n_0-1, n_0\}$. 
Therefore by Lemma~\ref{l2}~(1)  for $0\leq i \leq n_0-1$ 
$$\nu^s_{-i}(a) = \nu^s_{-i}(a)+2\lambda_0\mu^s_{-i+1}(a) = Z_{-i}(a,p^s).$$
and for  $n_0+2\leq i < n$, by Lemma~\ref{l2}~(3) and (4), 

$$\nu^s_{-i}(a)  = Y_{-i}(a,p^s).$$
If $i = n_0+1$ then by Lemma~\ref{l2}~(5)
$$\nu^s_{-n_0-1}(a) = Y_{-n_0-1}(a, p^s)-2\lambda_0\mu^s_{-n_0-1}(a) = Y_{-n_0-1}(a,p^s).$$

\vspace{5pt}

\noindent{\underline {Assertion}}~(1).\quad
If $a/q\in [0, \tfrac{1}{2}-\tfrac{n-2}{2p}]$ then 
$\mu_{-n_0+1}^s(a) = 0$, by Lemma~\ref{l1}~(2)(b).

This gives, by Lemma~\ref{l2}~(1)
$$\nu^s_{-n_0}(a) = \nu^s_{-n_0}(a)+\mu^s_{-n_0+1}(a) = Z_{-n_0}(a,p^s)$$
and  by Lemma~\ref{l2}~(2)
$$ 2\lambda_0\mu^s_{-n_0}(a) = Z_{-n_0-1}(a,p^s)- \nu^s_{-n_0-1}(a),$$
whereas by Lemma~\ref{l2}~(5),  $\nu^s_{-n_0-1}(a) = Y_{-n_0-1}(a, p^s)$.

\vspace{5pt}

\noindent{\underline {Assertion}}~(2).\quad
If  $a/q\in [\tfrac{1}{2}+\tfrac{n-2}{2p},~~ 1]$ then  $\mu^s_{-n_0}(a) = 0$, by Lemma~\ref{l1}~(2)(a).
By Lemma~\ref{l2}~(2) and (5), 
$$Y_{-n_0-1}(a, p^s) = 
\nu^s_{-n_0-1}(a) = 
Z_{-n_0-1}(a, p^s)+2\lambda_0\mu^s_{-n_0+1}(a)$$
which implies 
$\mu^s_{-n_0+1}(a)  =  \frac{1}{2\lambda_0}\left[L^s_{-n_0-1}(a) 
- Z^s_{-n_0-1}(a)\right]$
and
$$\begin{array}{lcl}
\nu^s_{-n_0}(a) & = & Z_{-n_0}(a,p^s)-2\lambda_0\mu^s_{-n_0+1}(a)\\\
& = & Z_{-n_0}(a, p^s) - Y_{-n_0-1}(a, p^s) + Z_{-n_0-1}(a, p^s).\end{array}$$
\end{proof}

Now using the fact that the HK density function $f_{R_{p, n+1}}$ is 
continuous and is the  limit function of the sequence of functions 
$\{f_s(R, I)\}_{s\in \N}$, we deduce that the rank functions are well defined 
and continuous.

\begin{lemma}\label{cont} 
The functions
 $\nnu_{i}^{(p)}:\R \longto \R$
and $\mmu_{i}^{(p)}:\R\longto \R$ 
given by 
$$\nnu_i^{(p)}(x) = \lim_{s\to \infty}
\frac{\nu_{-i}^s(\lfloor xq\rfloor)}{q^n}\quad\mbox{and}\quad
\mmu_i^{(p)}(x) = \lim_{s\to \infty}2\lambda_0\frac{\mu_{-i}^s(\lfloor 
xq\rfloor)}{q^n}$$

are  well defined continuous functions.

Moreover, for each $i$ we have  $0\leq \nnu_i^{(p)}(x)\leq 1$ and 
$0\leq \mmu_i^{(p)}(x)\leq  2$.

\end{lemma}
\begin{proof} By (\ref{ap}), it is enough to prove the assertion for 
all $\nnu_i^{(p)}$ and $\mmu_i^{(p)}$ restricted to the interval $[0, 1)$.

By Theorem~1.1 of [T1] (see Definition~\ref{dhkd}) 
 the sequence $\{f_s(R_{p, n+1}, {\bf m})\}_{s\in \N}$ 
converges uniformly to the HK density function $f_{R_{p, n+1}}$.
Therefore 
$$f_{R_{p, n+1}}(x) = \lim_{s\to \infty}f_s(R_{p, n+1}, {\bf m})(x)
 =  \lim_{s\to \infty}
\ell(R_{p, n+1}/{\bf m}^{[q]})_{\lfloor xq\rfloor}/q^n.$$

On the other hand 
if $x\in [0, 1)$ then  by Lemma~\ref{rk1} and Lemma~\ref{rk2}

\begin{equation}\label{conte}
\ell(R_{p, n+1}/{\bf m}^{[q]})_{\lfloor (x+i)q\rfloor}/q^n = 
\frac{1}{q^n}\left(\nu^s_{-i}
(\lfloor xq\rfloor) +
2\lambda_0 \mu^s_{-i+1}(\lfloor xq\rfloor)\right).\end{equation}

By Lemma~\ref{l3}, $\mu^s_{-i+1}(\lfloor xq\rfloor) = 0$ 
if $i\not\in \{n_0, n_0+1, n_0+2\}$. Hence 
$$ \nnu_i^{(p)}(x) = 
 \lim_{s\to \infty}\frac{\nu^s_{-i}
(\lfloor xq\rfloor)}{q^n} = f_{R_{p, n+1}}(x+i)$$ 
is a well defined continuous function.
We can argue for other rank functions as follows.
By Lemma~\ref{l2}~(4)
$$\nnu_{n_0+2}^{(p)}(x) = \lim_{s\to \infty}
\frac{\nu_{-n_0-2}^s(\lfloor xq\rfloor)}{q^n} = \lim_{s\to \infty}
\frac{Y_{-n_0-2}(\lfloor xq\rfloor, q)}{q^n}
={\bf Y}_{n_0+2}(x+n_0+2).$$
By (\ref{conte}) 
$$\frac{\mu^s_{-n_0-1}(\lfloor xq\rfloor)}{q^n} = 
\frac{\ell(R_{p, n+1}/{\bf m}^{[q]})_{\lfloor (x+n_0)q\rfloor}}{q^n} -
\frac{\nu^s_{-n_0-2}(\lfloor xq\rfloor)}{q^n}.$$
Since both the terms on the right hand side has limit as $s\to \infty$
we get 

\begin{equation}\label{n_0+2}
 \mmu_{n_0+1}^{(p)}(x) 
  =  f_{R_{p, n+1}}(x+n_0+2) - {\bf Y}_{n_0+2}(x+n_0+2)\end{equation}
is a well defined and continuous function.

Using 
similar arguments and Lemma~\ref{l2}, we can show that the functions 
$\nnu_{n_0+1}^{(p)}$, $\mmu_{n_0}^{(p)}$, $\mmu_{n_0-1}^{(p)}$ then 
$\nnu_{n_0}^{(p)}$ are well defined and continuous:

$$\nnu_{n_0+1}^{(p)}(x) = {\bf Y}_{n_0+1}(x+n_0+1) - 
\mmu_{n_0+1}^{(p)}(x),\quad\mbox{by Lemma~\ref{l2}~(5)},$$

$$\mmu_{n_0}^{(p)}(x) = f_{R_{p, n+1}}(x+n_0+1) -\nnu_{n_0+1}^{(p)}(x),
\quad\mbox{by}~(\ref{conte}) $$

\begin{equation}\label{n_0+1}{\begin{array}{lcl}
\mmu_{n_0-1}^{(p)}(x) & = & 
\nnu_{n_0+1}^{(p)}(x)+\mmu_{n_0}^{(p)}(x)-
{\bf Z}_{n_0+1}(x+n_0+1),\quad \mbox{by Lemma~\ref{l2}~(2)}\\
 & = & f_{R_{p, n+1}}(x+n_0+1) - {\bf Z}_{n_0+1}(x+n_0+1).\end{array}}\end{equation}
and 
$$\nnu_{n_0}^{(p)}(x) = f_{R_{p, n+1}}(x+n_0) - \mmu_{n_0-1}^{(p)}(x).$$ 

Hence all the ranks functions are well defined continuous functions.

To prove the second assertion of the lemma. we 
compute the ranks of the bundles on both the sides of (\ref{dce}) and (\ref{dco}), and 
use the fact that the spinor bundle $\sS$ has rank $\lambda_0$. This gives

$$p^{sn} = \nu_0^s(a) + \cdots + \nu_{-n+1}^s(a)+ \lambda_0\bigl[\mu_{-n_0+1}^s(a)
+\mu_{-n_0}^s(a)+\mu_{-n_0-1}^s(a)\bigr].$$

In particular, we have $0\leq  \nu^s_{-i}(a),~~ \lambda_0\mu^s_{-i}(a)  
\leq q^n = p^{sn}$. Therefore
$${\mmu}_{i}^{(p)}(x)= 2\lambda_0\lim_{s\to \infty}
{\mu^s_{-i}(\lfloor xq\rfloor)}/{q^n} \leq 2\quad\mbox{and}\quad
{\nnu}_{i}^{(p)}(x)= \lim_{s\to \infty}{\nu^s_{-i}(\lfloor 
xq\rfloor)}/{q^n} \leq 1.$$
\end{proof}

\begin{defn}\label{poly}We call $\nnu_i^{(p)}$ and $\mmu_i^{(p)}$ 
the {\em rank functions} of $Q_n$ or of $R_{p, n+1}$.
 
We say $\nnu_i^{(p)}$ 
 is {\it polynomial in the interval} 
$I_1\subset [0, 1)$ if there exists  
a polynomial 
$F_i(X)\in \Q[x]$  such that
$\nnu_i^{(p)}(x) = F_i(x)$,
for all $x\in I_1$. We can define the similar notion  for $\mmu^{(p)}_{i}$.
\end{defn}

In the following Proposition, which is one of the key steps in the paper, we 
describe the range where these 
functions are in fact  polynomial functions and hence are independent of $p$
though the range itself might depend on $p$.

\begin{propose}\label{rc} Let $p>2$ such that $p>n-2$. 
If $n\geq 4$ is even then for  $x\in [0, 1)$ and for $i\nin \{n_0-1, n_0, n_0+1\}$ 
the function  $\mmu_{i}^{(p)}(x)= 0$, and   
\begin{enumerate}
\item $\nnu_{i}^{(p)}(x)+\mmu_{i-1}^{(p)}(x) = {\bf Z}_i(x+i)$, if $x\in [0, 1)$ 
and $0\leq i \leq n_0$.
\item $\nnu_{n_0+1}^{(p)}(x)+\mmu_{n_0}^{(p)}(x) = {\bf Z}_{n_0+1}(x+n_0+1)$, 
if $x\in [0, 1-\tfrac{n-2}{2p})$.
\item $\nnu^{(p)}_{n_0+2}(x) = {\bf Y}_{n_0+2}(x+n_0+2)$, if $x\in [0, 1)$.
\item  $\nnu_{n_0+2}^{(p)}(x)+\mmu_{n_0+1}^{(p)}(x) = {\bf Y}_{n_0+2}(x+n_0+2)$, 
if $x\in [\tfrac{n-2}{2p}, 1))$.
\item $\nnu_{i}^{(p)}(x)+\mmu_{i-1}^{(p)}(x) = 
\nnu_{i}^{(p)}(x) = 
{\bf Y}_i(x+i)$, if $x\in [0, 1)$ and 
$n_0+3\leq i < n$.
\end{enumerate}

If $n\geq 3$ is odd then  for  $x\in [0, 1)$ and for $i\nin \{n_0-1, n_0\}$ 
the function  $\mmu_{i}^{(p)}(x)= 0$, and
\begin{enumerate}
\item $\nnu_{i}^{(p)}(x)+\mmu_{i-1}^{(p)}(x) = {\bf Z}_i(x+i)$, if $x\in [0, 1)$ and $0\leq i \leq  
n_0$.
\item $\nnu_{n_0+1}^{(p)}(x)+\mmu_{n_0}^{(p)}(x) = {\bf Z}_{n_0+1}(x+n_0+1)$, 
if $x\in [0, \frac{1}{2}-\tfrac{n-2}{2p})$.
\item $\nnu_{n_0+1}^{(p)}(x)+\mmu_{n_0}^{(p)}(x) = {\bf Y}_{n_0+1}(x+n_0+1)$,
if $x\in [\tfrac{1}{2}+\tfrac{n-2}{2p}, 1))$.
\item $\nnu^{(p)}_{n_0+1}(x) = {\bf Y}_{n_0+1}(x+n_0+1)$, if  $x\in [0, 1)$.
\item $\nnu^{(p)}_{n_0+2}(x) = {\bf Y}_{n_0+2}(x+n_0+2)$, if  $x\in [0, 1)$.
\item $\nnu_{i}^{(p)}(x)+\mmu_{i-1}^{(p)}(x) = \nnu_{i}^{(p)}(x)
= {\bf Y}_i(x+i)$, if $x\in [0, 1)$ and 
$n_0+2\leq i < n$.
\end{enumerate}

Moreover for both even and odd  $n$,
$\nnu_{i}^{(p)}(x) =\mmu_{i-1}^{(p)}(x) = 0$, for all $x\in [0, 1)$ and $i\geq n$. 
\end{propose}
\begin{proof} For both parities of $n$ and for $i=n_0$
 the assertion~(1) follows 
from 
(\ref{n_0}).  
Lemma~\ref{l3} gives the vanishing of $\mu^s_{-n_0+1}(a)$ for $a$ 
in the certain range. This applied to  (\ref{n_0+1}) gives assertion~(2).
The rest of the assertions  follow from Lemma~\ref{l3}, Lemma~\ref{l2} along 
with Remark~\ref{r3}.\end{proof}

\vspace{10pt}

\section{The HK density function $f_{R_{p, n+1}}$ and $f_{R^{\infty}_{n+1}}$}

Now we are ready to give the first formulation of the HK density 
function $f_{R_{p, n+1}}$ provided $p > n-2$.

We follow the Notations~\ref{n1}. By Lemma~\ref{l1}, for 
 $q=p^s$ and an integer $0\leq a <q $,  there exist 
integers  $\nu^s_t(a)$ and $\mu^s_t(a)$  
 given by   the decomposition 
$$ F_*^s(\sO(a)) = \oplus _{t=0}^{n-1}\sO(-t)^{\nu^s_{-t}(a)} \oplus 
\oplus_{t=-n_0+1}^{-n_0-1}
\sS(t)^{\mu^s_t(a)}.$$

\vspace{5pt}

\noindent{\underline{\large{Proof of Proposition}}~\ref{l4}}.\quad 
If $x\in [i, i+1)$ then  by Lemma~\ref{rk2}
$$f_{R_{p, n+1}}(x)  = 
\nnu_{i}^{(p)}(x-i)+\mmu_{i-1}^{(p)}(x-i).$$

\vspace{5pt}

\noindent{\underline{Case}}~(A).\quad
Suppose $n\geq 4$ is an even number.

\vspace{5pt}
\noindent{\underline{Case}}~(a).\quad Let $i\in \{0, 1, \ldots, n_0\}$. Then 
by Proposition~\ref{rc}
$$f_{R_{p, n+1}}(x) = 
\nnu_{i}^{(p)}(x-i)+\mmu_{i-1}^{(p)}(x-i) = {\bf Z}_{i}(x).$$

\vspace{5pt}

\noindent{\underline{Case}}~(b).\quad Let $i = n_0+1$ and 

\noindent{\underline {Sub case}}~(b)~(1).\quad
If  $x\in [n_0+1, n_0+2-\tfrac{n-2}{2p})$, then 
by Proposition~\ref{rc}~(2)
$$f_{R_{p, n+1}}(x) = \nnu_{n_0+1}^{(p)}(x-n_0-1)+\mmu_{n_0}^{(p)}(x-n_0-1) = 
{\bf Z}_{n_0+1}(x).$$

\vspace{5pt}

\noindent{\underline {Sub case}}~(b)~(2).\quad If $x\in 
[n_0+2-\tfrac{n-2}{2p}, n_0+2)$, 
then, by (\ref{n_0+1}),  
$$f_{R_{p, n+1}}(x) = {\bf Z}_{n_0+1}(x)+\mmu_{n_0-1}^{(p)}(x-n_0-1).$$
\vspace{5pt}

\noindent{\underline{Case}}~(c).\quad Let $i = n_0+2$ then by (\ref{n_0+2})
 $$f_{R_{p, n+1}}(x) = {\bf Y}_{n_0+2}(x)+ \mmu_{n_0+1}^{(p)}(x-n_0-2).$$

\vspace{5pt}

If in addition $x\in [n_0+2+\tfrac{n-2}{2p}, n_0+3)$
then the last assertion of Lemma~\ref{l3}
implies $\mmu_{n_0+1}^{(p)}(x-n_0-2) = 0$.

\noindent{\underline{Case}}~(d).\quad Let $i\in \{n_0+3, \cdots, n-1\}$. Then 
by Proposition~\ref{rc} 
$$f_{R_{p, n+1}}(x) = \nnu_{i}^{(p)}(x-i)+\mmu_{i-1}^{(p)}(x-i) = 
{\bf Y}_{i}(x) \quad\mbox{for}\quad x\in [i, i+1).$$

\vspace{5pt}

\noindent{\underline{Case}}~(e).\quad Let $i \geq n$. Then by Lemma~\ref{l2}~(6), 
$$f_{R_{p, n+1}}(x)  =   \nnu_{i}^{(p)}(x-i)+\mmu_{i-1}^{(p)}(x-i) = 0.$$
This completes the proof of the theorem when $n$ is an even number.

\vspace{5pt}

Assume $n\geq 3$ is odd. Let $x\in [i, i+1)$
Then by 
Proposition~\ref{rc}~(1) and (6)
$$\begin{array}{lcl}
f_{R_{p,n+1}}(x)  & = & 
{\bf Z}_{i}(x),\quad \mbox{if}\quad i\leq n_0\\
& = & 
{\bf Y}_{i}(x),\quad \mbox{if}\quad i\geq n_0+2.
\end{array}$$

If $i = n_0+1$ then 
$$f_{R_{p,n+1}}(x)  =   \nnu^{(p)}_{n_0+1}(x-n_0-1)+
\mmu^{(p)}_{n_0}(x-n_0-1),\quad \mbox{if}\quad x-n_0-1\in [0, 1).$$
Now applying 
Proposition~\ref{rc}~(2), (4)  and (3) we get
$$\begin{array}{lcl}
f_{R_{p,n+1}}(x)  
& = & {\bf Z}_{n_0+1}(x),\quad \mbox{if}\quad x-n_0-1\in [0, 
\tfrac{1}{2}-\tfrac{n-2}{2p})\\
& = & {\bf Y}_{n_0+1}(x)+\mmu^{(p)}_{n_0}(x-n_0-1),\quad \mbox{if}\quad 
x-n_0-1\in [\tfrac{1}{2}-\tfrac{n-2}{2p},\quad \tfrac{1}{2}+\tfrac{n-2}{2p})\\
& = & {\bf Y}_{n_0+1}(x),\quad \mbox{if}\quad 
x-n_0-1\in [\tfrac{1}{2}+\tfrac{n-2}{2p},\quad 1).
\end{array}$$
This completes the proof of the proposition.$\Box$

\vspace{10pt}

\begin{thm}\label{f0} The function 
  $f^{\infty}_{R_{n+1}}:[0, \infty)\longto [0,\infty)$ 
given by 
$$f^{\infty}_{R_{n+1}}(x) := \lim_{p\to \infty}f_{R_{p, n+1}}(x)$$
is partially symmetric continuous function, that is 
$$f^{\infty}_{R_{n+1}}(x) =  f^{\infty}_{R_{n+1}}(n-x),\quad\mbox{for}\quad 
0\leq x\leq (n-2)/2$$
and 
is described as follows:

\begin{enumerate}
\item If $n\geq 4$ is even then 
  $$f^{\infty}_{R_{n+1}}(x)  = \begin{cases}  {\bf Z}_{i}(x)& \quad\mbox{if}\quad 
 i\leq x < i+1 \quad\mbox{and}\quad 0\leq i \leq  n_0+1\\\\
  {\bf Y}_{i}(x) &\quad\mbox{if}\quad  
i\leq x < i+1 \quad\mbox{and}\quad n_0+2\leq i <n
\end{cases}$$
and $f^{\infty}_{R_{n+1}}(x) = 0$ otherwise.

\item If $n\geq 3$ is  an odd number then 
$$f^{\infty}_{R_{n+1}}(x)  = \begin{cases} 
  {\bf Z}_{i}(x)& \quad\mbox{if}\quad 
 i\leq x < i+1 \quad\mbox{and}\quad 0\leq i \leq  n_0\\\\
 {\bf  Z}_{n_0+1}(x) &\quad\mbox{if}\quad  
(n_0+1) \leq x < (n_0+\frac{3}{2})\\\\
{\bf Y}_{n_0+1}(x) & \quad\mbox{if}\quad  
(n_0+\frac{3}{2})\leq x < (n_0+2)\\\\
 {\bf Y}_{i}(x) &\quad\mbox{if}\quad  
i\leq x < i+1 \quad\mbox{and}\quad n_0+2\leq i <n
\end{cases}$$
and $f^{\infty}_{R_{n+1}}(x) = 0$ otherwise.
\end{enumerate}
\end{thm}

\begin{proof}The description of the function $f^{\infty}_{R_{n+1}}:[0, \infty)\longto 
[0, \infty)$ follows from  Proposition~\ref{l4}.
To prove the symmetry, we consider the 
${\bf Z}_{j}:[0, \infty)\longto [0, \infty)$ and 
${\bf Y}_{j}:[0, \infty)\longto [0\, \infty)$ 
and 
\vspace{5pt} 

\noindent{\bf Claim}.\quad ${\bf Z}_{j}(x) = {\bf Y}_{n-1-j}(n-x)$, if  
$j\leq x < j+1$, where $j \leq n_0-1$.

\vspace{5pt}

\noindent{\underline {Proof of the claim}}:\quad  By induction on $j\geq 0$, first 
we prove the assertion  that 
 $$\lim_{q\to \infty}Z_{-j}(a,q)/q^n = 
\lim_{q\to \infty}Y_{-(n-1-j,q)}(q-a,q)/q^n\quad\mbox{for}\quad 0\leq a <q.$$

If $j=0$ then 
$$\lim_{q\to \infty}Z_{0}(a, q)/q^n = \lim_{q\to \infty}L_a/q^n = 
\lim_{q\to \infty}L_{a+n}/q^n = 
\lim_{q\to \infty}Y_{-(n-1)}(q-a, q)/q^n.$$ 

Assume that the assertion holds for all $j$ where  $0\leq j <i \leq n_0-1$.
Now 
\begin{multline*}
\lim_{q\to \infty} \frac{Z_{-i}(a,q)}{q^n}  =  \lim_{q\to \infty}{L_{a}}/{q^n}-
{\left[L_1Z_{-i+1}(a,q)+\cdots+L_iZ_0(a,q)\right]}/q^n\\\
 = \lim_{q\to \infty}{L_{a+n}}/{q^n}-
\left[L_1Y_{-(n-i)}(q-a,q)+\cdots+L_iY_{-(n-1)}(q-a,q)\right]/q^n\\\
 =  \lim_{q\to \infty} {Y_{-(n-1-i)}(q-a,q)}/{q^n},\end{multline*}
where the second equality follows from the induction hypothesis.

Now to prove the claim, it is enough to prove for $x = m/q$, where $m\in \Z_{\geq 0}$.
If  $j\leq x < j+1$ then $m = a+jq$, where $0\leq a<q$. Now 
\begin{multline*}
{\bf Z}_{j}(x)  =  \lim_{q\to \infty}Z_{-j}(m -jq,q)/q^n
 =  \lim_{q\to \infty}Y_{-(n-1-j)}((j+1)q-m),q)/q^n\\\
 =  \lim_{q\to \infty}Y_{-(n-1-j)}((n-m)q-(nq-q-jq),q)/q^n
 =  {\bf Y}_{n-1-j}(n-x).\end{multline*}
This proves the claim.

\vspace{5pt}

\noindent{If} $n$ is even then $n_0 = (n/2)-1$.
Let  $0\leq x < (n-2)/2= n_0$ then $i \leq x <(i+1)$ for some $0\leq i < n_0 $.
Now 
$$f^{\infty}_{R_{n+1}}(x) = {\bf Z}_{i}(x) = {\bf Y}_{n-1-i}(n-x) = 
f^{\infty}_{R_{n+1}}(n-x).$$
 where the second equality follows as  
$n-(i+1) < n-x\leq n-i$.

\vspace{5pt} 

\noindent{If} $n$ is odd then $n_0 = (n-1)/2$.
Let $0\leq x <(n-2)/2= n_0-(1/2)$.

If $i\leq x <(i+1)$, where $i<n_0$ then 
$$f^{\infty}_{R_{n+1}}(x) = {\bf Z}_{i}(x) = {\bf Y}_{n-1-i}(n-x) = 
f^{\infty}_{R_{n+1}}(n-x).$$
 If $(n_0-1)\leq x < n_0-1/2$ then again 
$$f^{\infty}_{R_{n+1}}(x) = {\bf Z}_{n_0-1}(x) = {\bf Y}_{n_0+1}(n-x) = 
f^{\infty}_{R_{n+1}}(n-x).$$
\end{proof}

\begin{rmk}The same argument as above proves that $f_{R_{p, n+1}}$ is partially 
symmetric and the symmetry is given by 
$$f_{R_{p, n+1}}(x) = f_{R_{p, n+1}}(n-x)\quad \mbox{for}\quad 0\leq x \leq 
\tfrac{n-2}{2}\left(1-\tfrac{1}{p}\right).$$ 
\end{rmk}

\vspace{5pt}

\noindent{\underline {Proof of Corollary~\ref{t2}}}.\quad
If $n$ is even then $n_0 = \frac{n-2}{2}$ and if $n$ is odd then $n_0 = \frac{n-1}{2}$.
Therefore we can express the intervals as follows

$$n\quad\mbox{is even}\implies \mbox{$\left[n_0+2 -\frac{n-2}{2p},~~ n_0+2 +\frac{n-2}{2p}\right) = 
\left[\frac{n+2}{2}-\frac{n-2}{2p},~~\frac{n+2}{2}+\frac{n-2}{2p}\right)$}$$
$$n\quad\mbox{is odd}\implies
\mbox{$\left[n_0+\frac{3}{2} -\frac{n-2}{2p},~~ n_0+\frac{3}{2} 
+\frac{n-2}{2p}\right) = 
\left[\frac{n+2}{2}-\frac{n-2}{2p},~~\frac{n+2}{2}+\frac{n-2}{2p}\right)$}.$$

Now the corollary follows from 
 Proposition~\ref{l4} and Theorem~\ref{f0}. $\Box$

\vspace{10pt}

\begin{thm}\label{t1}Let $p\neq 2$ and  let $p > n-2$. Then 
$$1+m_{n+1} + \tfrac{2n-4}{p} 
\geq e_{HK}(R_{p, n+1}, {\bf m})\geq   1+m_{n+1}.$$

In fact, 
\begin{enumerate}\item if $n\geq 3$ odd and  $p\geq 3n-4$ or 
\item if $n\geq 4$ even and $p\geq (3n-4)/2$ 
\end{enumerate}
then $e_{HK}(R_{p, n+1}, {\bf m}) > 1+m_{n+1}$.
\end{thm}

\begin{proof}We prove the theorem for the case when $n$ is odd, as the proof for the case when 
$n$ is even  is very similar.

By Proposition~\ref{l4} and Theorem~\ref{f0}

$$0  \leq \int_0^\infty f_{R_{p, n+1}}(x)dx - \int_0^\infty 
f_{R^{\infty}_ {n+1}}(x)dx =  
\int_{\tfrac{1}{2}-\tfrac{n-2}{2p}}^{\tfrac{1}{2}+
\tfrac{n-2}{2p}}\mmu_{n_0}^{(p)}(x))dx   \leq \frac{2n-4}{p},$$
where the last inequality follows from Lemma~\ref{cont}.
On the other hand by Theorem~1.1 of [T1] we have 
$$ e_{HK}({R_{p, n+1}}, {\bf m}) =   \int_0^{\infty} f_{{R_{p, n+1}}, {\bf m}}(x)dx$$
and by
 the result of Gessel-Monsky [GM],
$\lim_{p\to \infty}e_{HK}({R_{p, n+1}}, {\bf m}) = 1+m_{n+1}$.
Hence by dominated convergence theorem
$$ \int_0^{\infty} f_{R^{\infty}_{n+1}, {\bf m}}(x)dx =   
\lim_{p\to \infty}\int_0^{\infty} f_{{R_{p, n+1}}, {\bf m}}(x)dx = 1+m_{n+1}.$$

The strict inequality assertion follows from Remark~\ref{strict}.
\end{proof}

\begin{cor}\label{c1}Let $p>2$ be a prime number such that  $p > n-2$
then the $F$-threshold of the ring
$R_{p ,n+1}$ is $c^{\bf m}({\bf m}) = n$.
\end{cor}
\begin{proof}  By Theorem~E of [TW1],  
the $F$-threshold 
$c^{\bf m}({\bf m}) = 
\mbox{max}~\{x\mid f_{R_{p, n+1}}(x)\neq 0\}$.

Now, by Proposition~\ref{l4}, $f_{R_{p, n+1}}(x) = 0$, for $x\geq n$ and for 
$n-1\leq x < n$, 
$$f_{R_{p, n+1}}(x) = {\bf Y}_{n-1}(x) = 
\lim_{q\to \infty}\frac{Y_{-n+1}(\lfloor xq\rfloor-(n-1)q,q)}{q^n} = 
\frac{2(n-x)^n}{n!}\neq 0,$$
where the last equality follows as $Y_{-n+1}(a,p^s) = L_{q-a-n}$. This proves the corollary.
\end{proof}

\section{Generating  polynomials for $f_{R_{p, n+1}}$}
In this and the next section we setup the ground to give second formulation of the HK 
density function $f_{R_{p, n+1}}$.

Throughout the section we denote
 $n_0 = \lceil\tfrac{n}{2}\rceil -1$ and, unless stated otherwise, assume that 
the characteristic $p\geq 3n-4$ if $n$ is odd and $p\geq (3n-4)/2$ if $n$ is even.

By  Proposition~\ref{rc} we know  that the every function  
$\nnu_i^{(p)}+\mmu_{i-1}^{(p)}:[0, 1)\longto [0, \infty)$  is a polynomial outside the 
difficult range which is 
\begin{enumerate}
\item  $[\frac{1}{2}-\frac{n-2}{2p}, \quad \frac{1}{2}+\frac{n-2}{2p})$, 
if $n\geq 3$ is an odd integer
and 
\item  $[0, \quad \frac{n-2}{2p})\cup [1-\frac{n-2}{2p},~~~1)$,~~~ 
if $n\geq 4$ is an even  integer.
\end{enumerate}

In 
Proposition~\ref{lp} we consider the same property for
$\nnu_i^{(p)}$ and $\mmu_{i-1}^{(p)}$ instead of 
$\nnu_i^{(p)}+\mmu_{i-1}^{(p)}$.

The formulation of $f_{R_{p, n+1}}$ in Proposition~\ref{l4} suggest 
that we only need to analyze  the rank function $\mmu_{n_0}^{(p)}$ when $n$ is odd, and
the rank functions $\mmu_{n_0-1}^{(p)}$ and $\mmu_{n_0+1}^{(p)}$
when $n$ is even. But the statements look less technical if we consider all 
the rank functions together in a tuple.

\subsection{Polynomials for the rank functions  $\nnu^{(p)}_{i}$ 
and $\mmu^{(p)}_{j}$}

\begin{notations}\label{nl}For the sake of uniformity in the indexing, henceforth  
we would write the decomposition of $F_*^s(\sO(a))$ as follows: 
For $q =p^s$  
$$F_*^s(\sO(a)) = \bigoplus_{i=0}^{n-1}
\sO(-i)^{l_{i}(a, q)} 
\oplus 
\sS(-n_0+1)^{l_n(a,q)}\oplus \sS(-n_0)^{l_{n+1}(a,q)}
\oplus \sS(-n_0-1)^{l_{n+2}(a,q)}.$$

We call the tuple $(l_0(a, q), \ldots, l_{n+2}(a, q))$
{\em the rank tuple} of $F_*^s(\sO(a))$ and 
now onwards
we relabel the rank functions
 $\nnu_0^{(p)}$, $\ldots$, $\nnu_{n-1}^{(p)}$, $\mmu_{n_0-1}^{(p)}$,
$\mmu_{n_0}^{(p)}$, $\mmu_{n_0+1}^{(p)}$ as
the functions $\eell_0^{(p)}$, $\ldots, \eell_{n-1}^{(p)}$,
$\eell_n^{(p)}$,$\eell_{n+1}^{(p)}$,
$\eell_{n+2}^{(p)}$ respectively.
\end{notations}

Following set $\{{\bf l}_i(x)\}_i\cup \{{\bf r}_i(x)\}_i \subset \Q[x]$ of polynomials 
is going to be the first 
 generating set of polynomials for the function $f_{R_{p, n+1}}$ when $n$ is odd. When 
$n$ is even the set will be $\{{\bf m}_i(x)\}_i\subset \Q[x]$.
The sets $\{F_i(t)\}_i, \{G_i(t)\}_i, \{{\tilde F}_i(t)\}_i\subset \Q[t]$ come from 
integrating these polynomials
which will be used for the formulation of $e_{HK}(R_{p, n+1})$, but will list them here.

\begin{defn}\label{5.4}Let ${\bf Z}_{i}(x)$ and ${\bf Y}_{i}(x)$ be as in Remark~\ref{r3}.
We recall that $n_0 = \tfrac{n-1}{2}$ if $n$ is odd, and 
$n_0 = \tfrac{n-2}{2}$ if $n$ is even.
\begin{enumerate}
\item Let  $n\geq 3$ be an odd integer. Then we  define 
$${\bf l}_{i}(x) = \begin{cases}
{\bf Z}_{i}(x+i) &\mbox{if}\quad 0\leq i \leq n_0\\
{\bf Y}_{i}(x+i) &\mbox{if}\quad n_0+1\leq i <n\\
0 &\mbox{if}\quad i = n\\
\frac{1}{2\lambda_0}\left[{\bf Z}_{n_0+1}(x+n_0+1)-{\bf Y}_{n_0+1}
(x+n_0+1)\right] & \mbox{if}\quad i = n+1\end{cases}.$$

and 
$${\bf r}_{i}(x) = \begin{cases}
{\bf Z}_{i}(x+i) &\mbox{if}\quad 0\leq i <n_0\\
{\bf Z}_{n_0}(x+n_0) -
\left[{\bf Y}_{n_0+1}(x+n_0+1)-{\bf Z}_{n_0+1}(x+n_0+1)\right] & \mbox{if}\quad i = n_0\\
{\bf Y}_{i}(x+i) &\mbox{if}\quad n_0+1\leq i <n\\
\frac{1}{2\lambda_0}\left[{\bf Y}_{n_0+1}(x+n_0+1)-{\bf Z}_{n_0+1}(x+n_0+1)\right] & \mbox{if}
\quad i = n\\
0 &\mbox{if}\quad i = n+1. \end{cases}$$

Moreover we consider   another  two sets of polynomials
\begin{equation}\label{gf}\{F_0(t), F_1(t), \ldots, F_{n+1}(t)\}\cup
\{G_0(t), G_1(t), \ldots, G_{n+1}(t)\}\subset \Q[t]\end{equation}
given by  

$$\int_0^{\tfrac{1}{2}-\tfrac{n-2}{2}t}{\bf l}_{i}(x)dx = 
F_i(t), \quad\mbox{and}\quad 
\int_{\tfrac{1}{2}+\tfrac{n-2}{2}t}^1{\bf r}_{i}(x)dx = 
G_i(t).$$

\item Let  
$n\geq 4$ be an even integer, we define  polynomials in $\Q[x]$ as follows.
 
$${\bf m}_i(x) = \begin{cases}
{\bf Z}_{i}(x+i) &\mbox{if}\quad 0\leq i \leq n_0\\
{\bf Y}_{i}(x+i) &\mbox{if}\quad n_0+1\leq i <n\\
0 &\mbox{if}\quad i = n\\
\frac{1}{2\lambda_0}\left[{\bf Z}_{n_0+1}(x+n_0+1)-{\bf Y}_{n_0+1}(x+n_0+1)\right] & \mbox{if}
\quad i = n+1\\
0 &\mbox{if}\quad i = n+2.
\end{cases}$$

Also another   set of polynomials
\begin{equation}\label{gfe}
\{{\tilde F}_0(t), {\tilde F}_1(t), \ldots, {\tilde F}_{n+2}(t)\}\subset \Q[t]\end{equation}
given by

$$\int_{\tfrac{n-2}{2}t}^{\tfrac{1}{2}-\tfrac{n-2}{2}t}{{\bf m}}_{i}(x)dx = 
{\tilde F}_i(t).$$
\end{enumerate}

Note that when $n$ is odd then ${\bf r}^{(p)}_{n+2} \equiv 0$ as 
$\mu_{-n_0-1}^{(p)}(a) = 0$ for $0\leq a <q$.
\end{defn}

The proof of the following proposition is just the corollary of 
Lemma~\ref{l3}  and Lemma~\ref{cont}.

\begin{propose}\label{lp}Let $p\neq 2$ such that 
$p > n-2$. 
\begin{enumerate}
\item Let $n\geq 3$ be an  odd integer
then each rank function ${\bf r}_i^{(p)}:[0, 1]\longto [0, 2)$ is a polynomial 
function on the 
interval $[0, \tfrac{1}{2}-\tfrac{n-2}{2p})\cup [\tfrac{1}{2}+\tfrac{n-2}{2p}, 1)$ 
and is given 
by 
$$\eell_i^{(p)}(x) = \begin{cases}{\bf l}_{i}(x),\quad\mbox{ 
for all}\quad x\in [0, \tfrac{1}{2}-\tfrac{n-2}{2p})\\
{\bf r}_{i}(x),\quad\mbox{ for all}\quad 
x\in [\tfrac{1}{2}+\tfrac{n-2}{2p},~~~~1).\end{cases}$$

Moreover  
$$\int_0^{\tfrac{1}{2}-\tfrac{n-2}{2}t}\eell^{(p)}_{i}(x)dx = 
F_i(t)\quad\mbox{and}\quad
\int_{\tfrac{1}{2}+\tfrac{n-2}{2}t}^1{\eell}^{(p)}_{i}(x)dx = 
G_i(t).$$

In particular all  $F_i(1/{p})$ and $G_i(1/{p})$ are nonnegative for $p>n-2$.

\item Let $n\geq 4$ be even integer then 
each rank function $\eell_i^{(p)}:[0, 1)\longto [0, 2)$ is a polynomial function on the 
interval $[\tfrac{n-2}{2p}, \quad 1-\tfrac{n-2}{2p})$ and is given by  
$${\eell}_i^{(p)}(x) = {{\bf m}}_i(x),\quad\mbox{ for all}
\quad x\in [\tfrac{n-2}{2p}, \quad 1-\tfrac{n-2}{2p}).$$

Moreover  
$$ \int_{\tfrac{n-2}{2}t}^{1-\tfrac{n-2}{2}t}{\eell}^{(p)}_i(x)dx = 
{\tilde F}_i(t)\geq 0.$$
\end{enumerate}
\end{propose}

\subsection{Polynomial matrices associated  to   $\nu^1_{-i}(a)$, $\mu^1_{-i}(a)$, 
${\tilde \nu}^1_{-i}(a)$ and ${\tilde \mu}^1_{-i}(a)$}

\vspace{5pt}

Next we will consider a set of $n-2$ matrices with entries in $\Q[t]$. This 
set along with the set given in Definition~\ref{5.4} will be the generating 
set of polynomials for the function $f_{R_{p, n+1}}$. 

We recall  
   two sets of matrices as given in Notations~\ref{Bmat} from the Appendix, 
\begin{enumerate}
\item $\{{\B}_{0}^{(p)}, {\B}_{1}^{(p)}, \ldots, {\B}_{n-3}^{(p)}\} \subset 
M_{n+2}(\Z_{\geq 0})$ when 
$n\geq 3$ is odd and
\item $\{{\C}_{0}^{(p)}, {\C}_{1}^{(p)}, \ldots, {\C}_{n-3}^{(p)}\} \subset 
M_{n+2}(\Z_{\geq 0})$, when 
$n\geq 4$ is even
\end{enumerate} 

We write 
$${\B}_{n_i}^{(p)} = {[b_{k_1, k_2}^{(n_i)}(p)]}_{0\leq k_1, k_2\leq n+1}, \quad 
C_{n_i}^{(p)} = \bigl[c_{k_1, k_2}^{(n_i)}(p)\bigr]_{0\leq k_1, k_2\leq n+2}.$$

\begin{defn}\label{polye}
 Lemma~\ref{entries} and Lemma~\ref{eentries} imply that  
each $b_{k_1, k_2}^{(n_i)}(p)$ and  $c_{k_1, k_2}^{(n_i)}(p)$ has 
a polynomial expression, that is, 
there exist  polynomials 
$$\{H_{k_1, k_2}^{(n_i)}(t)\in \Q[t]\mid 0\leq n_i\leq n-3\quad\mbox{and}\quad 
0\leq k_1, k_2 \leq n+1\}$$
and 
$$\{{\tilde H}_{k_1, k_2}^{(n_i)}(t)\in \Q[t]\mid 
0\leq n_i\leq n-3\quad\mbox{and}\quad 
0\leq k_1, k_2 \leq n+2\}$$
 such that 
 $${b^{(n_i)}_{k_1, k_2}(p)}/{p^n} =  
H_{k_1, k_2}^{(n_i)}(t)\vert_{t= 1/p},
 \quad\mbox{for all}\quad p \geq 3n-4$$
and
$${c^{(n_i)}_{k_1, k_2}(p)}/{p^n} =  
{\tilde H}_{k_1, k_2}^{(n_i)}(t)\bigr\rvert_{t =1/p},
\quad\mbox{for all}\quad p \geq (3n-4)/2.$$

We denote 
\begin{enumerate}\item if $n\geq 3$ is odd, then 
$${\HH}^{(n_i)}(t) = \left[H_{k_1, k_2}^{(n_i)}(t)\right]_{0\leq k_1, k_2\leq n+1}$$ 
\item $n\geq 4$ is even, then 
$${\tilde {\HH}}^{(n_i)}(t) = 
\left[{\tilde H}_{k_1, k_2}^{(n_i)}(t)\right]_{0\leq k_1, k_2\leq n+2}.$$ 
\end{enumerate}
\end{defn}

\section{Decomposition of the difficult range}

We showed in the previous section that  the rank functions 
${\bf r}^{(p)}_i:[0, 1)\longto [0, \infty)$ 
are  piecewise polynomial functions
in the complement of  the difficult range, which is 
(1) the 
 intervals $[0, \frac{1}{2}-\frac{n-2}{2p})$ and 
$[\frac{1}{2}+\frac{n-2}{2p}, 1)$, when $n\geq 3$ is odd,  
and (2)   the interval $[\tfrac{n-2}{2p},~~~1-\tfrac{n-2}{2p}]$ 
when $n\geq 4$ is even.
Here with the help of matrices $\{{\HH}^{(n_i)}(t)\}_i$ and 
$\{{\tilde {\HH}}^{(n_i)}(t)\}_i$ as given in Definition~\ref{polye} we will be able to 
 describe these functions in the difficult range too.

To do this 
 we almost cover the difficult range  by countably infinitely many intervals which are 
indexed  combinatorially  such that each
${\bf r}_i^{(p)}$ and hence  the HK density function, 
when restricted to one such 
interval,  has
 neat expression in terms of the finite set of  polynomials as given 
in Definition~\ref{5.4} and Definition~\ref{polye}.

Let $\sM= \{0, 1. \ldots, n-3\}$. For the sake of abbreviation we will denote a tuple 
$(n_1, \ldots, n_l)\in\sM^l$ by ${\underline n}$, provided there is no confusion.
 If $n\geq 3$ is odd then
we construct a set of semi open intervals 
$$\{I^{(p)}_{{\underline n}}, J^{(p)}_{{\underline n}}\mid 
{\underline n} = (n_1, \ldots, n_l)\in  \bigcup_{l\geq 1}\sM^l\}$$ such that 

\begin{enumerate}
\item  the indexing set is $\sM^l$ which is independent of $p$,
 although  each interval 
$I^{(p)}_{\underline n}$ and $J^{(p)}_{\underline n}$ will depend on $p$. 
\item These semi open intervals $I^{(p)}_{\underline n}, J^{(p)}_{\underline n}$
{\em almost cover} ({\em i.e.}, the uncovered part is of measure 
$0$) the difficult range
$[\tfrac{1}{2}-\tfrac{(n-2)}{2p},~~ \tfrac{1}{2}+\tfrac{(n-2)}{2p})$.
 \item These semi open intervals are disjoint from each other.
\end{enumerate}

Similarly, for even $n\geq 4$, we construct such a set of 
disjoint semi open intervals 
$\{{\tilde I}^{(p)}_{\underline n}\mid 
{\underline n}\in {{\tilde \sM}}_0\cup {{\tilde \sM}}_1 = \cup_{l\in \N}\sM^l \}$ 
which almost cover the difficult range, namely, the interval 
$[0,\quad \tfrac{n-2}{2p})\cup [1-\tfrac{n-2}{2p}, \quad 1)$.
Now we start with a formal set up

\begin{defn}\label{soint}Let $n\geq 3$ be an integer. 
Let ${{\sA}} \subseteq \{0, 1, \ldots, p-1\}$ 
be a set of $n-2$ elements indexed by the set $\sM = \{0, 1, \ldots, n-3\}$,
that is  every $j_i\in {{\sA}}$ is indexed by unique $n_i\in \sM$.
Let 
$\sB = \{b_0, b_0+1, \ldots, b_0+t_0-1\}$ $\subset \{0, 1, \ldots, p-1\}$ be a fixed 
 set of consecutive $t_0$ integers
such that ${\sA}\cap \sB = \phi$.
Then, for $l\geq 1$,  we define 
 $$I^{\sB(p)}_{(n_1, \ldots, n_l)} = 
\bigl[\sum_{i=1}^l\tfrac{j_i}{p^i}+ \tfrac{b_0}{p^{l+1}},\quad  
\tfrac{b_0+t_0}{p^{l+1}}+ \sum_{i=1}^l\tfrac{j_i}{p^{i}}\bigr).$$
Similarly if
 $\sD = \{d_0, d_0+1, \ldots, d_0+t_1-1\}$ 
is another such set of consecutive $t_1$ integers in $\{0, \ldots, p-1\}$ 
such that ${\sA}\cap \sD = \phi$ then we define 
 $$I^{\sD(p)}_{(n_1, \ldots, n_l)} = 
[\sum_{i=1}^l\tfrac{j_i}{p^i}+ \tfrac{d_0}{p^{l+1}},\quad  
\tfrac{d_0+t_1}{p^{l+1}}+ \sum_{i=1}^l\tfrac{j_i}{p^{i}}).$$
\end{defn}

\begin{lemma}\label{disjoint}Let $n\geq 3$ be an integer.
If the sets ${\sA}$, $\sD$ and $\sB$, as given in Definition~\ref{soint}, 
 are mutually disjoint sets  then 
 the set of semi open intervals
$$\{I^{\sB(p)}_{(n_{i_1}, \ldots, n_{i_l})}, \quad I^{\sD(p)}_{(n_{k_1}, 
\ldots, n_{k_{l_1}})}\mid
(n_{i_1}, \ldots, n_{i_l})\in \sM^{l},~~~(n_{k_1}, 
\ldots, n_{k_{l_1}})\in \sM^{l_1}
 \quad l, l_1\geq 1 \}$$
is a set of mutually disjoint intervals.
\end{lemma}
\begin{proof} By construction,  each of these semi open intervals is a subset  of  $[0, 1)$.
Therefore, for any $x$ in one of such intervals and for any $q=p^s$, 
we have the $p$-adic expansion
$\lfloor xq\rfloor = a_0+a_1p+\cdots + a_{s-1}p^{s-1}$.
On the other hand we can check that 
if $a$ is  a nonnegative integer with the 
$p$-adic expansion $a = a_0+a_1p+\cdots +a_{s-1}p^{s-1}$  
and $l_0$ is an integer then 
\begin{enumerate}
\item $a/p^s < l_0/p \iff a_{s-1} < l_0$ and 
 \item  $a/p^s \geq  l_0/p \iff a_{s-1} \geq l_0$.
 \end{enumerate}

Let us fix an element  
$x\in I^{\sB(p)}_{(n_{i_1}, \ldots, n_{i_l})}$.

 Then 
for $q\geq p^{l+1}$, the $p$-adic expansion of $\lfloor xq\rfloor $ is given by  
$$\lfloor xq\rfloor = a_0+a_1p+\cdots +
 a_{s-l-1}p^{s-l-1}+ j_{i_l}p^{s-l}+ \cdots +j_{i_1}p^{s-1},$$
where  $a_{s-l-1}\in \sB$ as $b_0 \leq a_{s-l-1} < b_0+t_0$.

Now suppose $x\in I^{\sB(p)}_{(n_{k_1}, \ldots, n_{k_{l_1}})}$. 
Without loss of generality 
we assume $l_1\geq l$, then $j_{i_t} = j_{k_t}$ for $t\leq l$, If $l_1> l$ then 
$j_{k_{l+1}} = a_{s-l-1} \in {\sA}$, which contradicts the assumption that 
${\sA}\cap \sB = \phi$. Therefore $l_1 = l$ and  $(n_{k_1}, \ldots, n_{k_{l_1}}) = 
(n_{i_1}, \ldots, n_{i_{l_1}})$.

If $x\in I^{\sD(p)}_{(n_{k_1}, \ldots, n_{k_{l_1}})}$, where without loss of generality 
we assume $l_1\geq l$ then $j_{i_t} = j_{k_t}$ for $t\leq l$.
Now if $l_1=l$ then 
$ d_0 \leq a_{s-l-1} <d_0+t_1$, which implies $a_{s-l-1}\in \sD$,  otherwise 
$a_{s-l-1} \in {\sA}$. Both outcome contradict the mutual  disjointness property of 
the sets ${\sA}$, $\sB$ and $\sD$. Hence the lemma.
\end{proof}

\begin{defn}\label{cover}Let $S_1\subset S_2$ be two Lebesgue measurable 
set in $\R$. We say $S_1$ {\em almost covers} $S_2$ if the Lebesgue 
measure of the set $S_2\setminus S_1$ is  $0$ in $\R$.
\end{defn}

We use  Lemma~\ref{disjoint} to construct a set of  semi open intervals which  
almost cover  the difficult range.
We first construct such sets when $n\geq 4$ is an even integer.

\vspace{5pt}

\subsection{Decomposition of the difficult range when $n$  is even}

\begin {defn}\label{deven} Let $n\geq 4$ be an even integer. Let $p > n-2$ be a prime. 
Let 
 $${\sA}_e =  \{0, 1, \ldots, \tfrac{n}{2}-2\}\cup 
\{p+1-\tfrac{n}{2}, \ldots, p-1\}$$

and let $\sB = \{\tfrac{n}{2}-1, \ldots, p-\tfrac{n}{2}\}$ be the set of 
consecutive integers. Then $\sA_e$ and $\sB$ are disjoint.

We indexed the elements of the set ${\sA}_e$ by the set  $\sM = \{0,1, 
\ldots, n-3\}$ as follows:
An element $j_i\in {\sA}_e$ is indexed by $n_i\in \sM$ if  
$$j_i = \left\{ \begin{matrix} n_i & \mbox{if}\quad n_i\leq 
\left(\tfrac{n}{2}-2\right)\\\
                p-(n_i-\tfrac{n}{2}-2) & \mbox{if}\quad n_i > 
\left(\tfrac{n}{2}-2\right).\end{matrix}\right.$$

Henceforth in this subsection we  write  the 
set $I^{\sB(p)}_{(n_1, \ldots, n_l)}$
as the set ${\tilde I}^{(p)}_{(n_1, \ldots, n_l)}$.
Following the Definition~\ref{soint}, 
for $(n_1, \ldots, n_l)\in \sM^l$, where $l\geq 1$, we define

$${\tilde I}^{(p)}_{(n_1, \ldots, n_l)} =  
[\sum_{i=1}^l\tfrac{j_i}{p^i}+
\tfrac{n-2}{2p^{l+1}},\quad   \sum_{i=1}^{l}\tfrac{j_i}{p^i}+
\tfrac{1}{p^l} - \tfrac{n-2}{2p^{l+1}})$$

Let $${\tilde \sM}_0 = \{(n_1, \ldots, n_l)\in \sM^l\mid  n_1 \leq \tfrac{n}{2}-2\}_{l\in \N}\quad
\mbox{and}\quad 
{\tilde \sM}_1 = \{(n_1, \ldots, n_l)\in \sM^l \mid  n_1 \geq \tfrac{n}{2}-1\}_{l\in \N}.$$ 

For the sake of abbreviation, provided there is no confusion,  
we will denote an element $(n_1, \ldots, n_l) \in \sM^l$ by 
${\underline n}$

It is easy to check that
$$\bigcup_{{\underline n}\in {\tilde \sM}_0}{\tilde I}^{(p)}_{\underline n} \subset 
[0, \quad \tfrac{n-2}{2p})\quad\quad\mbox{and}\quad\quad
 \bigcup_{{\underline n}\in {\tilde \sM}_1}{\tilde I}^{(p)}_{\underline n} \subset 
[1-\tfrac{n-2}{2p},\quad  1).$$
\end{defn}

 \begin{lemma}\label{emeasure} 
\begin{enumerate}\item The set~~~ $\bigcup_{{\underline n}\in {\tilde \sM}_0} {\tilde I}^{(p)}_{\underline n}$ 
almost covers 
the set $[0 \quad \tfrac{n-2}{2p})$.
\item The set~~~ 
 $\bigcup_{{\underline n}\in {\tilde \sM}_1} {\tilde I}^{(p)}_{\underline n}$
almost covers the set 
$[1-\tfrac{n-2}{2p}, 1)$.
\end{enumerate}
\end{lemma}
\begin{proof}We prove the first assertion 
as the proof for the second assertion is very similar.
Let 
$$A_k = [0 \quad \tfrac{n-2}{2p})\setminus 
\bigcup_{\{(n_1, \ldots, n_l)\in {\tilde \sM}_0\mid l\leq k\}} {\tilde I}^{(p)}_{(n_1, \ldots, n_l)}.$$

It is enough to  prove  that 
the Lebesgue measure of the set $(A_k)$ is 
$$\mu(A_k) = (n-2)^{k+1}/2p^{k+1},\quad\mbox{for all}\quad k\geq 1.$$
We prove the statement by induction on $k$. For $k= 1$

$$A_1 =  [0 \quad \tfrac{n-2}{2p})\setminus 
\bigcup_{0\leq n_i <(n-2)/2} \left[\tfrac{n_i}{p}+\tfrac{n-2}{2p^2}, \quad 
\tfrac{n_i+1}{p}-\tfrac{n-2}{2p^2}\right).$$
Hence 
$$\mu(A_1) = \tfrac{n-2}{2p} - 
\tfrac{n-2}{2}\left[\tfrac{1}{p}-\tfrac{n-2}{p^2}\right] = 
\tfrac{(n-2)^2}{2p^2}.$$
We assume the statement for $k-1$, that is  $\mu(A_{k-1}) = (n-2)^{k}/2p^{k}$.
Now 
$$A_k = A_{k-1} \setminus \bigcup_{(n_{i_1}, \ldots, n_{i_k})
\in {\tilde \sM}_0}{\tilde I}^{(p)}_{(n_{i_1}, \ldots, n_{i_k})}.$$
Since, by Lemma~\ref{disjoint}, these semi open intervals are mutually 
disjoint we have 
$$\mu(A_k) = \mu(A_{k-1}) - \sum_{(n_{i_1}, \ldots, n_{i_k})
\in {\tilde \sM}_0}\ell ({\tilde I}^{(p)}_{(n_{i_1}, \ldots, n_{i_k})}).$$

But 
$$\sum_{(n_{i_1}, \ldots, n_{i_k})
\in {\tilde \sM}_0}\ell({\tilde I}^{(p)}_{(n_{i_1}, \ldots, n_{i_k})}) = \tfrac{n-2}{2}(n-2)^{k-1}
\left(\tfrac{1}{p^k}-\tfrac{n-2}{p^{k+1}}\right) = \tfrac{(n-2)^k}{2p^k}-
\tfrac{(n-2)^{k+1}}{2p^{k+1}},$$
which implies $\mu(A_k) = \tfrac{(n-2)^{k+1}}{2p^{k+1}}$. By induction this
 proves the result for all $k\geq 1$.
\end{proof}

\subsection{Decomposition of the difficult range when $n$  is odd}

\begin{defn}\label{nch}Let $n\geq 3$ be an odd integer and $p >n-2$ a prime.
Let 
$$m_0 = \frac{p}{2}-\frac{n-2}{2}\quad\mbox{and}\quad 
{\sA}_o = \{m_0, m_0+1, \ldots, m_0+n-3\}$$
and let $\sB = \{0, 1, \ldots, m_0-1\}$ and 
$\sD = \{m_0+n-2, \ldots, , p-1\}$ be two sets of consecutive integers.
We indexed the set ${\sA}_o$ by the set  
$\sM = \{0, 1, \ldots, n-3\}$ so that 
an element $j_i$ of ${\sA}_o$ is indexed by $n_i$ if $j_i = n_i+m_0$. Here 
$\sA_o$, $\sB$ and $\sD$ are mutually disjoint sets.

Henceforth in this subsection we 
denote the interval  $I^{\sB(p)}_{(n_1, \ldots, n_l)}$ by 
$I^{(p)}_{(n_1, \ldots, n_l)}$ and the interval $I^{\sD(p)}_{(n_1, 
\ldots, n_l)}$ by $J^{(p)}_{(n_1, \ldots, n_l)}$.
Following Definition~\ref{soint}, for 
 $(n_1, \ldots, n_l)\in \sM^l$, where $l\geq 1$,  we define two semi open intervals 
 
$$I^{(p)}_{(n_1, \ldots, n_l)} = \left[\mbox{$\sum_{i=1}^l\frac{j_i}{p^i},
\quad
\sum_{i=1}^l\frac{j_i}{p^i}+\frac{m_0}{p^{l+1}}$}\right)$$
and
$$J^{(p)}_{(n_1, \ldots, n_l)} = \left[\mbox{$\frac{m_0+n-2}{p^{l+1}}+
\sum_{i=1}^l\frac{j_i}{p^i},\quad
\sum_{i=1}^l\frac{j_i}{p^i}+\frac{p}{p^{l+1}}$}\right).$$

For the sake of abbreviation, provided there is no confusion,  
we will denote an element $(n_1, \ldots, n_l) \in \sM^l$ by 
${\underline n}$.
\end{defn}

\begin{lemma}\label{l1c} Let $n\geq 3$ be an odd integer and $p>n-2$ be an odd
prime. Then 
the set 
$$ \bigcup_{{\underline n}\in \cup_{l\geq 1}\sM^l}
 I^{(p)}_{{\underline n}}\cup 
\bigcup_{{\underline n}\in \cup_{l\geq 1}\sM^l} J^{(p)}_{{\underline n}}\quad\mbox{
almost cover the interval}\quad
[\tfrac{1}{2}-\tfrac{n-2}{2p}, \tfrac{1}{2}+\tfrac{n-2}{2p}).$$
\end{lemma}

If $n=3$ then this covering misses the point $2.5$ as shown in Section~9. 

\begin{proof} 
It is obvious that 
$$\bigcup_{{\underline n}\in \sM^l}I^{(p)}_{\underline n}\cup J^{(p)}_{\underline n}
\subseteq  [\tfrac{1}{2}-\tfrac{n-2}{2p},\quad \tfrac{1}{2}+
\tfrac{n-2}{2p}).$$

Let us denote 
$$A_k = [\tfrac{1}{2}-\tfrac{n-2}{2p}, \quad \tfrac{1}{2}+\tfrac{n-2}{2p})
\setminus  \bigcup_{\{(n_1, \ldots, n_l)\in {\sM^l}\mid l\leq k\}}
 I^{(p)}_{\underline n}\cup J^{(p)}_{\underline n}$$ 
 It is enough to prove that, for each $k\geq 1$, 
 the Lebesgue measure of the set
$A_k =  ((n-2)/{p})^{k+1}$, which can be checked using the induction argument 
for $k$ as done in Lemma~\ref{emeasure}.
\end{proof}

\section{polynomial expression for the rank functions 
${\bf r}_i^{(p)}$ and $f_{R_{p, n+1}}$}

In the previous section,  for $n$ even and for $n$ odd each, 
 we have constructed an almost covering of the 
respective difficult range by semi open intervals such that
the intervals are disjoint and combinatorially indexed.

By construction, the elements belonging to the same subinterval 
have a $p$-adic expression of certain `type', we use this  
 in the following key Proposition to show that the rank functions 
$\eell_i^{(p)}$ when restricted to 
such  intervals, 
are polynomials.

\begin{notations}For $n\geq 3$ odd, let 
\begin{enumerate}
\item $\{{\bf l}_{0}(x), \ldots,{\bf l}_{n+1}(x)\}\subset  \Q[x]$
and $\{{\bf r}_{0}(x), \ldots,{\bf r}_{n+1}(x)\}\subset \Q[x]$ 
be as  in Definition~\ref{5.4}.
\item Let $\{{\HH}^{(n_i)}(t)\mid 0\leq n_i\leq n-3\}\subset M_{n+2}(\Q[t])$ be  
 as in Definition~\ref{polye}.
\item Let $I^{(p)}_{(n_1, \ldots, n_l)}$ and $J^{(p)}_{(n_1, \ldots, n_l)}$ be the 
semi open intervals as in Definition~\ref{nch}.
\item Let $$\phi^{(p)}:\bigcup_{\{{\underline n}\in 
\sM^l\mid l\geq 1\}}I^{(p)}_{\underline n}\cup J^{(p)}_{\underline n}\longto 
[0, \tfrac{1}{2}-\tfrac{n-2}{2p})\cup [\tfrac{1}{2}+\tfrac{n-2}{2p},~~~1),$$
such that for each ${\underline{n}} = (n_1, \ldots, n_l)$ the restriction maps  
 $$\phi^{(p)}:I^{(p)}_{\underline{n}}\longto [0, \tfrac{1}{2}-\tfrac{n-2}{2p})
\quad\mbox{and}\quad \phi^{(p)}:J^{(p)}_{\underline{n}}
\longto [\tfrac{1}{2}+\tfrac{n-2}{2p},~~~1),$$
are the surjective maps given by 
$x\to   p^l(x- \sum_i\tfrac{j_i}{p^i})$, where $j_i = n_i+\tfrac{p}{2}-\tfrac{n-2}{2}$.
\end{enumerate}
\end{notations}

\begin{propose}\label{pex}Let  $n\geq 3$ is an odd integer, where $p\geq 3n-4$.
 Then, for ${\underline{n}} = (n_1, \ldots, n_l)$
  the rank functions 
 $\eell^{(p)}_{i}:I^{(p)}_{\underline{n}} \longto \R$ are given  by the formula
\begin{multline}\label{pexe1}
\Bigl[\eell^{(p)}_0(x),\ldots, \eell^{(p)}_{n+1}(x)\Bigr]_{1\times n+2} =\\ 
\left[{\bf l}_{0}(\phi^{(p)}(x)), \ldots, 
{\bf l}_{n+1}(\phi^{(p)}(x))\right]_{1\times n+2}\cdot {\HH}^{(n_l)}(t)
\cdots {\HH}^{(n_1)}(t)\Bigr\rvert_{t=1/p}.\end{multline}
Similarly the rank functions  $\eell^{(p)}_{i}:J^{(p)}_{(n_1, \ldots, n_l)} \longto \R$ 
are given by the formula
\begin{multline}\label{pexe3}
\left[\eell^{(p)}_0(x),\ldots, \eell^{(p)}_{n+1}(x)\right]_{1\times n+2} =\\ 
\left[{\bf r}_{0}(\phi^{(p)}(x)), \ldots, 
{\bf r}_{n+1}(\phi^{(p)}(x))\right]_{1\times n+2}\cdot 
{\HH}^{(n_l)}(t)
\cdots {\HH}^{(n_1)}(t)\Bigr\rvert_{t=1/p}.\end{multline}

In particular we get a formula for $\mmu^{(p)}_{n_0}(x) = \eell^{(p)}_{n+1}(x)$.
\end{propose}

\begin{proof}We fix  $x\in I^{(p)}_{(n_1, \ldots, n_l)}$, then  
$j_i = n_i+\tfrac{p}{2}-\tfrac{n-2}{2}$ and
$$I^{(p)}_{(n_1, \ldots, n_l)} = 
\Bigl[\sum_{i=1}^l\tfrac{j_i}{p^i},
\quad \sum_{i=1}^l\tfrac{j_i}{p^i}+\tfrac{m_0}{p^{l+1}}\Bigr).$$
Therefore for any given  
$q\geq p^{l+1}$ the $p$-adic expansion 
of $\lfloor xq\rfloor $ is given by 
$$\lfloor xq\rfloor = a_0+\cdots + a_{s-l-1}p^{s-l-1}+j_lp^{s-l}+\cdots +
j_1p^{s-1},$$
where 
  $y:= \phi^{(p)}(x)\in [0, \tfrac{1}{2}-\tfrac{n-2}{2p})$ lies in the
complement of the difficult range. Hence by  
Proposition~\ref{lp} 
$$\eell_i^{(p)}(y) =  
{\bf l}_{i}(y),\quad \mbox{for all}\quad 0\leq i\leq n+1.$$
Let us denote 
$$A_{s-l} = a_0+\cdots + a_{s-l-1}p^{s-l-1} = \lfloor yp^{s-l}\rfloor.$$ 
Then 
 $\sO(\lfloor xq\rfloor) = \sO(A_{s-l})\tensor 
F^{*s-l}\sO(j_l+\cdots + j_1p^{l-1})$.
Therefore, by projection formula
$$F_*^s(\sO(\lfloor xq\rfloor)) = 
F_*^l\Bigl(F_*^{s-l}(\sO(A_{s-l}))\tensor \sO(j_l+\cdots+j_1p^{l-1})\Bigr).$$

Let the rank tuple of $F_*^{{s-l}}(\sO(A_{s-l}))$ 
be denoted as 
$$(l_0, l_1, \ldots, l_{n+1}) =  \left(l_0^{s-l}(A_{s-l}), l_1^{s-l}(A_{s-l}),  
\ldots, l_{n+1}^{s-l}(A_{s-l})\right)$$
where
 $$F_*^{s-l}(\sO(A_{s-l}))  = \sO^{l_0} \oplus \cdots \oplus \sO(-n+1)^{l_{n-1}}\oplus
\sS(-n_0+1)^{l_n}\oplus \sS(-n_0)^{l_{n+1}}.$$
Then we can write
 
\begin{multline*}F_*^{s-l}(\sO(A_{s-l}))\tensor \sO(j_l+\cdots+j_1p^{l-1})\\
  = \bigl(\sO(j_l)^{l_0} \oplus \cdots \oplus \sO(j_l-n+1)^{l_{n-1}}\oplus
\sS(j_l-n_0+1)^{l_n}\oplus \sS(j_l-n_0)^{l_{n+1}}\bigr)
\tensor F^*\sO(j_{l-1}+\cdots + j_1p^{l-2}).\end{multline*}

Again, by projection formula 
\begin{multline*}
F_*^s(\sO(\lfloor xq\rfloor)) = F_*^{l-1}\Bigl(F_*\bigl(F_*^{s-l}(\sO(A_{s-l}))
\tensor \sO(j_l+\cdots+j_1p^{l-1})\bigr)\Bigr)\\
= F_*^{l-1}\Bigl(\bigl(F_*\sO(j_l)^{l_0}
\oplus\cdots 
\oplus F_*S(-n_0+j_l)^{l_{n+1}}\bigr)
\otimes \sO(j_{l-1}+\cdots +j_1p^{l-2})\Bigr).\end{multline*}

Note that $p\geq 3n-4$ implies that, for all  $0\leq t < n$ we have  
$0\leq j_i- t <p$. Hence by Lemma~\ref{l1}~(2)
the decomposition of each $F_*\sO(j_l-t)$ will only have bundles of type
$\sO$, $\sO(-1)$, $\ldots, \sO(1-n)$
and $\sS(-n_0+1)$ and $\sS(-n_0)$.
Therefore we can write

\begin{multline*}
F_*\sO(j_l)^{l_0}\oplus \cdots \oplus F_*\sO(j_l-n+1)^{l_{n-1}}
\oplus F_*S(-n_0+1+j_l)^{l_{n}}
\oplus F_*S(-n_0+j_l)^{l_{n+1}}\\\
 = \sO^{b_0}\oplus\cdots \oplus\sO(1-n)^{b_{n-1}}\oplus
\sS(-n_0+1)^{b_n}\oplus \sS(-n_0)^{b_{n+1}},\end{multline*}
where as a matrix

$[b_0, \ldots, b_{n+1}]_{1\times n+2} = [l_0, \ldots, l_{n+1}]_{1\times n+2}
\cdot {\B}^{(p)}_{n_l}$, and  where we recall 

$${\B}^{(p)}_{n_l} = \left[\begin{matrix}
\nu_{0}(j_l), \ldots, \nu_{-n+1}(j_l), \mu_{-n_0+1}(j_l), \mu_{-n_0}(j_l)\\
\vdots\\
\nu_{0}(j_l-n+1), \ldots,   \nu_{-n+1}(j_l-n+1), \mu_{-n_0+1}(j_l-n+1),
\mu_{-n_0}(j_l-n+1)\\
{\tilde \nu}_{0}(j_l-n_0+1),  \ldots, {\tilde \nu}_{-n+1}(j_l-n_0+1), 
{\tilde \mu}_{-n_0+1}(j_l-n_0+1),
{\tilde\mu}_{-n_0}(j_l-n_0+1)\\
{\tilde \nu}_{0}(j_l-n_0),  \ldots, {\tilde \nu}_{-n+1}(j_l-n_0), 
{\tilde \mu}_{-n_0+1}(j_l-n_0), {\tilde \mu}_{-n_0}(j_l-n_0)
\end{matrix}\right].$$

Iterating this we can write the rank tuple  for $F_*^s(\sO(\lfloor xq\rfloor))$
as 
\begin{multline}\label{ff1}
\left[\nu_0^s(\lfloor xq\rfloor), \ldots,\nu_{n-1}^s(\lfloor xq\rfloor), 
\mu_{-n_0+1}^s(\lfloor xq\rfloor), \mu_{-n_0}^s(\lfloor xq\rfloor)\right]_{1
\times n+2}\\
 = \left[l_0,  \ldots, l_{n-1}, l_n, l_{n+1}\right]_{1\times n+2} 
\cdot {\B}^{(p)}_{n_l}\cdots {\B}^{(p)}_{n_1}.\end{multline}

We note that each matrix ${\B}_{n_i}^{(p)}$ is independent of $s$, where $q=p^s$.
Now
\begin{multline*}
\lim_{s\to \infty}\frac{l_i}{p^{sn}} = \lim_{s\to \infty}
\frac{l_i^{s-l}(A_{s-l})}{p^{ln}(p^{s-l})^n} = \lim_{s\to \infty}
\frac{l_i^{s-l}(\lfloor yp^{s-l}\rfloor)}{p^{ln}(p^{s-l})^n} 
= \frac{{\bf r}_i^{(p)}(y)}{p^{ln}} = \frac{{\bf l}_i(y)}{p^{ln}},\end{multline*}
where the second last equality follows from Lemma~\ref{cont}.

Hence  multiplying both the sides of (\ref{ff1}) by $q^n = p^{sn}$ and taking the limit 
as $s\to \infty$ we get 
\begin{multline*}
\left[\nnu^{(p)}_0(x), \ldots, \nnu^{(p)}_{n-1}(x), 
\mmu_{n_0-1}(x), \mmu_{n_0}(x)\right]_{1\times n+2}\\
 = \left[{\bf l}_{0}(y), {\bf l}_{1}(y),  
\ldots, {\bf l}_{n+1}(y)\right]_{1\times n+2}\cdot(\tfrac{1}{p^n}\cdot {\B}^{(p)}_{n_l})
\cdots (\tfrac{1}{p^n}\cdot {\B}^{(p)}_{n_1}),\end{multline*}

This gives (\ref{pexe1}).

We can argue similarly for the function
$\eell_i^{(p)}:J^{(p)}_{\underline{n}}\longto \R$.
 where  now $y$ lies  in the interval
$[\tfrac{1}{2}+\tfrac{n-2}{2p}, 1)$ and hence ${\bf l}_{i}$ will be replaced
by ${\bf r}_{i}$.
\end{proof}

 The proof of the following proposition is  along the same 
lines as for Proposition~\ref{pex}. We recall the following.

\begin{notations} For  $n\geq 4$ an even integer, 
\begin{enumerate}
 \item let $\{{{\bf m}}_{0}(x), \ldots, {{\bf m}}_{n+2}(x) 
\in \Q[x]\}$  
as  in Definition~\ref{5.4}.
\item Let $\{{\tilde {\HH}}^{(n_i)}(t)\in M_{n+3}(\Q[t])\mid 0\leq n_i\leq n-3\}$
 as in Definition~\ref{polye}.
\item Let ${\tilde I}^{(p)}_{(n_1, \ldots, n_l)}$ be  the interval as in  
 Definition~\ref{deven}. 
\item Let $$\psi^{(p)}:\bigcup_{\{{\underline n}\in 
\sM^l\mid l\geq 1\}}{\tilde I}^{(p)}_{\underline n}\longto 
[\tfrac{n-2}{2p},~~~ 1-\tfrac{n-2}{2p}),$$
such that for each ${\underline{n}} = (n_1, \ldots, n_l)$ the restriction map  
 $$\psi^{(p)}:{\tilde I}^{(p)}_{\underline{n}}
\longto [\tfrac{n-2}{2p},~~~ 1-\tfrac{n-2}{2p}),\quad\mbox{given by}\quad 
y\to   p^l(x- \sum_i\tfrac{j_i}{p^i})$$
is  surjective.

\end{enumerate}
\end{notations}

\begin{propose}\label{epex}Let $n\geq 4$ be an even integer, where $p\geq (3n-4)/2$.
Then the functions  
$\eell^{(p)}_{i}:{\tilde I}^{(p)}_{(n_1, \ldots, n_l)} \longto \R$ 
are given  by the formula
\begin{multline*}
\left[\eell^{(p)}_0(x),\ldots, \eell^{(p)}_{n+2}(x)\right]_{1\times n+3} =\\ 
\left[{{\bf m}}_{0}\bigl(\psi^{(p)}(x)\bigr), \ldots, 
{{\bf m}}_{n+2}\bigl(\psi^{(p)}(x)\bigr)
\right]_{1\times n+3}\cdot 
{\tilde {\HH}}^{(n_l)}(t)
\cdots {\tilde {\HH}}^{(n_1)}(t)\Bigr\rvert_{t=1/p}.\end{multline*}

In particular we get  formulas for $\mmu^{(p)}_{n_0+1}(x) = \eell^{(p)}_{n+2}(x)$ 
and for $\mmu^{(p)}_{n_0-1}(x) = \eell^{(p)}_{n}(x)$.
\end{propose}

\vspace{5pt}

Now we are ready to
define the HK density function $f_{R_{p, n+1}}:[0, \infty)\longto \R$
almost everywhere, that is outside a set of measure $0$ in $\R$. 

\begin{notations}
We recall that $\{{\bf Z}_{i}(x)\}$ and $\{{\bf Y}_{i}(x)\}$ are polynomials as 
in Notations~\ref{n2} and  
 $n_0 = \lceil \tfrac{n}{2}\rceil-1$.
\begin{enumerate}\item
If $n$ is odd then, by Lemma~\ref{l1c}
the set 
$[\tfrac{1}{2}-\tfrac{(n-2)}{2p},~~\tfrac{1}{2}+\tfrac{(n-2)}{2p})$ is almost covered by
the set of disjoint intervals 
$\{I^{(p)}_{\underline n},  J^{(p)}_{\underline n}
\mid{\underline n}\in \cup_{l\geq 1}\sM^l\}$, which are  as in Definition~\ref{nch}.
\item 
If $n\geq 4$ is even then, by Lemma~\ref{emeasure}
the set 
$\left[1-\tfrac{n-2}{2p}, 1\right)$ is almost covered by the set 
$\{{\tilde I}^{(p)}_{\underline n}\mid {\underline n}\in {\tilde \sM_0}\},$
and the set 
$[1-\tfrac{(n-2)}{2p},~~~1)$ is almost covered by 
the set of intervals $\{{\tilde I}^{(p)}_{\underline n}\mid 
{\underline n}\in {\tilde \sM_1}\}$, 
 as in Definition~\ref{deven}. 
\end{enumerate}
\end{notations}

Following theorem shows that outside a set of measure $0$ the HK density function 
$f_{R_p, n+1}$ is a piecewise polynomial function though they will be 
infinitely many pieces.

\begin{thm}\label{expr} Let $n\geq 3$ be an odd integer and let $p\geq 3n-4$.
 Then 
the HK density function is  defined  outside a set of measure $0$
as follows:
$$f_{R_p, n+1}(x)  = 
\begin{cases} {\bf Z}_{i}(x) 
 \quad\mbox{if}\quad 
 i\leq x < i+1 \quad\mbox{and}\quad 0\leq i \leq  n_0\\\\
 {\bf  Z}_{n_0+1}(x), \quad\mbox{if}\quad  
(n_0+1) \leq x < (n_0+\frac{3}{2}) -\frac{n-2}{2p}\\\\
{\bf  Y}_{n_0+1}(x)+ {\mmu^{(p)}_{n_0}}(x-n_0-1),\quad
\mbox{if}\quad  \tfrac{1}{2}-\tfrac{n-2}{2p}\leq x -(n_0+1) < 
\tfrac{1}{2}+\tfrac{n-2}{2p}\\\\
{\bf Y}_{n_0+1}(x),\quad\mbox{if}\quad  
(n_0+1)+\frac{1}{2}+\frac{n-2}{2p} \leq x < (n_0+2)\\\\
{\bf Y}_{i}(x),\quad\mbox{if}\quad 
 i \leq x < i+1 \quad\mbox{and}\quad n_0+2\leq i <  n
\end{cases}$$

where
\begin{enumerate}
\item[(i)] for 
$x-n_0-1\in I^{(p)}_{(n_1, \ldots, n_l)}$
\begin{multline}\label{f1}
\mmu_{n_0}^{(p)}(x-n_0-1) = (n+1)^{th}~~\mbox{entry of the matrix}\quad\\\
\Bigl[{\bf l}_{0}\bigl(\phi^{(p)}(x-n_0-1)\bigr), \ldots, {\bf  l}_{n+1}
\bigl(\phi^{(p)}(x-n_0-1)\bigr)\Bigr]_{1\times n+2}\cdot 
{\HH}^{(n_l)}(t)\cdots {\HH}^{(n_1)}(t)\Bigr\rvert_{t=1/{p}}.\end{multline}

\item[(ii)] For  $x-n_0-1\in J^{(p)}_{(n_1, \ldots, n_l)}$
\begin{multline}\label{f2}
\mmu_{n_0}^{(p)}(x-n_0-1) = (n+1)^{th}~~\mbox{entry of the matrix}\quad\\\
\left[{\bf r}_{0}\bigl(\phi^{(p)}(x-n_0-1)\bigr), \ldots, 
{\bf r}_{n+1}\bigl(\phi^{(p)}(x-n_0-1)\bigr)\right]_{1\times n+2}\cdot
{\HH}^{(n_l)}(t)\cdots {\HH}^{(n_1)}(t)\Bigr\rvert_{t=1/{p}}.\end{multline}
\end{enumerate}
\end{thm}

\begin{thm}\label{eexpr}
If $n\geq 4$ is  an even number and $p\geq (3n-4)/2$ then the 
HK density function is defined outside measure $0$ (in $\R$) as follows:

$$f_{R_{p, n+1}}(x) =
 \begin{cases}  {\bf Z}_{i}(x), \quad\mbox{if}\quad 
 i\leq x < i+1 \quad\mbox{and}\quad 0\leq i \leq  n_0\\\\
{\bf  Z}_{n_0+1}(x),\quad\mbox{if}\quad  
(n_0+1) \leq x < (n_0+2) -\frac{n-2}{2p}\\\\
 {\bf  Z}_{n_0+1}(x)+
{\mmu}^{(p)}_{n_0-1}(x-n_0-1),
\quad\mbox{if\quad  $1-\frac{n-2}{2p}\leq x -(n_0+1)< 1$} \\\\
{\bf  Y}_{n_0+2}(x)+ 
{\mmu}^{(p)}_{n_0+1}(x-n_0-2),
\quad\mbox{if\quad  $0\leq x -(n_0+2) < \small{\frac{n-2}{2p}}$}\\\\
 {\bf Y}_{n_0+2}(x),\quad\mbox{if}\quad  
(n_0+2)+\frac{n-2}{2p} \leq x < (n_0+3)\\\\
  {\bf Y}_{i}(x),\quad\mbox{if}\quad 
 i\leq x < i+1 \quad\mbox{and}\quad n_0+3\leq i < n,
\end{cases}$$
 where 
\begin{enumerate}
\item[(1)] for 
$x-n_0-1\in {\tilde I}^{(p)}_{(n_1, \ldots, n_l)}$, where 
$(n_1, \ldots, n_l)\in {\tilde \sM_1}$ 
\begin{multline}\label{ef1}
\mmu_{n_0-1}^{(p)}(x-n_0-1) = n^{th}~~\mbox{entry of the matrix}\quad\\\
\left[{{\bf m}}_{0}\bigl(\psi^{(p)}(x-n_0-1)\bigr), \ldots, 
{\bf m}_{n+2}\bigl(\psi^{(p)}(x-n_0-1)\bigr)\right]_{1\times n+3}\cdot
{\tilde {\HH}}^{(n_l)}(t)\cdots 
{\tilde {\HH}}^{(n_1)}(t)\Bigr\rvert_{t=1/{p}}.\end{multline}

\item[(2)] For 
$x-n_0-2\in {\tilde I}^{(p)}_{(n_1, \ldots, n_l)}$,  where $(n_1, \ldots, n_l)\in 
{\tilde \sM_0}$
\begin{multline}\label{ef2}
\mmu_{n_0+1}^{(p)}(x-n_0-2) = (n+2)^{th}~~\mbox{entry of the matrix}\quad\\\
\left[{{\bf m}}_{0}\bigl(\psi^{(p)}(x-n_0-2)\bigr), 
\ldots, {{\bf m}}_{n+2}\bigl(\psi^{(p)}(x-n_0-2)\bigr)\right]_{1\times n+3}\cdot 
{\tilde {\HH}}^{(n_l)}(t)\cdots {\tilde {\HH}}^{(n_1)}(t)\Bigr\rvert_{t=1/{p}}.\end{multline}
\end{enumerate}
\end{thm}

\section{Applications}
In the previous section we constructed an infinite set of semi open intervals 
such that they  cover the domain of 
definition of the  HK density function $f_{R_{p, n+1}}$ except on a 
set of measure $0$. 

Since $f_{R_{p, n+1}}$ is a continuous function,  
we can ignore the measure $0$ part as far as the integration is concerned.
Moreover $f_{R_{p, n+1}}$ restricted to such semi open 
 interval is a polynomial functions 
which are independent of $p$. 
Though the length of each interval depends on $p$, the set of indexing is 
independent of $p$.
These properties of the formulation of $f_{R_{p, n+1}}$ 
allow us to compare the HK multiplicity $e_{HK}(R_{p, n+1})$ as 
$p$ varies.
As  a consequence   
we settle  an old conjecture of Yoshida (Conjecture~(2) in the introduction)
in the following form 

\vspace{5pt}

\noindent{\bf Theorem}~(C)\quad{\em Let $n\geq 3$ be an integer then there 
exists $\epsilon > 0$ 
such that for $p\geq 1/\epsilon $,
$e_{HK}(R_{p_1, n+1})$ is a strictly decreasing function of $p$.

Moreover, following Notations~\ref{epsilon}, the $\epsilon = \epsilon_1$
or  $\epsilon = \epsilon_2$ depending on the parity of $n$.}

\vspace{5pt}

Here we state and prove  a basic result from  analysis in the form which 
suits our purpose.

\begin{lemma}\label{positive}If $H(t) = t^i(b_0+b_1t+\cdots + b_mt^m)\in \Q[t]$ 
is a polynomial such that 
$H(1/{p}) \geq 0$ for all $p\gg 0$ then either $H(t)$ is a zero polynomial or 
$tH(t)$ is 
a strictly increasing function on the interval $[0, \epsilon_H)$, where we define   
$$\epsilon_H = \mbox{min}~\{1, 
\frac{b_0}{2|b_1|+3|b_2|+\cdots +(m+1)|b_m|}\}.$$ 
 
In particular $tH(t)>0$ for $t\in (0, \epsilon_H)$.
\end{lemma}
\begin{proof}Suppose $H(t)$ is not a zero polynomial. Further 
we can assume $m > 0$.
\vspace{5pt}

\noindent{\bf  Claim}.\quad  The number  $b_0$ is strictly positive.
\vspace{5pt}

\noindent{Proof of the claim}:\quad  Suppose $b_0 <0$.

Since $H(t)$ is  a nonzero polynomial it has only finitely many zeroes. Therefore 
there exists $\epsilon_1 >0$ such that $H(t)$ has no zero on 
the interval $(0, \epsilon_1)$.
Now by choosing  $p_0\geq 1/\epsilon_1$ we can ensure that 
 $H(t)$ has no zero in the set 
$\{1/{p}\mid p\geq p_0\}$. Since $-b_0 >0$, we can further choose $\epsilon_1$  
such that for all $t\in (0,~~\epsilon_1)$ we have 
$t\left(|b_{1}|+\cdots + |b_m|t^{m-1}\right) < -b_0$ which implies
$$ b_0+b_{1}t+\cdots + b_mt^{m} \leq b_0 + 
|b_{1}|t+\cdots + |b_m|t^{m} <0.$$

But then it contradicts the hypothesis that 
$H\bigl(1/p)\bigr) \geq 0$
for all $p \gg  0$. This proves the claim.
 
\vspace{5pt}
 
Let   $G(t) = t(b_{0}+\cdots + b_{m}t^{m})$, then 
$$\frac{d(G(t))}{dt} = 
b_0+2b_{1}t+\cdots + (m+1)b_mt^{m} > 0,\quad\mbox{for}\quad t\in [0, \epsilon_H),$$
 as
$$-2b_1t - 3b_2t^2-\cdots -  (m+1)b_mt^m
< t\bigl(2|b_1|+\cdots + (m+1)|b_m|\bigr) <  b_0.$$
Hence $G(t)$ and therefore $t(b_0+\cdots + b_mt^m)(t^i) = tH(t)$ is 
 a strictly increasing function on $[0, \epsilon_H)$.
In particular $tH(t) > 0$ for 
 $t\in (0, \epsilon_H)$.
\end{proof}

In  the rest of the section we use the notations as in 
 Definitions~\ref{polye} and 
\ref{5.4} and Proposition~\ref{lp}.

The following lemma proves that  the integral of $\eell_i^{(p)}$, restricted to the 
each semi open  interval, is determined by 
 a  polynomial in $\Q[t]$ evaluated at $t=1/{p}$.

\begin{lemma}\label{integrals}
 Let $n\geq 3$ be an odd integer and let $p\geq 3n-4$. 
Then for  $(n_1, \ldots, n_l)\in \sM^l$, where $l\geq 1$

\begin{multline}\label{pexe2}
\Bigl[\int_{I^{(p)}_{(n_1, \ldots, n_l)}}\eell^{(p)}_0(x)dx,\ldots, 
\int_{I^{(p)}_{(n_1, \ldots, n_l)}}\eell^{(p)}_{n+1}(x)\Bigr]_{1\times n+2} = \\
\left[F_{0}(t), \ldots, F_{n+1}(t)\right]_{1\times n+2}\cdot 
(t\cdot{\HH}^{(n_l)}(t))
\cdots (t\cdot{\HH}^{(n_1)}(t))\Bigr\rvert_{t={1}/{{p}}}\end{multline}

and
\begin{multline}\label{pexe4}
[\int_{J^{(p)}_{(n_1, \ldots, n_l)}}\eell^{(p)}_0(x)dx,\ldots, 
\int_{J^{(p)}_{(n_1, \ldots, n_l)}}\eell^{(p)}_{n+1}(x)]_{1\times n+2} = \\
\left[G_{0}(t), \ldots, G_{n+1}(t)\right]_{1\times n+2}\cdot 
\bigl(t\cdot{\HH}^{(n_l)}(t)\bigr)
\cdots \bigl(t\cdot{\HH}^{(n_1)}(t)\bigr)\Bigr\rvert_{t={1}/{{p}}}.\end{multline}

Let $n\geq 4$ be an even integer and let $p\geq (3n-4)/2$ then 
\begin{multline}\label{pexe5}
\Bigl[\int_{{\tilde I}^{(p)}_{(n_1, \ldots, n_l)}}\eell^{(p)}_0(x)dx,\ldots, 
\int_{I^{(p)}_{(n_1, \ldots, n_l)}}\eell^{(p)}_{n+2}(x)\Bigr]_{1\times n+3} = \\
\left[{\tilde F}_{0}(t), \ldots, {\tilde F}_{n+2}(t)\right]_{1\times n+3}\cdot 
\bigl(t\cdot{\tilde {\HH}}^{(n_l)}(t)\bigr)
\cdots \bigl(t\cdot{\tilde {\HH}}^{(n_1)}(t)\bigr)\Bigr\rvert_{t={1}/{{p}}}.
\end{multline}
\end{lemma}
\begin{proof}For $I^{(p)}_{(n_1, \ldots, n_l)}$ and 
$J^{(p)}_{(n_1, \ldots, n_l)}$ as in Definition~\ref{nch}
$$ \int_{I^{(p)}_{(n_1, \ldots, n_l)}}
{\bf l}_{i_0}\Bigl(p^l(x-\sum_i\frac{j_i}{p^i})\Bigr) dx = 
 \frac{1}{p^l}\int_0^{\tfrac{1}{2}-\tfrac{n-2}{2p}}{\bf l}_{i_0}(x)dx= 
t^{l}F_{i_0}(t)\Bigr\rvert_{t = 1/{p}}$$
and

$$\int_{J^{(p)}_{(n_1, \ldots, n_l)}}
{\bf r}_{i_0}\Bigl(p^l(x-\sum_i\frac{j_i}{p^i})\Bigr) dx =
 \frac{1}{p^l}\int_{\tfrac{1}{2}+\tfrac{n-2}{2p}}^{1}{\bf m}_{i_0}(x)dx = 
t^{l}\cdot G_{i_0}(t)\Bigr\rvert_{t=1/{p}}.$$
Therefore by (\ref{pexe1}) and (\ref{pexe3}) of Proposition~\ref{pex} the assertions follow.

For even $n$ the assertions follow by Proposition~\ref{epex}
\end{proof}

\begin{notations}\label{epsilon}
For a polynomial $H(t)\Q[t]$ we define $\epsilon_H$ as given in Lemma~\ref{positive}. 
Following Definitions~\ref{5.4} and Definition~\ref{polye}, let 
$$\sS_{gp} = \{H_{k_1, k_2}^{(n_i)}(t)\mid 0\leq n_i\leq n-3,~~0\leq k_1, k_2\leq n+1\}
\cup \{F_i(t), G_i(t)\mid 0\leq i\leq n+1\}$$ 
and let $\epsilon_1  = \min \{\epsilon_H\mid H\in \sS_{gp}\}$.

Similarly let 
$${\tilde \sS}_{gp} = \{{\tilde H}_{k_1, k_2}^{(n_i)}(t)\mid 
0\leq n_i\leq n-3,~~0\leq k_1, k_2\leq n+1\}
\cup \{{\tilde F}_0(t), \ldots, {\tilde F}_{n+2}(t)\}$$
and let $\epsilon_2 = \min \{\epsilon_H \mid H\in {\tilde \sS}_{gp}\}$.
\end{notations}

\begin{lemma}\label{stin} If $n\geq 3$ is odd number and   $p \geq 1/\epsilon_1$. Then 
\begin{enumerate}
\item[(o1)] for $(n_1, \ldots, n_l)\in \cup_{l\geq 1}\sM^l$ 
$$\int_{I^{(p)}_{(n_1, \ldots, n_l)}}\eell^{(p)}_{i}(x)dx\quad\mbox{and}\quad
 \int_{J^{(p)}_{(n_1, \ldots, n_l)}}\eell^{(p)}_{i}(x)dx\quad\mbox{
 are decreasing functions of}\quad p.$$ 
\item[(o2)] Moreover, for each  $n_1 \in \{0, \ldots, n-3\}$ 
  $$\int_{I^{(p)}_{(n_1)}}\eell^{(p)}_{n}(x)dx = 
\int_{I^{(p)}_{(n_1)}}\mmu^{(p)}_{n_0-1}(x)dx\quad\mbox{ 
is a strictly decreasing function of}~~p.$$ 
\end{enumerate}

 If $n\geq 4$ is an even integer and  $p \geq 1/\epsilon_2$. Then 
\begin{enumerate}
\item[(e1)] for $(n_1, \ldots, n_l)\in {\tilde \sM}_0\cup {\tilde \sM}_1$  
$$\int_{{\tilde I}^{(p)}_{(n_1, \ldots, n_l)}}\eell^{(p)}_{i}(x)dx~~\mbox{is 
a  decreasing function of}~~ p.$$ 
\item[(e2)]
Moreover, for  each  $n_1 \in \{\tfrac{n}{2}-1, \ldots, n-3\}$ 
  $$\int_{{\tilde I}^p_{(n_1)}}\mmu^{(p)}_{n_0-1}(x)dx\quad\mbox{ 
is a strictly decreasing function of}~~p.$$ 
\end{enumerate}
\end{lemma}

\begin{proof} Assertion~(o1) and Assertion~(e1) follow from Lemma~\ref{positive} and
Lemma~\ref{integrals}.

\vspace{5pt}
\noindent{\underline{Assertion}}~(o2).\quad We recall that 
$\mmu_{n_0}^{(p)}(x) = \eell_{n+1}^{(p)}(x)$. Therefore, by Proposition~\ref{pex}, 
 for $p\geq 3n-4$, we have 
\begin{equation}\label{o2}\int_{I^{(p)}_{(n_1)}}\mmu^{(p)}_{n_0}(x)dx = 
\sum_{i=0}^{n+1}t\Bigl[{\HH}^{(n_1)}(t)\Bigr]_{i,n+1} F_{i}(t)\Bigr\rvert_{t=1/p}
= \sum_{i=0}^{n+1}tH^{(n_1)}_{i,n+1}(t) F_{i}(t)\Bigr\rvert_{t=1/p},\end{equation}
where  $n_1\in \{0, \ldots, n-3\}$,

  By Lemma~\ref{positive}, it is sufficient to prove 
$tH^{(n_1)}_{n-1,n+1}(t) F_{n-1}(t) \neq 0$, which 
follows as we have 
 $$H^{(n_1)}_{n-1, n+1}({1}/p) = {b_{n-1, n+1}^{(n_1)}}/{p^{n}}
= {\mu_{-n_0}(j_1-n+1)}/{p^{n}} > 0,$$
where the last inequality holds  as, by Lemma~\ref{*c1}, 
the condition $0\leq j_1-n+1 \leq m_0-2$ gives  
 $\mu_{-n_0}(j_1-n+1)/p^{n} > 0$.
Also
 $$ F_{n-1}(\tfrac{1}{p}) =  
\int_0^{\tfrac{1}{2}-\tfrac{n-2}{2p}}{\bf l}_{n-1}(x)dx =
\int_0^{\tfrac{1}{2}-\tfrac{n-2}{2p}}\frac{2}{n!}(1-x)^ndx > 0$$
as  $p > n-2$.
\vspace{5pt}

\noindent{\underline{Assertion}}~(e2).\quad
Now consider $n_1\in \{\tfrac{n}{2}-1, \ldots, n-3\}$ then 
 following the notations of Proposition~\ref{lp}
$${\tilde F}_0(1/p) = \int_{\tfrac{n-2}{2p}}^{1-\frac{n-2}{2p}}{\bf Z}_0(x)dx
= \frac{2}{(n+1)!}\left[\left(1-\tfrac{n-2}{2p}\right)^{n+1}- 
\left(\tfrac{n-2}{2p}\right)^{n+1}\right] >0.$$

Also, by Lemma~\ref{e*c1}, the condition $j_1 = p+\tfrac{n}{2}-2-n_1 
= {\tilde m_0}+n-3-n_1 \geq {\tilde m_0}$
gives  
 $${\tilde H}^{(n_1)}_{0,n}(1/p) = 
{c_{0n}^{(n_1)}}/{p^{n}} = 
\mu_{-n_0+1}(j_1)/p^{n} > 0.$$

Whereas, 
by Proposition~\ref{epex}, for $p\geq (3n-4)/2$ 
\begin{equation}\label{e2}\int_{{\tilde I}^{(p)}_{(n_1)}}\mmu^{(p)}_{n_0-1}(x)dx = 
\sum_{i=0}^{n+2}\left(t{\tilde H}_{i,n}^{(n_1)}(t)
{\tilde F}_{i}(t)\right)\mid_{t= 1/p}.\end{equation}

Hence, by the same logic as above for $n$ equal to odd case, this integral
is a strictly increasing function of $t$ for $t\in [0, \epsilon_2)$, where $\epsilon_2$
is as in the above notations.
\end{proof}

\begin{rmk}\label{strict}By Definition~\ref{polye}
 the polynomials 
$F_i(t)\bigr\rvert_{1/p}$, $H_{k_1,k_2}^{(n_i)}(t)\bigr\rvert_{1/p}$ are nonnegative for 
$p\geq 3n-4$ and 
${\tilde F}_i(t)\bigr\rvert_{1/p}$, ${\tilde H}_{k_1,k_2}^{(n_i)}(t)\bigr\rvert_{1/p}$ 
are nonnegative for $p\geq (3n-4)/2$.
Therefore (\ref{o2}) and (\ref{e2}) imply that 
\begin{enumerate}
\item if $n\geq 3$ odd and $p\geq 3n-4$ then 
$\int \mmu_{n_0}^{(p)}(x)dx >0$.

\item If $n\geq 4$ even and $p\geq (3n-4)/2$ then 
$\int \mmu_{n_0-1}^{(p)}(x)dx >0$.
\end{enumerate}

Therefore   there exists interval $I_1\subset [0, 1)$ 
such that  the rank function  
$\mmu_{n_0}^{(p)}$ are strictly positive on $I_1$ in case $n$ is odd. Similar
assertion holds for  $n$ even.
\end{rmk}

\begin{thm}\label{pinfty}If $n\geq 3$  then there exists $\epsilon >0$ such that
for $p\geq 1/\epsilon $, the HK multiplicity  
$e_{HK}(R_{p, n+1})$
of the ring $R_{p, n+1}$ is  a strictly decreasing function of $p$.

Moreover, following Notations~\ref{epsilon}, one can take  $\epsilon = \epsilon_1$
or  $\epsilon = \epsilon_2$ depending on the parity of $n$.
\end{thm}
\begin{proof} Suppose $n\geq 3$ is an odd number. 
Then
\begin{equation}\label{expehk}e_{HK}(R_{p, n+1}) = 
\int_0^{\infty}f_{R_p, n+1}(x)dx 
= \int_0^{\infty}f^{\infty}_{R_{n+1}}(x)dx +
 \int_{\frac{1}{2}-\frac{n-2}{2p}}^{\frac{1}{2}+\frac{(n-2)}{2p}}
\mmu^{(p)}_{n_0}(x)dx,\end{equation}
 where the first equality follows from Theorem~1.1 of [T1] and the second from  
Proposition~\ref{l4}. Also by Proposition~\ref{l4},   the function 
$\mmu^{(p)}_{n_0}$ is a  real valued continuous function on the 
interval $[0, 1)$. Hence  by Lemmas~\ref{disjoint} and~\ref{l1c}  

\begin{equation}\label{expehk1}\int^{\frac{1}{2}+\frac{(n-2)}{2p}}_{\frac{1}{2}-\frac{n-2}{2p}}\mmu^{(p)}_{n_0}(x)dx = 
\sum_{{\underline n}\in \cup_{l\geq 1}\sM^l}\int_{I^{(p)}_{\underline n}}
\mmu^{(p)}_{n_0}(x)dx + 
\sum_{{\underline n}\in \cup_{l\geq 1}\sM^l} \int_{J^{(p)}_{\underline n}}\mmu^{(p)}_{n_0}(x)dx.\end{equation}

Similarly, when $n\geq 4$ is even then   
\begin{multline}\label{expehk2}
e_{HK}(R_{p, n+1})  = \displaystyle{  \int_0^{\infty}
f^{\infty}_{R_{n+1}}(x)dx  + 
\int_{\frac{1}{2}-\frac{(n-2)}{2p}}^1\mmu^{(p)}_{n_0-1}(x)dx+
 \int_0^{\frac{(n-2)}{2p}}
\mmu^{(p)}_{n_0+1}(x)dx},\\
 =  \displaystyle{\int_0^{\infty}
f^{\infty}_{R_{n+1}}(x)dx  + \sum_{{\underline n}\in 
{\tilde \sM}_1}\int_{{\tilde I}^{(p)}_{\underline n}}
\mmu^{(p)}_{n_0-1}(x)dx+\sum_{{\underline n}\in {\tilde \sM}_0}\int_{{\tilde I}^{(p)}_{\underline n}}
\mmu^{(p)}_{n_0+1}(x)dx},\end{multline}

where the last equality follows  by  Lemmas~\ref{disjoint} and \ref{emeasure}.
Hence theorem follows by  Lemma~\ref{stin}.
\end{proof}

In fact for $p\gg 0$, the HK multiplicity $e_{HK}(R_{p, n+1})$ can be expressed 
as a quotient of two rational polynomials.
For this first we  recall the Cofactor Formula for the Inverse of a Matrix.

Let $R$ be  a commutative ring with $1$.
$$\mbox{Let}\quad {\B}\in M_{n+2}(R)\quad\mbox{and let}\quad {\B_c} = 
{[(-1)^{i+j}\det \B_{ji}]}_{i,j}
 \in M_{n+2}(R),$$ 
where ${\B}_{ij}\in M_{n+1}(R)$ 
is obtained from $\B$ by 
deleting $i^{th}$ row and $j^{th}$ column.
 Then ${\B}{\B_c} = {\B_c}{\B} = (\det {\B}){\I}_{n+2}$. 

In particular $\det {\B}$ is a unit in $R$ 
if and only if ${\B}$ is unit in $M_{n+2}(R)$ and in that case we have 
${\B}^{-1} = \frac{1}{\det {\B}}\cdot {\B_c}\in M_{n+2}(R)$.

\begin{notations}\label{nodd} 
Here  $F_i(t)$ and $G_i(t)$ are polynomials in $\Q[t]$ as in Proposition~\ref{lp}, and
 ${\HH}^{(n_i)}(t)\in M_{n+2}(\Q[t])$ 
as in Definition~\ref{polye}.
We define 
$${\B}(t) =  t\cdot {\HH}^{(0)}(t) + t\cdot{\HH}^{(1)}(t)+\cdots + 
t\cdot{\HH}^{(n-3)}(t),$$ 

Also for any matrix $\B$, we   will denote  
$\frac{1}{\det {\B}}\cdot {\B_c}$ by  $\frac{\B_c}{\det {\B}}$.
\end{notations}

\begin{thm}\label{ehkexpr} Let $n\geq 3$ be an odd integer then let $p\geq 5$ and  
$p> 2^{\lfloor n/2\rfloor}(n-2)$. Then
\begin{multline*}
e_{HK}(R_{p, n+1}) = (1+m_{n+1})
+ (n+2)^{th}~~\mbox{entry of the matrix}\\\
\bigl[F_0(t)+G_0(t), \ldots, 
F_{n+1}(t)+G_{n+1}(t)\bigr]_{1\times n+2}\cdot 
{\B}(t)\cdot \frac{[\I_{n+2}-\B(t)]_c}{\det ({\I}_{n+2}-{\B}(t))}\Bigr\rvert_{t=1/p},
\end{multline*}
where $[\I_{n+2}-{\B}(t)]_c$ is the cofactor matrix of ${\I}_{n+2}-{\B}(t)$.
\end{thm}
\begin{proof}By (\ref{expehk}) and (\ref{expehk1})
$$e_{HK}(R_{p, n+1}) 
= \int_0^{\infty}f^{\infty}_{R_{n+1}}(x)dx + 
\sum_{{\underline n}\in \cup_{l\geq 1}\sM^l}\int_{I^{(p)}_{\underline n}}
\mmu^{(p)}_{n_0}(x)dx + 
\sum_{{\underline n}\in \cup_{l\geq 1}\sM^l} 
\int_{J^{(p)}_{\underline n}}\mmu^{(p)}_{n_0}(x)dx.$$

 If $p\geq 3n-4$ then by Lemma~\ref{integrals}~(\ref{pexe2})

\begin{multline*}
\sum_{{\underline n}\in \sM^l}\int_{I^{(p)}_{\underline n}}
\mmu^{(p)}_{n_0}(x)dx = 
(n+2)^{th}~~\mbox{entry of}~~ \\\
\left[F_0(t), \ldots, 
F_{n+1}(t)\right]_{1\times n+2}\cdot 
\left(t\cdot{\HH}^{(0)}(t) + t\cdot{\HH}^{(1)}(t)+\cdots + 
t\cdot{\HH}^{(n-3)}(t)\right)^l
\Bigr\rvert_{t= 1/{p}}\end{multline*}
Therefore
$$\sum_{{\underline n}\in \cup_{l\geq 1}\sM^l}\int_{I^{(p)}_{\underline n}}
\mmu^{(p)}_{n_0}(x)dx = 
(n+2)^{th}~~\mbox{entry of}~~ \left[F_0(t), \ldots, 
F_{n+1}(t)\right]_{1\times n+2}\cdot \sum_{l\geq 1}{\B}(t)^l\Bigr\rvert_{t= 
1/{p}}.$$

Similarly,  by Lemma~\ref{integrals}~(\ref{pexe4})
$$\sum_{{\underline n}\in \cup_{l\geq 1}\sM^l}\int_{J^{(p)}_{\underline n}}
\mmu^{(p)}_{n_0}(x)dx = 
(n+2)^{th}~~\mbox{entry of}~~ \left[G_0(t), \ldots, 
G_{n+1}(t)\right]_{1\times n+2}\cdot \sum_{l\geq 1}{\B}(t)^l\Bigr\rvert_{t= 
1/{p}}.$$ 

Hence to prove the theorem it only remains to  prove the following 

\vspace{5pt}

\noindent{{\bf Claim}}.\quad $\det(\I_{n+2}-\B(1/p))\neq 0$ if 
$p> 2^{\lfloor n/2\rfloor}(n-2)$ and $p\geq 5$.
\vspace{5pt}

\noindent{Proof of the claim}:\quad We fix a prime $p$ as in the claim, and 
 consider the $l^1$ norm on $\Q^{n+2}$, that is 
for $v = (v_0, \ldots, v_{n+1}) \in \Q^{n+2}$ the $l^1$ norm of $v$ is given by  
$\|v\| = |v_0|+\cdots |v_{n+1}|$.

Note the-t if the claim does not hold true then there exists 
a nonzero vector $v\in \Q^{n+2}$ such that
$v\cdot (\I_{n+2}-\B(1/p)) = 0$, which implies  $\|v\| = \|v\cdot \B(1/p)\|$.

Now we show that  $l^1$ norm of the matrix $\B(1/p)$ is $< 1$, that is   
$\|v\cdot \B(1/p)\| <\|v\|$, for any nonzero vector $v$.
We recall
$$\B(1/p) = {1}/{p}\sum_{i=0}^{n-3}\HH^{(n_i)}(1/p) = {1}/{p}
\sum_{i=0}^{n-3}{\B^{(p)}_{n_i}}/{p^n}.$$
This gives 
 $$\|v\cdot \B(1/p)\| \leq \frac{1}{p}\sum_{i=0}^{n-3}\|v\cdot 
\frac{\B^{(p)}_{n_i}}{p^n}\| \leq\frac{1}{p}\sum_{i=0}^{n-3}
\sum_{j=0}^{n+1}|v_j|\frac{b_{j0}^{(n_i)}(p)+\cdots + b_{j,n+1}^{(n_i)}(p)}{p^n},$$ 
where $(b_{j0}^{(n_i)}(p), b_{j1}^{(n_i)}(p), \ldots, b_{j,n+1}^{(n_i)}(p))$ 
denotes the $j^{th}$ row of the  matrix $\B^{(p)}_{n_i}$ consisting of 
nonnegative integers.

On the other hand, for any integer $m$,   $\rank(F_*(\sO(m))) = p^n$ and 
$\rank(F_*(\sS(m))) = (2^{\lfloor n/2\rfloor})p^n$, which implies, following 
Notations~\ref{Bmat}, that

$$b_{j0}^{(n_i)}(p) + b_{j1}^{(n_i)}(p)+
\cdots + b_{j,n+1}^{(n_i)}(p)
\leq 2^{\lfloor n/2\rfloor} p^n~~~ \text{for every}~~~ 0\leq j\leq n+1.$$

Therefore 
$$\|v\cdot \B(1/p)\| \leq \frac{1}{p}\sum_{i=0}^{n-3}\sum_{j=0}^{n+1}|v_j|
2^{\lfloor n/2\rfloor} 
\leq \frac{2^{\lfloor n/2\rfloor}(n-2)}{p}\|v\| < \|v\|.$$
\end{proof}

\begin{notations}\label{neven}
Here  ${\tilde F}_i(t)$ are as in Proposition~\ref{lp} and 
where, for ${\tilde {\HH}}^{(n_i)}(t)$ as in Definition~\ref{polye}, we define
$$\begin{array}{lcl}
{\C}(t) & =  & t\cdot{\tilde {\HH}}^{(0)}(t) + {t}\cdot{\tilde {\HH}}^{(1)}(t)
+\cdots +
t\cdot{\tilde {\HH}}^{(n-3)}(t)\\ 
{\C}_1(t) & = &  t\cdot{\tilde {\HH}}^{(0)}(t) + 
{t}\cdot{\tilde {\HH}}^{(1)}(t)+\cdots +
t\cdot{\tilde {\HH}}^{(\frac{n}{2}-2)}(t). \end{array}$$
\end{notations}

\begin{thm}\label{eehkexpr}
Let $n\geq 4$ be an even integer  then for $p > 2^{\lfloor n/2\rfloor} (n-2)$
\begin{multline*}
e_{HK}(R_{p, n+1}) = (1+m_{n+1})\\ 
+ (n+1)^{th}~~\mbox{entry of}~~ 
\left[{\tilde F}_0(t), \ldots, 
{\tilde F}_{n+2}(t)\right]_{1\times n+3}\cdot 
(\C(t)-\C_1(t))\cdot\frac{[\I_{n+2}-\C(t)]_c}{\det ({\I}_{n+2}-{\C}(t))}\Bigr\rvert_{t=1/p}\\\
+ (n+3)^{th}~~\mbox{entry of}~~
\left[{\tilde F}_0(t), \ldots, 
{\tilde F}_{n+2}(t)\right]_{1\times n+3}\cdot 
{\C}_1(t)\cdot\frac{[\I_{n+2}-\C(t)]_c}{\det ({\I}_{n+2}-{\C}(t))}\Bigr\rvert_{t=1/p}.
\end{multline*}
\end{thm}
\begin{proof} By (\ref{expehk2})
$$e_{HK}(R_{p, n+1})  =  \int_0^{\infty}
f^{\infty}_{R_{n+1}}(x)dx +  \sum_{{\underline n}\in 
{\tilde \sM}_1}\int_{{\tilde I}^{(p)}_{\underline n}}
\mmu^{(p)}_{n_0-1}(x)dx+\sum_{{\underline n}\in {\tilde \sM}_0}
\int_{{\tilde I}^{(p)}_{\underline n}}
\mmu^{(p)}_{n_0+1}(x)dx,$$
where  by Definition~\ref{deven} we have 
 ${\tilde \sM}_0 = \{(n_1, \ldots, n_l)\in \sM^l\mid  n_1 \leq 
\tfrac{n}{2}-2\}_{l\in \N}$ and 
${\tilde \sM}_1 = \{(n_1, \ldots, n_l)\in \sM^l \mid  n_1 
\geq \tfrac{n}{2}-1\}_{l\in \N}$.
Therefore, by  Lemma~\ref{integrals}~(\ref{pexe4})
\begin{multline*}\sum_{{\underline n}\in 
{\tilde \sM}_1}\int_{{\tilde I}^{(p)}_{\underline n}}
\mmu^{(p)}_{n_0-1}(x)dx = 
 (n+1)^{th}~~\mbox{entry of}\\ 
\left[{\tilde F}_0(t), \ldots, 
{\tilde F}_{n+2}(t)\right]_{1\times n+3}\cdot 
\sum_{l\geq 0}{\C}(t)^l\cdot t\cdot\left({\tilde {\HH}}^{(\tfrac{n}{2}-1)}(t) + 
\cdots + {\tilde {\HH}}^{(n-3)}(t)\right)\Bigr\rvert_{t= 1/{p}}.\end{multline*}
and 
\begin{multline*}\sum_{{\underline n}\in 
{\tilde \sM}_0}\int_{{\tilde I}^{(p)}_{\underline n}}
\mmu^{(p)}_{n_0+1}(x)dx = 
 (n+3)^{th}~~\mbox{entry of}\\
 \left[{\tilde F}_0(t), \ldots, 
{\tilde F}_{n+2}(t)\right]_{1\times n+3}\cdot 
\sum_{l\geq 0}{\C}(t)^l\cdot t\cdot \left({\tilde {\HH}}^{(0)}(t) + 
{\tilde {\HH}}^{(1)}(t)+\cdots +
{\tilde {\HH}}^{(\frac{n}{2}-2)}(t)\right)\Bigr\rvert_{t= 1/{p}}.\end{multline*}

Now, again we can show that $\det(\I_{n+3}-\C(1/p))\neq 0$ as, following 
Notations~\ref{Bmat}, 
for any nonzero vector $v\in \Q^{n+3}$
 $$\|v\cdot \C\bigl(\frac{1}{p}\bigr)\| \leq \sum_{i=0}^{n-3}\|v\cdot 
\frac{\C^{(p)}_{n_i}}{p^{n+1}}\| \leq \sum_{i=0}^{n-3}
\sum_{j=0}^{n+2}|v_j|\frac{c_{j0}^{(n_i)}(p)+\cdots + c_{j,n+2}^{(n_i)}(p)}{p^{n+1}}
\leq \frac{2^{\lfloor {n}/{2}\rfloor}(n-2)}{p}\|v\|,$$ 
where $(c_{j0}^{(n_i)}(p), c_{j1}^{(n_i)}(p), \ldots, c_{j,n+2}^{(n_i)}(p))$ 
denotes the $j^{th}$ row of the  matrix $\C^{(p)}_{n_i}$ and therefore 
 $$c_{j0}^{(n_i)}(p)+ c_{j1}^{(n_i)}(p)+ \cdots + c_{j,n+2}^{(n_i)}(p)) \leq 
2^{\lfloor {n}/{2}\rfloor} p^n, \quad\mbox{for}\quad 0\leq j\leq n+2.$$
This proves the theorem.\end{proof}

\begin{rmk}\label{rodd}In the above proof we have also proved the following.
\begin{enumerate}
\item[] If  $n\geq 3$ is an odd integer and $p\geq 3n-4$ then 

\begin{multline*}e_{HK}(R_{p, n+1})  =  (1+m_{n+1})
+(n+2)^{th}~~\mbox{entry of}~~ \\\
  \left[F_0(t)+G_0(t), \ldots, F_{n+1}(t)+G_{n+1}(t)\right]_{1\times n+2}
\cdot \sum_{l\geq 1}{\B}(t)^l \Bigr\rvert_{t=1/p},\end{multline*}

\item[]
If $n\geq 4$ is an even integer and $p\geq (3n-4)/2$ then 

\begin{multline*}
e_{HK}(R_{p, n+1})  = 1+m_{n+1}+ (n+1)^{th}~~\mbox{entry of}~
\Bigl[{\tilde F}_0(t), \ldots, {\tilde F}_{n+2}(t)\Bigr]
\cdot {\C}_1(t)\cdot
\sum_{l\geq 0}{\C}(t)^l\Bigr\rvert_{t=1/p}\\\
+ (n+3)^{th}~~\mbox{entry of}~
\Bigl[{\tilde F}_0(t), \ldots, {\tilde F}_{n+2}(t)\Bigr]
\cdot ({\C}(t)-{\C}_1(t))\cdot
\sum_{l\geq 0}{\C}(t)^l\Bigr\rvert_{t=1/p}.\end{multline*}
\end{enumerate}
\end{rmk}

\section{The HK density function for $R_{p, 4}$}

\begin{notations}\label{n4}
Let $n =3$ and let $p\geq 5$ be a prime. Here $n_0=1$.

Following the notations as in Definition~\ref{nch}
we describe the almost cover  of the difficult range  
$[\tfrac{1}{2}-\tfrac{1}{2p},~~~\tfrac{1}{2})\cup 
[\tfrac{1}{2},~~~\tfrac{1}{2}+\tfrac{1}{2p})$: Here $\sM = \{0\}$
and  $m_0 = \frac{p-1}{2}$.
Let $I^{(p)}(l) = I^{(p)}_{(0, \ldots l~\mbox{times})}$ 
and $J^{(p)}(l) = J^{(p)}_{(0, \ldots l~\mbox{times})}$.
Then 
$$I^{(p)}(l) = 
[\sum_{i=1}^l\tfrac{m_0}{p^i}, \quad \sum_{i=1}^{l+1}
\tfrac{m_0}{p^i}) = \left[\tfrac{1}{2}-\tfrac{1}{2p^l},\quad
\tfrac{1}{2}-\tfrac{1}{2p^{l+1}}\right)$$
and 
$$J^{(p)}(l) = \left[\tfrac{m_0+1}{p^{l+1}}+\sum_{i=1}^l\tfrac{m_0}{p^i}, \quad 
\sum_{i=1}^{l}\tfrac{m_0}{p^i}+\tfrac{p}{p^{l+1}}\right) = 
\left[\tfrac{1}{2}+\tfrac{1}{2p^{l+1}},\quad
\tfrac{1}{2}+\tfrac{1}{2p^{l}}\right).$$

Hence it follows that  
$\{I^{(p)}(l), J^{(p)}(l)\}_{l\in \N}$ cover the difficult range 
everywhere  except at the point $\tfrac{1}{2}$.
Therefore
 $$[\tfrac{1}{2} - \tfrac{1}{2p},~~\tfrac{1}{2}+ \tfrac{1}{2p}) 
= \bigcup_{l\in \N}I^{(p)}(l) 
\cup \{\tfrac{1}{2}\}\cup\bigcup_{l\in \N}J^{(p)}(l)$$
\end{notations}

\vspace{5pt}

For $0\leq a <p$ we have 
$$F_*(\sO(a)) = \sO^{\nu^1_{0}(a)}\oplus \sO(-1)^{\nu^1_{-1}(a)}
\oplus\sO(-2)^{\nu^1_{-2}(a)}\oplus\sS^{\mu^1_{0}(a)}\oplus\sS(-1)^{\mu^1_{-1}(a)}$$
and 
$$F_*(\sS(a)) = \sO^{{\tilde \nu}^1_{0}(a)}\oplus \sO(-1)^{{\tilde \nu}^1_{-1}(a)}
\oplus\sO(-2)^{{\tilde \nu}^1_{-2}(a)}\oplus\sS^{{\tilde \mu}^1_{0}(a)}
\oplus\sS(-1)^{{\tilde \mu}^1_{-1}(a)}.$$

Following Notations~\ref{n2} and \ref{n2s} we have 
$$L_a= \frac{1}{6}(2a^3+9a^2+13a+6)\quad\mbox{and}\quad 
{\tilde L_a} = \frac{2}{3}(a^3+3a^2+2a).$$

By Lemma~\ref{*c1}, we get 
$$\begin{array}{lcl}
\mu_{-1}^1(a) & = & \frac{1}{4}\left[L_{a+2p}-L_1L_{a+p}+
(L_1^2-L_2)L_a -L_{p-a-3}\right]\neq  0, \quad\mbox{if}\quad 
a\leq m_0-2\\
 & = & 0\quad\mbox{otherwise}\\\
\mu_{0}^1(a) & = & -\frac{1}{4}\left[L_{a+2p}-L_1L_{a+p}+
(L_1^2-L_2)L_a - L_{p-a-3}\right]\neq 0, 
\quad\mbox{if}\quad 
m_0\leq a\\
 & = & 0\quad\mbox{otherwise}.\end{array}$$

$$\begin{array}{lcl}
{\tilde \mu}_{-1}^1(a) & = & \frac{1}{4}\left[{\tilde L}_{a+2p}-
L_1{\tilde L}_{a+p}+(L_1^2-L_2){\tilde L}_a -{\tilde L}_{p-a-3}\right]\neq 0, 
\quad\mbox{if}\quad 
a\leq m_0-1\\
 & = & 0\quad\mbox{otherwise}\\\
{\tilde \mu}_{0}^1(a) & = & \frac{-1}{4}\left[{\tilde L}_{a+2p}-
L_1{\tilde L}_{a+p}+(L_1^2-L_2){\tilde L}_a - {\tilde L}_{p-a-3}\right]\neq 0, 
\quad\mbox{if}\quad 
m_0 \leq a\\
 & = & 0\quad\mbox{otherwise}.\end{array}$$

In particular 

\begin{equation}\label{emu0}\mu_0 := \mu^1_0(m_0) 
= \frac{1}{4}\left[L_{\tfrac{p-5}{2}}- L_{\tfrac{5p-1}{2}}
+ L_1L_{\tfrac{3p-1}{2}} - (L_1^2-L_2)Y_{\tfrac{p-1}{2}} \right]\neq 0,\end{equation}

 \begin{equation}\label{emu1}\mu_{-1} := \mu^1_{-1}(m_0-2) = \frac{1}{4}
\left[L_{\tfrac{5p-5}{2}}
-L_1L_{\tfrac{3p-5}{2}} + (L_1^2-L_2)Y_{\tfrac{p-5}{2}} -L_{\tfrac{p-1}{2}}\right]
\neq 0 \end{equation}

\begin{equation}\label{emut0}{\overline \mu_0} := {\tilde \mu}^1_0(m_0) 
= \frac{1}{4}\left[{\tilde L}_{\tfrac{p-5}{2}}- {\tilde L}_{\tfrac{5p-1}{2}}
+ L_1{\tilde L}_{\tfrac{3p-1}{2}} - (L_1^2-L_2){\tilde L}_{\tfrac{p-1}{2}} \right]
\neq 0,\end{equation}

 \begin{equation}\label{emut1}{\overline  \mu_{-1}} := {\tilde \mu^1}_{-1}(m_0-1) = 
\frac{1}{4}\left[{\tilde L}_{\tfrac{5p-3}{2}}
-L_1{\tilde L}_{\tfrac{3p-3}{2}} + (L_1^2-L_2){\tilde L}_{\tfrac{p-3}{2}} 
-{\tilde L}_{\tfrac{p-1}{2}}\right]\neq 0.\end{equation}
Hence $\mu_0$, $\mu_{-1}$, ${\overline  \mu_{0}}$ and ${\overline  \mu_{-1}}$ 
are all positive integers.

\begin{thm}\label{t3}Let $k$ be a perfect field of characteristic $p\geq 5$ and let   
$$R_{p, 4} = \frac{k[x_0, x_1, x_2, x_3, x_4]}{(x_0^2+x_1^2+ x_2^2 + x_3^2 + x_4^2)}$$ 
and let  $\mu_0$, ${\bar \mu}_0$, $\mu_{-1}$ and ${\bar \mu}_{-1}$ be the  
 positive integers  given as above. Then 
$$f_{R_{p, 4}, {\bf m}}(x)  = \begin{cases}x^3/3 &  \quad\mbox{for}\quad  
0\leq x < 1\\\\
 x^3/3 -5/3(x-1)^3 & \quad\mbox{for}\quad  1\leq x < 2\\\\
 \frac{1}{3}x^3 -\frac{5}{3}(x-1)^3+\frac{11}{3}(x-2)^3 & \quad\mbox{for}\quad  
2\leq x < 2+\tfrac{1}{2}-\tfrac{1}{2p}\\\\
\frac{(3-x)^3}{3} & \quad\mbox{for}\quad 
2+\tfrac{1}{2}+\tfrac{1}{2p} \leq x <3\\\\
0 & \quad\mbox{for}\quad x\geq 3,
\end{cases}$$
and for $x\in 2+ \left[\tfrac{1}{2}-\tfrac{1}{2p},\quad
\tfrac{1}{2}+\tfrac{1}{2p}\right)$

\begin{multline*}
f_{R_{p, 4}, {\bf m}}(x) =  \tfrac{1}{3}x^3 -\tfrac{5}{3}(x-1)^3+
\tfrac{11}{3}(x-2)^3
+ \tfrac{4}{3} \sum_{i = 1}^{j} \left[x-2-\tfrac{1}{2}+\tfrac{1}{2p^i}\right]^3
(\mu_0({\overline {\mu}}_0)^{i-1})\\
 \mbox{if}~~~~ x\in 2+ \left[\tfrac{1}{2}-\tfrac{1}{2p^j},\quad
\tfrac{1}{2}-\tfrac{1}{2p^{j+1}}\right)~~~\mbox{ and}~~~j\geq 1\end{multline*}
and 
\begin{multline*}f_{R_{p, 4}, {\bf m}}(x)   =   
 \tfrac{(3-x)^3}{3}+ 
\tfrac{4}{3} \sum_{i = 1}^{j} \left[2+\tfrac{1}{2}+\tfrac{1}{2p^i}-x\right]^3
(\mu_{-1}{\overline {\mu_{-1}}}^{i-1}),\\
\mbox{if}\quad x\in 2+\left[\tfrac{1}{2}+\tfrac{1}{2p^{j+1}},\quad
\tfrac{1}{2}+\tfrac{1}{2p^{j}}\right)\quad\mbox{and}\quad  j\geq 1.\end{multline*}
\end{thm}
\begin{proof}
By Proposition~\ref{l4}, 
 for $n_0=1$ we get 
$$ f_{R_{p, 4}}(x)  = \begin{cases} {\bf Z}_{0}(x)   \quad\mbox{for}\quad 
 0\leq x < 1\\\\
{\bf  Z}_{1}(x) \quad\mbox{for}\quad  1\leq x < 2\\\\ 
 {\bf Z}_{2}(x) \quad\mbox{for}\quad  
2\leq x < 2+\tfrac{1}{2}-\tfrac{1}{2p}\\\\
 {\bf Z}_{2}(x) + 
\mmu^{(p)}_{0}(x-2)  \quad\mbox{for}\quad  
2+\tfrac{1}{2}-\tfrac{1}{2p}
\leq x < 2+\tfrac{1}{2} \\\\
 {\bf Y}_{2}(x) +\mmu^{(p)}_{1}(x-2)
\quad\mbox{for}\quad  
2+\tfrac{1}{2} \leq x <  2+\tfrac{1}{2}+\tfrac{1}{2p}\\\\
  {\bf Y}_{2}(x) 
\quad\mbox{for}\quad  
2+\tfrac{1}{2} + \tfrac{1}{2p}\leq x < 3\end{cases}$$
and $f_{R_{p,4}}(x) = 0$ otherwise.

Following Notations~\ref{n2} and Remark~\ref{r3} we get 
$${\bf Z}_0(x) = \frac{x^3}{3},\quad\quad 
{\bf Z}_{1}(x) = \frac{x^3}{3} -\frac{5}{3}(x-1)^3$$
$${\bf Z}_{2}(x) =  \frac{x^3}{3} -\frac{5}{3}(x-1)^3 +\frac{11}{3}(x-2)^3\quad
\mbox{and}\quad
{\bf Y}_{2}(x) = \frac{(3-x)^3}{3}.$$
So it remains to compute $\mmu^{(p)}_{0}(x)$ for $x\in I^{(p)}(j)$ and 
$\mmu^{(p)}_{1}(x)$ for $x\in J^{(p)}(j)$, for $j\geq 1$.

Here instead of using the formula as stated in Proposition~\ref{pex}, 
we used the formula~(\ref{ff1}) given in the proof of the proposition.

Now let  $x\in I^{(p)}(j)$ or $x\in J^{(p)}(j)$, then for $q\geq p^{j+1}$, we have the $p$-adic expansion
$$\lfloor xq\rfloor = a_0+a_1p+\cdots + a_{s-j-1}p^{s-j-1}+ m_0p^{s-j}+\cdots + 
m_0p^{s-1}.$$

Let   $A_{s-i} = a_0+a_1p+\cdots + a_{s-i-1}p^{s-i-1}$. Then 
$A_{s-j} = a_0+a_1p+\cdots + a_{s-j-1}p^{s-j-1}$, where 
$a_{s-j-1} < m_0$ if $x\in I^{(p)}(j)$ and 
$a_{s-j-1} > m_0$ if $x\in J^{(p)}(j)$

\vspace{5pt}

\noindent{\bf Claim}.\quad Let  $\mu_{0}$, $\mu_{-1}$, ${\overline \mu}_{0}$ and 
${\overline \mu}_{-1}$ be as given in the above notations.
Then 
\begin{enumerate}
\item  $x\in I^{(p)}(j)$ implies 
$$\mu^s_0(\lfloor xq\rfloor)  =   \nu_{0}^{s-1}(A_{s-1})(\mu_{0}) + 
\nu_{0}^{s-2}(A_{s-2})(\mu_0 {\overline {\mu_0}})
 +\cdots + \nu_{0}^{s-j}(A_{s-j})(\mu_0{\overline {\mu_0}}^{j-1}).$$
\item $x\in J^{(p)}(j)$ implies 
$$\mu_{-1}^s(\lfloor xq\rfloor)   = \nu_{-2}^{s-1}(A_{s-1})(\mu_{-1}) + 
\nu_{-2}^{s-2}(A_{s-2})(\mu_{-1} {\overline {\mu_{-1}}})
 +\cdots + 
\nu_{-2}^{s-j}(A_{s-j})(\mu_{-1}{\overline {\mu_{-1}}}^{j-1}).$$
\end{enumerate}
\vspace{5pt}

\noindent{\underline {Proof of the claim}}:\quad
By (\ref{ff1}) in the  proof of Proposition~\ref{pex} we have 

\begin{multline*}
{[\nu^s_{0}(a), \nu^s_{-1}(a), \nu^s_{-2}(a),
\mu^s_{0}(a),\mu^s_{-1}(a)]}_{1\times 5}\\
=  {[\nu^{s-j}_{0}(A_{s-j}), \cdots, \mu^{s-j}_{-1}(A_{s-j})]}_{1\times 5}\cdot 
[{\B}_0^{(p)}]^{j~\mbox{times}},\end{multline*}
where ${\B}_0^{(p)}$ is the matrix given by   
$${\B}_0^{(p)} = \left[\begin{matrix}
\nu_{0}(m_0) & \nu_{-1}(m_0)  & \nu_{-2}(m_0) & \mu_0 & 0\\\\ 
\nu_{0}(m_0-1) & \nu_{-1}(m_0-1) & \nu_{-2}(m_0-1) & 0 & 0\\\\
\nu_{0}(m_0-2) & \nu_{-1}(m_0-2)  & \nu_{-2}(m_0-2) & 0  & \mu_{-1} \\\\ 
{\tilde \nu_{0}}(m_0) & {\tilde \nu_{-1}}(m_0) & {\tilde \nu_{-2}}(m_0) &
{\overline \mu_{0}} & 0\\\\
{\tilde \nu_{0}}(m_0-1) & {\tilde \nu_{-1}}(m_0-1) & {\tilde \nu_{-2}}(m_0-1) & 
0  & {\overline {\mu_{-1}}}.\end{matrix}\right].$$

Moreover, by Lemma~\ref{l3}, if   $x\in I^{(p)}(j)$ then 
$\mu_0^{s-j}(A_{s-j}) = 0$ and 
if $x\in J^{(p)}(j)$ then  $\mu_{-1}^{s-j}(A_{s-j}) = 0$ 

Suppose $x\in I^{(p)}(j)$. Then
\begin{multline*}[\nu^s_{0}(a), \nu^s_{-1}(a), \nu^s_{-2}(a),
\mu^s_{0}(a),\mu^s_{-1}(a)]\\
 =  [\nu^{s-j+1}_{0}(A_{s-j+1}),-,-, 
\nu_{0}^{s-j}(A_{s-j})\mu_0, \mu^{s-j+1}_{-1}(A_{s-j+1})]\cdot 
[{\B}_0^{(p)}]^{j-1}\\
= [\nu^{s-j+2}_{0}(A_{s-j+2}), -, -,  \nu_{0}^{s-j+1}(A_{s-j+1})\mu_{0}+
\nu^{s-j}_{0}(A_{s-j})\mu_{0}{\overline{\mu_{0}}}, \mu^{s-j+2}_{-1}(A_{s-j+2})]
\cdot [{\B}_0^{(p)}]^{j-2}.\end{multline*}
Iterating this $j$ times we get the first identity given in the claim.
The second identity follows similarly.

We recall that 
$L_a= \frac{1}{6}(2a^3+9a^2+13a+6) = a^3/3+O(a^2)$.
Hence 
$$\lim_{s\to \infty} \frac{\nu_0^{s-i}(A_{s-i})}{p^{3s}} =
\frac{1}{3}\left[x-\tfrac{1}{2}+\tfrac{1}{2p^i}\right]^3$$
and

$$\lim_{s\to \infty} \frac{\nu_{-2}^{s-i}(A_{s-i})}{p^{3s}} =
\lim_{s\to \infty} \frac{L_{p^{s-i}-(a-m_0(p^{s-i}+\cdots + p^{s-1})-3}}{p^{3s}} = 
\frac{1}{3}\left[\tfrac{1}{2}+\tfrac{1}{2p^i}-x\right]^3.$$

Hence, for $x\in I^{(p)}(j)$, 
$$\mmu^{(p)}_0(x) = 
\frac{4}{3} \sum_{i = 1}^{j} \left[x-\tfrac{1}{2}+\tfrac{1}{2p^i}\right]^3
(\mu_0{\overline {\mu_0}}^{i-1})$$

If $x\in J^{(p)}(j)$ then 
$$\mmu^{(p)}_{1}(x) = \frac{4}{3} \sum_{i = 1}^{j} 
\left[\tfrac{1}{2}+\tfrac{1}{2p^i}-x\right]^3
(\mu_{-1}{\overline {\mu_{-1}}}^{i-1}).$$

This proves the theorem.
\end{proof}

\begin{rmk}\label{future}
In the case of projective curves and  projective toric varieties
and therefore their arbitrary Segre products (see Proposition~2.17 in [T1])
the set of non smooth points of the HK density function is a finite set.

However,  the above formulation in Theorem~\ref{t3}, 
gives that  $f_{R_{p, 4}}$ is  a 
$\sC^2(\R)$ function everywhere but
its third derivative does not exist at any point in the set
$\{2+\tfrac{1}{2}-\tfrac{1}{2p^j}, ~~2+\tfrac{1}{2}-\tfrac{1}{2p^j}\mid j\in \N\}$.
In particular  the set of non smooth points of $f_{R_{p, 4}}$ is infinite and  
has an accumulation point.

It would be an interesting direction to investigate the class of rings 
 with   $\sC^{n-1}(\R)$  HK density function, where $n+1$ is the 
Krull dimension of the ring.
\end{rmk}

 \section{Appendix}

Proposition~5.2 of [L] states that 
if $s=1$ then $F_*(\sO(a))$ and $F_*(\sS(a))$ both have  at most one type of 
spinor bundle, 
This also follows from  the following result of [A] regarding the decomposition of 
 Frobenius image of spinor bundles. Unlike the line bundle $\sO(a)$, where
we set up the theory for $F^s_*(\sO(a))$ for all $s\geq 1$, 
here we need  the information only for  the first Frobenius bundle 
$F_*(\sS(a))$.

Consider the decomposition 
$$F_*(\sS(a))
 = \oplus _{t\in \Z}\sO(t)^{{\tilde \nu}_t(a)} \oplus \oplus_{t\in \Z}
\sS(t)^{{\tilde \mu}_t(a)}.$$

By Theorem~[A], for $q = p$, 
\begin{enumerate}
\item $F_*(\sS(a))$ contains $\sO(t)$ if and only if $1\leq a-tp \leq n(p-1)$.
\item $F_*(\sS(a))$ contains $\sS(t)$ if and only if
$$\left(\tfrac{(n-2)(p-1)}{2}\right) \leq a-tp \leq
\left(\tfrac{(n-2)(p-1)}{2} + n-2 +p\right) -n + 1.$$
\end{enumerate}

\begin{notations}\label{n2s} Fix an integer $n\geq 3$. Let $n_0 = \lceil n/2\rceil -1$.
  For an integer $a$ such that 
$1-n\leq  a <p$ 

We define
 $${\tilde L_a} = H^0(Q, \sS(a)) = 2\lambda_0\left[ \binom{a+n}{n+1}-\binom{a+n-1}{n+1} \right]= 
 2\lambda_0\frac{(a+n-1)\cdots a}{n!}.$$

  For $0\leq i \leq n_0+1$, we define iteratively 
a set of integers 
${\tilde Z}_0(a)$, ${\tilde Z}_{-i}(a, p)$  as follows:
$$ {\tilde Z}_0(a) = {\tilde L}_{a},\quad  
\quad {\tilde Z}_{-1}(a, p)) = {\tilde Z}_0(a+p)-L_1{\tilde Z}_0(a)$$
 and in general 
$${\tilde Z}_{-i}(a, p) = {\tilde Z}_0(a+ip) - 
\left[L_1{\tilde Z}_{-i+1}(a, p)+L_2{\tilde Z}_{-i+2}(a, p)
+\cdots +L_i {\tilde Z}_0(a)\right].$$
 
Similarly,  for $n_0+1\leq i \leq n-1$, 
we define iteratively another set of integers   ${\tilde L}_{-i}(a, p)$
 as follows:

 $${\tilde Y}_{-n+1}(a, p)) = {\tilde L}_{p-a-n+1}
\quad\mbox{and for}\quad n_0+1\leq i < n-1 $$ 
 $${\tilde  Y}_{-i}(a, p) = {\tilde L}_{(n-i)p-a-n+1} -
\left[L_{n-i-1}{\tilde Y}_{-n+1}(a, p)+
\cdots + L_1{\tilde L}_{-i-1}(a, p)\right].$$ 
\end{notations}

\begin{rmk}\label{r3s}
By construction, it is easy to check that 
that there exists rational numbers  $\{{\tilde r}_{ij}, {\tilde s}_{ik}\}_{j, k}$
such that for every pair $(a, p)$, where $0\leq a <p$,  

${\tilde Z}_0(a) = {\tilde L}_{a}$ and 

 $${\tilde Z}_{-i}(a, p) = {\tilde r}_{i0}{\tilde L}_{a} + 
{\tilde r}_{i1}{\tilde L}_{a+p}+ 
\cdots  + {\tilde r}_{i(i-1)}{\tilde L}_{a+(i-1)p}+{\tilde L}_{a+ip}$$
and 
$${\tilde Y}_{-i}(a, p) = \sum_{j=1}^{n-i-1}{\tilde s}_{ij}{\tilde L}_{jp-a-n+1} 
+{\tilde L}_{(n-i)p-a-n+1}.$$

On the other hand, $L_m$ is  a polynomial in $m$ of degree $n$
which implies, for given $i$  there exist polynomials $P_i(X, Y), Q_i(X, Y)
\in \Q[X, Y]$, each of total degree $\leq n$, such that  
  $$Z_{-i}(a, p) = P_i(a, p)\quad\mbox{and}\quad Y_{-i}(a,p) = Q_i(a, p)\quad\mbox{for all}
\quad 0\leq a <p.$$
Similar assertions holds for each ${\tilde Z_{-i}}$ and ${\tilde Y_{-i}}$,
\end{rmk}

\begin{lemma}\label{*c1}
Let  $p > n-2 $, where $n\geq 3$ is an odd integer
 and $0\leq a < p$ is an integer. Then
$$F_*(\sO(a)) = \sO(-n+1)^{\nu_{-n+1}(a)}\oplus \cdots \oplus 
 \sO^{\nu_0(a)}\oplus 
\sS(-n_0+1)^{\mu_{-n_0+1}(a)}\oplus 
\sS(-n_0)^{\mu_{-n_0}(a)}.$$ 

Let $m_0 = \frac{p}{2}-\frac{n-2}{2}$. 
Then 
\begin{enumerate}
\item $\nu_{-i}(a) = \begin{cases} 
Z_{-i}(a, p),\quad\mbox{for}\quad 0\leq i \leq n_0-1\\\ 
 Y_{-i}(a, p),\quad\mbox{for}\quad n_0+1\leq i < n \end{cases} $ 
\item $\mu_{-i}(a) = 0$, if $i\nin \{n_0-1, n_0\}$. 
\item Moreover  if  $a\in [0, m_0-1]$ then 
\begin{enumerate}
\item $\nu_{-n_0}(a) = Z_{-n_0}(a, p)$
\item $\mu_{-n_0+1}(a) = 0$  
\item $\mu_{-n_0}(a) = \frac{1}{2\lambda_0}\left[Z_{-n_0-1}(a, p)-
Y_{-n_0-1}(a, p)\right]$
and 
$$\mu_{-n_0}(a) \neq 0 \iff a\leq m_0-2.$$ 
\end{enumerate}
\item Similarly,  for $a\in [m_0, p)$ we have 
\begin{enumerate}
\item $\nu_{-n_0}(a) =  Z_{-n_0}(a, p)-
Y_{-n_0-1}(a, p)+ Z_{-n_0-1}(a, p)$,
\item $\mu_{-n_0+1}(a) = \frac{1}{2\lambda_0}\left[Y_{-n_0-1}(a, p)-
Z_{-n_0-1}(a, p)\right] \neq 0$.
\item $\mu_{-n_0}(a) = 0$. 
\end{enumerate}
\end{enumerate}

Also we have 

\begin{enumerate}
\item ${\tilde \nu}_{-i}(a) = \begin{cases}
{\tilde Z_{-i}}(a, p)$, for $0\leq i \leq n_0-1\\\
{\tilde Y_{-i}}(a, p)$, for $n_0+1\leq i < n.\end{cases}$
\item Moreover if  $a\in [0, m_0-1]$ then  
\begin{enumerate}
\item ${\tilde \nu}_{-n_0}(a) = {\tilde Z_{-n_0}}(a, p)$
\item ${\tilde \mu}_{-n_0+1}(a) = 0$ and 
\item ${\tilde \mu}_{-n_0}(a) = \frac{1}{2\lambda_0}\left[{\tilde Z_{-n_0-1}}(a, p)-
{\tilde Y_{-n_0-1}}(a, p)\right] \neq 0$.
\end{enumerate}
\item For $a\in [m_0, p)$
\begin{enumerate}
\item ${\tilde \nu}_{-n_0}(a) =  {\tilde Z_{-n_0}}(a, p)-
{\tilde Y_{-n_0-1}}(a, p)+ {\tilde Z_{-n_0-1}}(a, p)$,
\item ${\tilde \mu}_{-n_0+1}(a) =  \frac{1}{2\lambda_0}
\left[{\tilde Y_{-n_0-1}}(a, p)-{\tilde Z_{-n_0-1}}(a, p)\right] \neq 0$  
\item ${\tilde \mu}_{-n_0}(a) = 0$. 
\end{enumerate}
\end{enumerate}
\end{lemma}

\begin{proof}For $\Delta$ as in Notations~\ref{n1},  
$\Delta = \frac{n-2}{2p}+\frac{1}{2}$.
Now $F_*\sO(a)$ contains $S(-n_0)$ if and only if 
$$0\leq a/p +\Delta \leq (p-2)/p \iff a\in [0, m_0-2].$$
 The rest of the first set of assertions follows from Lemma~\ref{l3}.

Similarly $F_*\sS(a)$ contains $S(t)$ if and only if 
$$0\leq a/p +\Delta-t-n_0\leq (p-1)/p.$$
This implies $-1\leq -t-n_0 \leq 0$ and hence if $S(t)$ occurs in 
$F_*S(a)$  then $t= -n_0$ or $t=-n_0+1$.

Moreover ${\tilde \mu}_{-n_0}(a) \neq 0 \implies a\in [0, m_0-1]$
and ${\tilde \mu}_{-n_0+1}(a) \neq 0 \implies a\in [m_0, p-1]$

In Lemma~\ref{l2}, we have used repeatedly the equalities 
$$\begin{array}{lllll}
h^0(Q, F_*^s\sO(m)) & =  &  h^0(Q, \sO(m)) & = & L_m\\\
h^n(Q, F_*^s\sO(m)) & = & h^n(Q, \sO(m)) & = & h^0(Q, \sO(-m-n)).
\end{array}$$
Similarly  equalities hold for spinor bundles:
$$\begin{array}{lllll}
h^0(Q, F_*^s(\sS(m))) & = &   h^0(Q, \sS(m)) & = & {\tilde L}_m\\\
h^n(Q, F_*^s\sS(m)) & = & h^n(Q, \sS(m)) & = & h^0(Q, S(1-m-n)) = {\tilde L}_{1-m-n}.
\end{array}$$

Now similar  equalities as in Lemma~\ref{l2} and Lemma~\ref{l3} holds if we replace ${\nu}_i(a)$ by 
${\tilde \nu}_i(a)$ and ${\mu}_i(a)$ by 
${\tilde \mu}_i(a)$ in the statement.  Hence the second set of assertions follow.
\end{proof}

\begin{lemma}\label{e*c1}
Let  $p > n-2 $, where $n\geq 4$ is an even integer and $1-n\leq a < p$ is
 an integer. Then
$$F_*(\sO(a)) = \oplus_{t=0}^{n-1}\sO(-t)^{\nu_{-t}(a)}\oplus
\oplus 
\sS(-n_0+1)^{\mu_{-n_0+1}(a)}\oplus 
\sS(-n_0)^{\mu_{-n_0}(a)}\oplus 
\sS(-n_0-1)^{\mu_{-n_0-1}(a)}.$$ 

Let ${\tilde m}_0 = p-\frac{n-2}{2} = p+1-\frac{n}{2}$. 
Then 
\begin{enumerate}
\item[(A1)] $\nu_{-i}(a) = \begin{cases} 
Z_{-i}(a, p),\quad\mbox{for}\quad 0\leq i \leq n_0-1\\\ 
 Y_{-i}(a, p),\quad\mbox{for}\quad n_0+1\leq i < n \end{cases} $ 

\item[(A2)] Moreover  if  $a\in [1-n, -\tfrac{n}{2}]$ then 
\begin{enumerate}
\item $\nu_{-n_0}(a) = Z_{-n_0}(a, p)$
\item $\mu_{-n_0+1}(a) = 0$ 
\item $\mu_{-n_0}(a) = 0$ and
\item $\mu_{-n_0-1}(a) = \frac{1}{2\lambda_0}\left[Y_{-n_0-1}(a, p)-
Z_{-n_0-1}(a, p)\right]$

\end{enumerate}

\item[(A3)] if  $a\in [-\frac{n}{2}, {\tilde m}_0-1]$ then 
\begin{enumerate}
\item $\nu_{-n_0}(a) = Z_{-n_0}(a, p)$,

\item $\mu_{-n_0+1}(a) = 0$,
\item $\mu_{-n_0}(a) = \frac{1}{2\lambda_0}\left[Z_{-n_0-1}(a, p)-
Y_{-n_0-1}(a, p)\right]$,
\item $\mu_{-n_0-1}(a) = 0$.
\end{enumerate}

\item[(A4)] Similarly,  for $a\in [{\tilde m}_0, p)$ we have 
\begin{enumerate}
\item $\nu_{-n_0}(a) =  Z_{-n_0}(a, p)-Y_{-n_0-1}(a, p)+ Z_{-n_0-1}(a, p)$,
\item $\mu_{-n_0+1}(a) = \frac{1}{2\lambda_0}\left[Y_{-n_0-1}(a, p)-
Z_{-n_0-1}(a, p)\right] \neq 0$.
\item $\mu_{-n_0}(a) = 0$.
\item $\mu_{-n_0-1}(a) = 0$,
\end{enumerate}
\end{enumerate}

Also we have the decomposition
$$F_*(\sO(a)) = \oplus_{t=0}^{n-1}\sO(-t)^{{\tilde \nu}_{-t}(a)}\oplus
\oplus 
\sS(-n_0+1)^{{\tilde \mu}_{-n_0+1}(a)}\oplus 
\sS(-n_0)^{{\tilde \mu}_{-n_0}(a)}\oplus 
\sS(-n_0-1)^{{\tilde \mu}_{-n_0-1}(a)}.$$

\begin{enumerate}
\item[(B1)] ${\tilde \nu}_{-i}(a) = \begin{cases}
{\tilde Z_{-i}}(a, p)$, for $0\leq i \leq n_0-1\\\
{\tilde Y_{-i}}(a, p)$, for $n_0+1\leq i < n.\end{cases}$

 \item[(B2)]If  $a\in [-\tfrac{n}{2}, {\tilde m}_0-1]$ then  
\begin{enumerate}
\item ${\tilde \nu}_{-n_0}(a) = {\tilde Z_{-n_0}}(a, p)$
\item ${\tilde \mu}_{-n_0+1}(a) = 0$ and 
\item ${\tilde \mu}_{-n_0}(a) = \frac{1}{2\lambda_0}\left[{\tilde Z_{-n_0-1}}(a, p)-
{\tilde Y_{-n_0-1}}(a, p)\right] \neq 0$,
\item ${\tilde \mu}_{-n_0-1}(a) = 0$
\end{enumerate}
\item[(B3)] For $a\in [{\tilde m}_0, p)$
\begin{enumerate}
\item ${\tilde \nu}_{-n_0}(a) =  {\tilde Z_{-n_0}}(a, p)-
{\tilde Y_{-n_0-1}}(a, p)+ {\tilde Z_{-n_0-1}}(a, p)$,
\item ${\tilde \mu}_{-n_0+1}(a) =  \frac{1}{2\lambda_0}
\left[{\tilde Y_{-n_0-1}}(a, p)-{\tilde Z_{-n_0-1}}(a, p)\right] \neq 0$  
\item ${\tilde \mu}_{-n_0}(a) = 0$,
\item ${\tilde \mu}_{-n_0-1}(a) = 0$. 
\end{enumerate}
\end{enumerate}
\end{lemma}
\begin{proof}Note that $\Delta = \tfrac{n-2}{2p}$. If $1-n\leq a <p$ then 
$-1< -\tfrac{n}{2p} \leq \tfrac{a}{p} + \Delta < 2$.
Now if 
$F_*(\sO(a))$ contains $\sO(t)$ then  
$$0\leq a -tp \leq n(p-1)\quad \mbox{implies}\quad t\in \{0, -1, \ldots, -n+1\}.$$

If 
$F_*(\sO(a))$ contains $\sS(t)$ then  
$$(n_0-\Delta)p\leq a-tp \leq (n_0-\Delta)p + p-2,$$
which implies 
$0\leq \tfrac{a}{p}+\Delta -t-n_0\leq 1-\tfrac{2}{p}$.
Hence $-2 < -t-n_0 < 2$, which gives $t\in \{-n_0+1, -n_0, -n_0-1\}$.

If 
$F_*(\sS(a))$ contains $\sO(t)$ then  
$$1\leq a-tp  \leq n(p-1)\quad \mbox{implies}\quad t\in \{0, -1, \ldots, -n+1\}.$$

Moreover
\begin{enumerate}
\item  ${\tilde \mu}_{-n_0-1}(a) \neq 0 \implies a\in [1-n,~~ -\tfrac{n}{2}]$,
\item  ${\tilde \mu}_{-n_0}(a) \neq 0 \implies a\in [-\tfrac{n}{2}+1,~~ {\tilde m}_0-1]$
and 
\item ${\tilde \mu}_{-n_0+1}(a) \neq 0 \implies a\in [{\tilde m}_0,~~ p-1]$.
\end{enumerate}

If $F_*(\sS(a))$ contains $\sS(t)$ then  
$$(n_0-\Delta)p\leq a-tp \leq (n_0-\Delta)p + p-1,$$
which implies 
$0\leq \tfrac{a}{p}+\Delta -t-n_0\leq 1-\tfrac{1}{p}$, 
which gives $t\in \{-n_0+1, -n_0, -n_0-1\}$.
Since  the above decomposition of $F_*(\sO(a))$ and 
$F_*(\sS(a))$ hold for all 
$-n+1 \leq a <p$, 
all the equalities stated  in Lemma~\ref{l2} hold true 
for every $-n+1\leq a <p$, where 
$$L_a = h^0(Q_n, \sO(a)) = (2a+n)\frac{(a+n-1)\cdots (a+1)}{n!},\quad\mbox{for}\quad
n+1\leq a.$$
In particular, by Lemma~\ref{l2}~(2) and (5) we have 
$$Y_{-n_0-1}(a,p)-2\lambda_0\mu_{-n_0-1}+2\lambda_0\mu_{-n_0} = Z_{-n_0-1}(a, p)+
2\lambda_0\mu_{-n_0+1}.$$
One can whack that 
if $a\leq -\tfrac{n}{2}$ then $\mu_{-n_0}(a) = 0$ and for $a\geq -
\tfrac{n}{2}$ we have $\mu_{-n_0-1}(a) = 0$. This observation and the rest of 
the argument made for $n$ odd case, will give all the assertion for 
$F_*(\sO(a))$.

The set of assertion for $F_*(\sS(a))$ is along the same lines so we leave the 
details to the reader.
\end{proof}

\begin{notations}\label{Bmat}Let $n_i\in \{0, \ldots, n-3\}$. We define a set of 
matrices ${\B}_{n_i}^{(p)}$ and ${\C}_{n_i}^{(p)}$ as follows.

Let  $$F_*(\sO(m)) = (\oplus_{t=0}^{n-1}\sO(-t)^{{\nu}_{-t}(m)})
\oplus\sS(-n_0+1)^{{\mu}_{-n_0+1}(m)}\oplus \sS(-n_0)^{{\mu}_{-n_0}(m)}
\oplus\sS(-n_0-1)^{{\mu}_{-n_0-1}(m)}$$  

$$F_*(\sS(m)) = (\oplus_{t=0}^{n-1}\sO(-t)^{{\tilde \nu}_{-t}(m)}
\oplus\sS(-n_0+1)^{{\tilde \mu}_{-n_0+1}(m)}\oplus 
\sS(-n_0)^{{\tilde \mu}_{-n_0}(m)}\oplus\sS(-n_0-1)^{{\tilde \mu}_{-n_0-1}(m)}.$$

\begin{enumerate}
\item  $n\geq 3$ is an odd integer,  we define  
 $(n+2)\times (n+2)$ matrix  
$${\B}^{(p)}_{n_i} = \left[\begin{matrix}
\nu_{0}(j_i), \ldots, \nu_{-n+1}(j_i), \mu_{-n_0+1}(j_i), \mu_{-n_0}(j_i)\\
\vdots\\
\nu_{0}(j_i-n+1), \ldots,   \nu_{-n+1}(j_i-n+1), \mu_{-n_0+1}(j_i-n+1),
\mu_{-n_0}(j_i-n+1)\\
{\tilde \nu}_{0}(j_i-n_0+1),  \ldots, {\tilde \nu}_{-n+1}(j_i-n_0+1), 
{\tilde \mu}_{-n_0+1}(j_i-n_0+1),
{\tilde\mu}_{-n_0}(j_i-n_0+1)\\
{\tilde \nu}_{0}(j_i-n_0),  \ldots, {\tilde \nu}_{-n+1}(j_i-n_0), 
{\tilde \mu}_{-n_0+1}(j_i-n_0), {\tilde \mu}_{-n_0}(j_i-n_0)
\end{matrix}\right].$$

Similarly

\item If $n\geq 4$ is an even integer then we define  
 $(n+3)\times (n+3)$ matrix  
$${\C}^{(p)}_{n_i} = \left[\begin{matrix}
\nu_{0}(j_i), \ldots , \mu_{-n_0+1}(j_i), \mu_{-n_0}(j_i),
\mu_{-n_0-1}(j_i)\\
\vdots\\
\nu_{0}(j_i-n+1), \ldots , \mu_{-n_0+1}(j_i-n+1),
\mu_{-n_0}(j_i-n+1),\mu_{-n_0-1}(j_i-n+1)\\
{\tilde \nu}_{0}(j_i-n_0+1),  \ldots,  
{\tilde \mu}_{-n_0+1}(j_i-n_0+1),
{\tilde\mu}_{-n_0}(j_i-n_0+1), {\tilde \mu}_{-n_0-1}(j_i-n_0+1)\\
{\tilde \nu}_{0}(j_i-n_0),  \ldots, 
{\tilde \mu}_{-n_0+1}(j_i-n_0), {\tilde \mu}_{-n_0}(j_i-n_0),{\tilde \mu}_{-n_0-1}(j_i-n_0)\\
{\tilde \nu}_{0}(j_i-n_0-1),  \ldots,  
{\tilde \mu}_{-n_0+1}(j_i-n_0-1),
{\tilde\mu}_{-n_0}(j_i-n_0-1), {\tilde \mu}_{-n_0-1}(j_i-n_0-1)
\end{matrix}\right].$$
\end{enumerate}

We denote
$${\B}_{n_i}^{(p)} = \left[b_{k_1, k_2}^{(n_i)}(p)\right]_{0\leq k_1, k_2\leq n+1}\quad\mbox{and}\quad
{\C}^{(p)}_{n_i} = \left[ c_{k_1, k_2}^{(n_i)}(p)\right]_{0\leq k_1, k_2\leq n+3}.$$ 
\end{notations}

\begin{lemma}\label{entries} For a given odd integer $n\geq 3$,
There exists a set of polynomials 
$$\{H_{k_1, k_2}^{(n_i)}(t)\in \Q[t]\mid 0\leq n_i\leq n-3\quad\mbox{and}\quad 
0\leq k_1, k_2 \leq n+1\}$$
 such that 
 $$\frac{b^{(n_i)}_{k_1, k_2}(p)}{p^n} =  
H_{k_1, k_2}^{(n_i)}\bigl(\frac{1}{p}\bigr) \geq 0,\quad\mbox{for all}
\quad p \geq  3n-4.$$
\end{lemma}
\begin{proof}
Let $0\leq k_1<n$ 
then, by Lemma~\ref{*c1} 
$$\begin{array}{lcl}
b^{(n_i)}_{k_1, k_2}(p) & = & \begin{cases}
 Z_{-k_2}(j_i-k_1, p) & \mbox{if}\quad 0 \leq k_2 \leq n_0-1\\\  
Y_{-k_2}(j_i-k_1, p) & \mbox{if}\quad n_0+1\leq k_2 < n\end{cases}\\\

b^{(n_i)}_{k_1, n_0}(p) & = &\begin{cases}
   Z_{-n_0}(j_i-k_1, p)  &  \mbox{if}\quad n_i < k_1 \\\
 Z_{-n_0}(j_i-k_1, p)-Y_{-n_0-1}(j_i-k_1,p)+Z_{-n_0-1}(j_i-k_1,p) 
&  \mbox{if}\quad n_i \geq k_1 \end{cases}\\\

b^{(n_i)}_{k_1, n}(p) & = &\begin{cases}
  0  &  \mbox{if}\quad n_i < k_1 \\\
\frac{1}{2\lambda_0}\left[Y_{-n_0-1}(j_1-k_1,p)-Z_{-n_0-1}(j_1-k_1,p)\right] 
&  \mbox{if}\quad n_i \geq k_1 \end{cases}\\\

b^{(n_i)}_{k_1, n+1}(p) & = &\begin{cases}
\frac{1}{2\lambda_0}\left[Z_{-n_0-1}(j_1-k_1,p)-Y_{-n_0-1}(j_1-k_1,p)\right] 
&  \mbox{if}\quad n_i < k_1\\\
  0  &  \mbox{if}\quad n_i \geq  k_1  \end{cases}
\end{array}$$

$$\begin{array}{lcl}
b^{(n_i)}_{n, k_2}(p) & = & \begin{cases}
 {\tilde Z_{-k_2}}(j_i-n_0+1, p) & \mbox{if}\quad 0 \leq k_2 \leq  n_0-1\\\ 
 {\tilde Y_{-k_2}}(j_i-n_0+1, p) & \mbox{if}\quad n_0+1\leq k_2 < n\end{cases}\\\

b^{(n_i)}_{n, n_0}(p) & = &\begin{cases}
   {\tilde Z_{-n_0}}(j_i-n_0+1, p)  &  \mbox{if}\quad n_i < n_0-1 \\\
 {\tilde Z_{-n_0}}(j_i-n_0+1, p)-{\tilde Y_{-n_0-1}}(j_i-n_0+1,p)
&  \mbox{if}\quad n_i \geq n_0-1\\\
+ {\tilde Z_{-n_0-1}}(j_i-n_0+1,p) &{} 
 \end{cases}\\\

b^{(n_i)}_{n+1, k_2}(p) & = & \begin{cases}
 {\tilde Z_{-k_2}}(j_i-n_0, p) & \mbox{if}\quad 0 \leq k_2 \leq  n_0-1\\\ 
 {\tilde Y_{-k_2}}(j_i-n_0, p) & \mbox{if}\quad n_0+1\leq k_2 < n\end{cases}\\\

b^{(n_i)}_{n+1, n_0}(p) & = &\begin{cases}
   {\tilde Z_{-n_0}}(j_i-n_0, p)  &  \mbox{if}\quad n_i < n_0 \\\
 {\tilde Z_{-n_0}}(j_i-n_0, p)-{\tilde Y_{-n_0-1}}(j_i-n_0,p)+{\tilde Z_{-n_0-1}}
(j_i-n_0,p) 
&  \mbox{if}\quad n_i \geq n_0 \end{cases}\\\

\end{array}$$

$$\begin{array}{lcl}
b^{(n_i)}_{n, n}(p) & = &\begin{cases}
  0  &  \mbox{if}\quad n_i < n_0-1 \\\
\frac{1}{2\lambda_0}\left[{\tilde Y_{-n_0-1}}(j_1-n_0+1,p)-
{\tilde Z_{-n_0-1}}(j_1-n_0+1,p)\right] 
&  \mbox{if}\quad n_i \geq n_0-1 \end{cases}\\\

b^{(n_i)}_{n, n+1}(p) & = &\begin{cases}
\frac{1}{2\lambda_0}\left[{\tilde Z_{-n_0-1}}(j_1-n_0+1,p)-
Y_{-n_0-1}(j_1-n_0+1,p)\right] 
&  \mbox{if}\quad n_i < n_0-1\\\
  0  &  \mbox{if}\quad n_i \geq n_0-1  \end{cases}\\\

b^{(n_i)}_{n+1, n}(p) & = &\begin{cases}
  0  &  \mbox{if}\quad n_i < n_0 \\\
\frac{1}{2\lambda_0}\left[{\tilde Y_{-n_0-1}}(j_1-n_0,p)-
{\tilde Z_{-n_0-1}}(j_1-n_0,p)\right]
&  \mbox{if}\quad n_i \geq n_0 \geq 0\end{cases}\\\

b^{(n_i)}_{n+1, n+1}(p) & = &\begin{cases}
\frac{1}{2\lambda_0}\left[{\tilde Z_{-n_0-1}}(j_1-n_0,p)-Y_{-n_0-1}(j_1-n_0,p)\right]
&  \mbox{if}\quad n_i < n_0\\\
  0  &  \mbox{if}\quad n_i \geq n_0  \end{cases}
\end{array}$$

By Remark~\ref{r3s}, for given integer $i$, 
there exists a polynomial $P_i(X,Y) \in \Q[X, Y]$ of degree $\leq n$ such that 
$Z_{-i}(a,p) = P_i(a, p)$, for all $0\leq a<p$  and for $p>n-2$.
For $Y_{-i}$, ${\tilde Z_{-i}}$ and ${\tilde Y_{-i}}$ we have similar polynomial 
expressions.

In particular, given $k_1$, $k_2$ and $n_i$, 
there exist  polynomials of total degree $\leq n$ 
$P_{k_1, k_2}^{(n_i)}(X, Y) = \sum_{m_1, m_2}
\lambda_{m_1, m_2}X^{m_1}Y^{m_2}$, which are independent of $p$ and 
for $0\leq k_1 < n$

\begin{equation}
b^{(n_i)}_{k_1, k_2}(p)
  =   \sum_{m_1, m_2}\lambda_{m_1, m_2}^{(k_1, k_2, n_i)}\left(\frac{p}{2}+
{n_i-{k_1}+1-n/2}\right)^{m_1}
\left(p^{m_2}\right),\end{equation}

\begin{equation}
b^{(n_i)}_{n, k_2}(p)
  =   \sum_{m_1, m_2}\lambda_{m_1, m_2}^{(n, k_2, n_i)}\left(\frac{p}{2}+
n_i-n_0+2 -\frac{n}{2}\right)^{m_1}
\left(p^{m_2}\right),\end{equation}

\begin{equation}
b^{(n_i)}_{n+1, k_2}(p)
  =   \sum_{m_1, m_2}\lambda_{m_1, m_2}^{(n+1, k_2, n_i)}\left(\frac{p}{2}+
n_i-n_0+1-\frac{n}{2}\right)^{m_1}
\left(p^{m_2}\right),\end{equation}

Now choosing 
$$H_{k_1, k_2}^{(n_i)}(t) =  
\sum_{m_1, m_2}\lambda_{m_1, m_2}^{({k_1}, k_2, n_i)}
\left(\frac{1}{2}+
n_it-{\tilde k_1}t+t-\frac{nt}{2}\right)^{m_1}
t^{(n-m_1-m_2)},$$
where  ${\tilde k_1} = n_0-1$ if $k_1 =n$ and ${\tilde k_1} = n_0$ if $k_1 =n+1$ and 
${\tilde k_1} = k_1$ for $0\leq k_1 <n$,
will prove the lemma.
\end{proof}

\begin{notations}\label{eBmat} Let $n\geq 4$ be an even integer and $p \geq (3n-4)/2$. 
Let ${\tilde \sS}_e = \{0, 1, \ldots, \tfrac{n-2}{2}-1\}\bigcup 
\{p-\tfrac{n-2}{2}, \ldots, p-1\}$,
which is indexed by the set $\sM = \{0, 1, \ldots, n-3\}$ as in Definition~\ref{deven}.
Here $n_0 = \tfrac{n}{2}-1$. 
For given $n_i\in \sM$ we define a $(n+3)\times (n+3)$ matrix 
${\C}_{n_i}^{(p)} = [c_{k_1, k_2}^{(n_i)}(p)]_{0\leq k_1, k_2\leq n+2}$ as follows:

$$\begin{array}{lcl}
 c_{k_1, k_2}^{(n_i)}(p) & = & \nu_{-k_2}(j_i-k_1)\quad\mbox{if}\quad
0\leq k_1, k_2 \leq n-1\\\
& = & \mu_{-n_0-\delta_2}(j_i-k_1)\quad\mbox{if}\quad
0\leq k_1 \leq n-1,~~ k_2 = n+1+\delta_2,~~-1\leq \delta_2\leq 1\\\

& = & {\tilde \nu}_{-k_2}(j_i-n_0-\delta_1)\quad\mbox{if}\quad
0\leq k_2 \leq n-1~~k_1 = n+1+\delta_1,~~~-1\leq \delta_1 \leq 1\\\

& = & {\tilde \mu}_{-n_0-\delta_2}(j_i-n_0-\delta_1)\quad\mbox{if}\quad
k_i = n+1+\delta_i,\quad -1\leq \delta_1,\delta_2  \leq 1.
\end{array}$$
\end{notations}

\vspace{5pt}

\begin{lemma}\label{eentries}
There exists a set of polynomials 
$$\{{\tilde H}_{k_1, k_2}^{(n_i)}(t)\in \Q[t]\mid 0\leq n_i\leq n-3\quad\mbox{and}\quad 
0\leq k_1, k_2 \leq n+2\}$$
 such that 
 $$\frac{c^{(n_i)}_{k_1, k_2}(p)}{p^n} =  
{\tilde H}_{k_1, k_2}^{(n_i)}\bigl(\frac{1}{p}\bigr) \geq 0,\quad\mbox{for all}
\quad p \geq  ({3n-4})/{2}.$$
\end{lemma}
\begin{proof}
\noindent{Case}~(1).\quad $n_i \in \{0, \ldots, \tfrac{n}{2}-2\}$. Note that in this case $j_i = n_i$.
\begin{enumerate}
\item Let  $0\leq k_1 \leq n-1$. Then  
$1-n\leq j_i-k_1 \leq \tfrac{n}{2}-2$.

If $k_1\in \{n_i+\tfrac{n}{2}, \ldots, n_i+n-1\}$ then
$1-n\leq j_i-k_1\leq -\tfrac{n}{2}$.
Hence applying Lemma~\ref{e*c1}~(A2), we get polynomial expression 
for $c_{k_1,k_2}^{(n_i)}(p)$,  for all $0\leq k_2\leq n+2$.

If $k_1\in \{0, \ldots, n_i+\tfrac{n}{2}-1\}$ then 
$ -\tfrac{n}{2} -1\leq j_i-k_1\leq \tfrac{n}{2}-2$.
Hence 
applying 
Lemma~\ref{e*c1}~(A3), we get polynomial expression for $c_{k_1,k_2}^{(n_i)}(p)$, 
for all $0\leq k_2\leq n+3$.

\item Let $k_1 = n+1+\delta_1$, where $-1\leq \delta_1\leq 1$. Then 
$-\tfrac{n}{2}\leq n_i-n_0-\delta_1 \leq 0$. Hence applying 
Lemma~\ref{e*c1}~(B2) we get polynomial expression for 
$c_{n+1+\delta_1,k_2}^{(n_i)}(p)$.
\end{enumerate}

\noindent{Case}~(2).\quad Let $n_i\in \{\tfrac{n}{2}-1, \ldots, n-3\}$.
Then $j_i\in \{p-\tfrac{n-2}{2}, \ldots, p-1\}$.
\begin{enumerate}
\item Let $0\leq k_1 \leq n-1$. Then  
$p-\tfrac{3n}{2}+2  \leq j_i-k_1 = p-1$, which implies 
$0\leq j_i-k_1 \leq p-1$ as $p\geq (3n-4)/2$.

If $k_1\in \{n-2-n_i, \ldots, n-1\}$ then
$0\leq j_i-k_1 \leq {\tilde m_0}-1$. Therefore 
applying Lemma~\ref{e*c1}~(A3), we get polynomial expression 
for $c_{k_1,k_2}^{(n_i)}(p)$, for all $0\leq k_2\leq n+2$.

If $k_1\in \{0, \ldots, n-3-n_i\}$  then ${\tilde m_0}\leq j_i-k_1 \leq p-1$ and hence
applying 
Lemma~\ref{e*c1}~(A4), we get polynomial expression for $c_{k_1,k_2}^{(n_i)}(p)$.
for all $0\leq k_2\leq n+2$.

\item Let $k_1 = n+1+\delta_1$, where $-1\leq \delta_1\leq 1$.

Then $p-n+1 \leq j_i-n_0-\delta_1 \leq {\tilde m_0}$.

(a)\quad If $n_i = \tfrac{n}{2}-1$ and $k_1 = n$ then 
$j_i-n_0-\delta_1 = {\tilde m_0}$.
In this case, by Lemma~\ref{e*c1}~(B3), 
we get polynomial expression for $c_{n,k_2}^{(n_i)}(p)$.

(b)\quad If either $n_i\in \{ \tfrac{n}{2}, \ldots, n-3\}$ or $k_1 \in \{n+1, n+2\}$
then 
$p-n+1 \leq j_i-n_0-\delta_1 \leq {\tilde m_0}-1$
and by Lemma~\ref{e*c1}~(B2), 
we get polynomial expression for $c_{k_1,k_2}^{(n_i)}(p)$, for all $0\leq k_2<n+2$.
\end{enumerate}\end{proof}

\end{document}